\documentclass[11pt,leqno]{amsart}
\usepackage{amsmath,amsfonts,latexsym,graphicx,amssymb,url, color}
\usepackage{hyperref}
\usepackage{txfonts} 
\newcommand{\average}{{\mathchoice {\kern1ex\vcenter{\hrule
height.4pt width 6pt depth0pt} \kern-9.7pt}
{\kern1ex\vcenter{\hrule height.4pt width 4.3pt depth0pt}
\kern-7pt} {} {} }}

\begin{document}

\newcommand{\dist}{\text{dist}} 
\newcommand{\abs}[1]{\left\vert#1\right\vert}
\newcommand{\diam}{\text{diam}}
\newcommand{\trace}{\text{trace}}
\newcommand{\R}{{\mathbb R}} 
\newcommand{\C}{{\mathbb C}}
\newcommand{\Z}{{\mathbb Z}}
\newcommand{\N}{{\mathbb N}}
\newcommand{\gradg}{\nabla_{\G}}
\newcommand{\e}{\epsilon} 
\newcommand{\bH}{\mathrm{\bf H}}
\newcommand{\calF}{{\mathcal F}}
\newcommand{\calO}{{\mathcal O}}
\newcommand{\calS}{{\mathcal S}}
\newcommand{\calM}{{\mathcal M}}
\newcommand{\calL}{\mathcal L}
\newcommand{\calW}{{\mathrm W}}
\newcommand{\bh}{{\bar h}}


\newcommand{\Rn}{\mathbb R^n}
\newcommand{\Rm}{\mathbb R^m}
\newcommand{\lA}{A^{\mbox{loc}}}
\renewcommand{\L}[1]{\mathcal L^{#1}}
\newcommand{\G}{\mathbb G}
\newcommand{\U}{\mathcal U}
\newcommand{\M}{\mathcal M}
\newcommand{\eps}{\epsilon}
\newcommand{\BVG}{BV_{\G}(\Omega)}
\newcommand{\no}{\noindent}
\newcommand{\ro}{\varrho}
\newcommand{\p}{\partial}
\newcommand{\rn}[1]{{\mathbb R}^{#1}}
\newcommand{\res}{\mathop{\hbox{\vrule height 7pt width .5pt depth 0pt
\vrule height .5pt width 6pt depth 0pt}}\nolimits}
\newcommand{\hhd}[1]{{\mathcal H}_d^{#1}} \newcommand{\hsd}[1]{{\mathcal
S}_d^{#1}} \renewcommand{\H}{\mathbb H}
\newcommand{\BVGL}{BV_{\G,{\rm loc}}}
\newcommand{\GH}{H\G}
\renewcommand{\diam}{\mbox{diam}\,}
\renewcommand{\div}{\mbox{div}\,}
\newcommand{\divg}{\mathrm{div}_{\G}\,}
\newcommand{\norm}[1]{\|{#1}\|_{\infty}} \newcommand{\modul}[1]{|{#1}|}
\newcommand{\per}[2]{|\partial {#1}|_{\G}({#2})}
\newcommand{\Per}[1]{|\partial {#1}|_{\G}} \newcommand{\scal}[3]{\langle
{#1} , {#2}\rangle_{#3}} \newcommand{\Scal}[2]{\langle {#1} ,
{#2}\rangle}
\newcommand{\fron}{\partial^{*}_{\G}}
\newcommand{\hs}[2]{S^+_{\H}({#1},{#2})} \newcommand{\test}{\mathbf
C^1_0(\G,\GH)} \newcommand{\Test}[1]{\mathbf C^1_0({#1},\GH)}
\newcommand{\card}{\mbox{card}}
\newcommand{\bom}{\bar{\gamma}}
\newcommand{\shpiu}{S^+_{\G}}
\newcommand{\shmeno}{S^-_{\G}}
\newcommand{\CG}{\mathbf C^1_{\G}}
\newcommand{\CH}{\mathbf C^1_{\H}}
\newcommand{\di}{\mathrm{div}_{X}\,}
\newcommand{\norma}{\vert\!\vert}
\newenvironment{myindentpar}[1]%
{\begin{list}{}%
         {\setlength{\leftmargin}{#1}}%
         \item[]%
}
{\end{list}}
%

\renewcommand{\baselinestretch}{1.2}
\theoremstyle{plain}
\newtheorem{theorem}{Theorem}[section]
\newtheorem{corollary}[theorem]{Corollary}
\newtheorem{lemma}[theorem]{Lemma}
\newtheorem{proposition}[theorem]{Proposition}
\newtheorem{definition}[theorem]{Definition}
\newtheorem{remark}[theorem]{Remark}
\providecommand{\bysame}{\makebox[3em]{\hrulefill}\thinspace}
\renewcommand{\theequation}{\thesection.\arabic{equation}}

\date{}
\title[Global $W^{2,p}$ estimates]{
Global $W^{2,p}$ estimates for solutions to the linearized Monge--Amp\`ere equations
}
\author{Nam Q. Le}
\author{Truyen Nguyen}
\thanks{Nam Q. Le\\
Department of Mathematics, Columbia University, New York, NY 10027, USA\\
email:namle@math.columbia.edu}
\thanks{
Truyen Nguyen\\
Department of Mathematics, The University of Akron, Akron, OH 44325, USA\\
email:tnguyen@uakron.edu}

\begin{abstract}
 In this paper, we establish global $W^{2,p}$ estimates for solutions to the linearized Monge-Amp\`ere equations under 
natural assumptions on the domain, Monge-Amp\`ere measures and boundary data. Our estimates are affine invariant analogues
of the global $W^{2,p}$ estimates of Winter for fully nonlinear, uniformly elliptic equations, and also linearized counterparts of Savin's global $W^{2,p}$ estimates for the
Monge-Amp\`ere equations.  

\end{abstract}

\maketitle

\tableofcontents

\setcounter{equation}{0}

\section{\bf Introduction and Statement of the Main Results}

In this paper we consider  the linearized Monge-Amp\`ere equations and investigate 
 global $L^p$ estimates for the second derivatives of their solutions. 
Let $\Omega\subset \R^n$ be a bounded convex domain and  $\phi$ be a locally uniformly convex function on  $\Omega$. The linearized  Monge-Amp\`ere equation corresponding to $\phi$ is 
\begin{equation}\label{LMA-eq}
\mathcal{L}_{\phi} u:= \sum_{i, j=1}^{n} \Phi^{ij} u_{ij}= f\quad \mbox{in}\quad \Omega,
\end{equation}
where $\Phi=\big(\Phi^{ij}\big)_{1\leq i, j\leq n} := (\det D^2 \phi )~ (D^2\phi)^{-1}
$ is the matrix of cofactors of the Hessian matrix $D^2\phi$.
As the coefficient matrix  $\Phi$ is positive semi-definite, $\mathcal{L}_{\phi}$ is a linear elliptic partial differential operator, possibly degenerate. The operator $\mathcal{L}_{\phi}$ appears in
several contexts including affine maximal surface equation in affine geometry \cite{TW, TW1, TW2, TW3}, Abreu's equation in the context of existence 
of K\"ahler metric of constant scalar curvatures
in complex geometry \cite{Don1, Don2, Don3, Don4, Zh}, and semigeostrophic equations in fluid mechanics \cite{B, CNP, Loe}. Solutions of many important problems in these contexts 
require a deep understanding of interior and boundary behaviors of solutions to  \eqref{LMA-eq}.

The regularity theory for the linearized Monge-Amp\`ere equation was initiated in the fundamental paper \cite{CG2} by Caffarelli and Guti\'errez. They 
 established an interior Harnack 
inequality for nonnegative solutions to the homogeneous
equation $\mathcal{L}_{\phi}u=0$ in terms of the pinching of the Hessian determinant $\lambda\leq \det D^{2} \phi\leq \Lambda$.
Their theory is an affine invariant version of the classical Harnack inequality for uniformly elliptic equations with measurable coefficients.
This result played a crucial role in 
Trudinger-Wang's resolution \cite{TW1} of Chern's conjecture in affine geometry concerning affine maximal hypersurfaces in $\R^3$ and in Donaldson's interior estimates for Abreu's equation 
in complex geometry \cite{Don2}.
 Another contribution to the regularity theory comes
from \cite{GT} where Guti\'errez and Tournier derived interior $W^{2,\delta}$ estimates for small $\delta$. The
interior regularity for equation \eqref{LMA-eq}  was further developed by
Guti\'errez and the second author in \cite{GN1, GN2} where the (sharp) interior $C^{1,\alpha}$ and $W^{2,p}$ estimates, respectively, were obtained.

Regarding the global regularity, by using Caffarelli-Guti\'errez's interior 
Harnack estimates and Savin's localization theorem,
Savin and the first author \cite{LS} established boundary H\"older gradient estimates  for solutions to the linearized Monge-Amp\`ere equation. 
Furthermore, the first author \cite{L} proved 
global H\"older estimates for solutions to \eqref{LMA-eq} in uniformly convex domains, which are
the global counterpart of Caffarelli-Guti\'errez's interior H\"older estimates  \cite{CG2}.

As mentioned above, Guti\'errez and the second author derived in \cite{GN2} the interior $W^{2,p}$ estimates for solutions of  \eqref{LMA-eq} in terms of the $L^{q}$-norm of $f$ where $q>\max\{n, p\}$, the 
pinching of the Hessian determinant
$\lambda\leq \det D^{2} \phi\leq \Lambda$ and the continuity of the Monge-Amp\`ere measure $\det D^2 \phi$. 
The purpose of our paper is to establish global $W^{2,p}$ estimates for solutions to the linearized Monge-Amp\`ere equation
\eqref{LMA-eq}  under natural assumptions on the domain, Monge-Amp\`ere measures and boundary data.\\

 Our first main theorem is concerned with global $W^{2,p}$ estimates for the linearized equation \eqref{LMA-eq} when the Monge-Amp\`ere measure
 $\det D^2\phi$ is close to a constant. 
 \begin{theorem}\label{main-result-constant}
Let $\Omega$ be a bounded, uniformly  convex domain with $\partial \Omega\in C^3$, and let $\phi\in C(\overline \Omega)$ be a convex function satisfying $\phi=0$ on $\p\Omega$.
Let $u\in C(\overline{\Omega})\cap W^{2,n}_{loc}(\Omega)$ be the solution to the linearized
Monge-Amp\`ere equation
\begin{equation*}
\left\{\begin{array}{rl}
\calL_\phi u &=f \qquad \mbox{ in}\quad \Omega,\\
u  &=0\ \ \ \ \ \ \ \mbox{ on}\quad \partial \Omega,
\end{array}\right.
\end{equation*}
where $f\in L^{q}(\Omega)$ with $n<q<\infty$. Then, for any $p\in (1, q)$, there exist $0<\e<1$ and $C>0$ depending only 
on $n,  p, q$ and $\Omega$ such that
\begin{equation*}
\|u\|_{W^{2,p}(\Omega)}\leq
C\|f\|_{L^{q}(\Omega)}
\end{equation*}
provided that the Monge-Amp\`ere measure of $\phi$ satisfies
$$1-\e\leq \det D^2\phi\leq 1 + \e\quad\mbox{in}\quad \Omega.$$
\end{theorem}

As a corollary of our method of the proof of Theorem \ref{main-result-constant}, we obtain global $W^{2,p}$ estimates for equation \eqref{LMA-eq}   when 
the Monge-Amp\`ere measure $\det D^2\phi$ is continuous. Our second main theorem states
as follows.
\begin{theorem}\label{main-result}
Let $\Omega$ be a bounded, uniformly  convex domain with $\partial \Omega\in C^3$, and let $\phi\in C(\overline \Omega)$ be the convex solution to the Monge-Amp\`ere equation
\begin{equation*}
\left\{\begin{array}{rl}
\det D^2 \phi &=g \qquad \mbox{ in}\quad \Omega,\\
\phi  &=0\ \ \ \ \ \ \ \mbox{ on}\quad \partial \Omega,
\end{array}\right.
\end{equation*}
where
$g\in C(\overline\Omega)$ is a continuous function satisfying $0<\lambda \leq g(x)\leq \Lambda$ in $\Omega$. Let $u\in C(\overline{\Omega})\cap W^{2,n}_{loc}(\Omega)$ be the solution to the linearized
Monge-Amp\`ere equation
\begin{equation*}
\left\{\begin{array}{rl}
\calL_\phi u &=f \qquad \mbox{ in}\quad \Omega,\\
u  &=\varphi\ \ \ \ \ \ \ \mbox{ on}\quad \partial \Omega,
\end{array}\right.
\end{equation*}
where $\varphi\in W^{2, s}(\Omega)$,  $f\in L^{q}(\Omega)$ with $n<q<s<\infty$. Then, for any $p\in (1, q)$, there exists $C>0$ depending only 
on $\lambda, \Lambda, n,  p, q,s, \Omega$ and the modulus of continuity of $g$ such that
\begin{equation*}
\|u\|_{W^{2,p}(\Omega)}\leq
C\left(\|\varphi\|_{W^{2, s}(\Omega)} + \|f\|_{L^{q}(\Omega)}\right).
\end{equation*}
\end{theorem}

 Our estimates are affine invariant analogues
of the global $W^{2,p}$ estimates of Winter \cite{Wi} for fully nonlinear, uniformly elliptic equations, and are also linearized counterparts of Savin's global $W^{2,p}$ estimates for the
Monge-Amp\`ere equation \cite{S3}. We note that the continuity condition on the Monge-Amp\`ere measure in Theorem~\ref{main-result} is sharp in view of  Wang's counterexample \cite{W}
for solutions to the Monge-Amp\`ere equation and the fact that $\calL_\phi \phi = n \det D^2\phi =ng$.
The global second derivative estimates in Theorem~\ref{main-result-constant} and Theorem~\ref{main-result} depend only on the bounds on the Hessian determinant $\det D^{2} \phi$ and its 
continuity or closeness to a constant, the geometry of $\Omega$ and the quadratic separation of
$\phi$ from its tangent planes on the boundary $\p\Omega$. This quadratic separation  is
guaranteed by the $C^{3}$ character of boundary domain $\p\Omega$, data $\phi\mid_{\p\Omega}$ and the uniform convexity 
of $\Omega$ (see Proposition~\ref{pro:quadsep}). Under 
the assumptions in the main theorems, the linearized Monge-Amp\`ere operator $\mathcal{L}_{\phi}$ is not uniformly 
elliptic, i.e., the eigenvalues of $\Phi = (\Phi^{ij})$ are not necessarily bounded away 
from $0$ and $\infty.$ Moreover, $\mathcal{L}_{\phi}$ can be possibly singular near the boundary. The degeneracy and 
singularity of $\mathcal{L}_{\phi}$ are the main difficulties in establishing our boundary regularity results. We handle 
the degeneracy of $\mathcal{L}_{\phi}$ by working as in \cite{CG2, GN1, GN2, LS, L} with sections of solutions to 
the Monge-Amp\`ere equations. These sections have the same role as Euclidean balls have in the classical theory. To 
overcome the singularity of $\mathcal{L}_{\phi}$ near the boundary, we use a Localization Theorem at the 
boundary for solutions to the Monge-Amp\`ere equations which was obtained by Savin \cite{S1,S2}. In order to obtain the desired
global second derivative estimates for solutions $u$ of $\mathcal{L}_{\phi}$, we need to have good global decay estimates for the distribution function of the second
derivatives of $u$. To this end, we approximate $u$ by
solutions of $\mathcal{L}_{w}$ where $w$ solves  the standard Monge-Amp\`ere equation  $\det D^2 w=1$ with appropriate boundary conditions,  and use fine geometric properties of 
boundary sections  for solutions to the Monge-Amp\`ere equation which were  obtained recently in \cite{LN}. \\

Though the statements of our main theorems are rather succinct, their proofs are quite delicate. There are essentially two main steps for the proof of the main estimates:

{\bf Step 1:} We consider the quasi distance $d(x,\bar x)$ induced by the solution $\phi$ to the Monge-Amp\`ere equation and is defined by $d(x,\bar x)^2 :=  \phi(x) -\phi(\bar x) - \nabla \phi(\bar x)\cdot (x -\bar x)$.
We then bound the distribution function of the second derivative $D^{2} u$ by the Lebesgue measures of the ``bad'' sets on whose 
complements the quasi 
distance $d(x, \bar{x})$  is comparable to the Euclidean distance $|x-\bar x|$ in a controllable manner
and the graph of $u$ is touched from above and below by ``quasi paraboloids'' generated by the quasi 
distance. Intuitively, the better the regularity of $\phi$ is, the faster these decay estimates can 
be expected. When $\phi(x) = \abs{x}^2/2$, the Monge-Amp\`ere measure $\det D^2\phi$ is the usual Lebesgue measure and $d(x, \bar{x})$ corresponds to 
the Euclidean distance. In this step, we establish
preliminary power decay estimates for the bad sets under natural assumptions on the domain $\Omega$ and the boundary data of $\phi$. As a result, we obtain 
global $W^{2,\delta}(\Omega)$ estimates 
for $u$ where  $\delta>0$ is small under these natural assumptions provided that the Monge-Amp\`ere measure $\det D^2\phi$ is close to a constant.  We also 
give a more direct proof of  
global $W^{2,\delta}$ estimates for solutions to the linearized Monge-Amp\`ere equations when the Monge-Amp\`ere measure is only assumed to be 
bounded away from $0$ and $\infty$. This direct proof is based on interior estimates without resorting to decay estimates of the distribution function of the
second derivatives. These estimates, that are of independent interest, are global counterparts of Guti\'errez-Tournier's interior $W^{2,\delta}$ 
estimates for solutions to the linearized
 equation \eqref{LMA-eq}. Our idea, which is similar to Savin's arguments in \cite{S3}, is rather simple but useful for the second step and can be roughly described as follows:
$$\text{local estimates} + \text{appropriate covering results}\Longrightarrow \text{global estimates}.$$ 
 
{\bf Step 2:} We 
improve the power decay estimates obtained in {\bf Step~1} assuming in addition that $\det D^2\phi$ is sufficiently close to $1$. This 
will involve two main auxiliary results: 
\begin{myindentpar}{1cm}
 1) a global 
stability of cofactor matrices: we prove that the cofactor matrices of the Hessian matrices of two convex functions defined on the same domain are close if their Monge-Amp\`ere measures and boundary values are close in the $L^{\infty}
$ norm;\\
2) a global approximation result: we approximate the solution $u$ by smooth solutions of linearized Monge-Amp\`ere equations
associated with convex functions whose Monge-Amp\`ere measures and boundary data are close to those of $\phi$. 
\end{myindentpar}
The main estimates will then follow from a covering 
theorem for boundary sections and a strong-type $p-p$ estimate for the maximal function corresponding to  boundary sections. \\

Without going into details, we now indicate key technical points that entail for  getting global $W^{2,p}$ estimates. First, we show that the distribution function 
$|\{x: |D^2 u| >\beta\}|$
of the second derivatives of
the solution $u$ to $\mathcal{L}_{\phi} u=f$ has some  decay of the form $C\beta^{-\tau}$ with $\tau>0$ small and $C>0$ depending only on the structural constants in our equation; see Proposition~\ref{Initial-Estimate} and
Proposition~\ref{power-decay-at-boundary}. In the next step, we refine these decay estimates by working in very small regions of the domain and by rescaling our equation and domain. In this rescaled
setting, the constant $C$ above can be improved, roughly by a factor of $\|\Phi-W\|_{L^{n}} + (\fint \abs{f}^n)^{1/n}$; see Lemma~\ref{lm:acceleration}. Here $W$ is the matrix of the cofactors
of $D^2 w$ where $w$ is the solution to the standard Monge-Amp\`ere equation $\det D^2 w=1$ having the same boundary values as $\phi$ in small regions. When $\det D^2\phi$ is close to $1$,
the term $\|\Phi-W\|_{L^{n}}$ can be made as small as we want thanks to the stability of cofactor matrices in Proposition~\ref{global-convergence}. The term 
$(\fint \abs{f}^n)^{1/n}$ is invariant under a rescaling of our equation that almost preserves the $L^{\infty}$-norm of the second derivative $D^2u.$ There are two natural rescalings
of our equation to be explained in Section~\ref{data_sec} but the aforementioned rescaling is the most crucial. As a consequence, $(\fint \abs{f}^n)^{1/n}$ can be made as small as
we want provided that $f$ has higher integrability than $L^{n}$, but this is the assumption in our main theorems.\\

The rest of the paper is organized as follows. In Section~\ref{data_sec}, we recall the main tool used in our proof: the Localization Theorem at the boundary for 
solutions to the Monge-Amp\`ere equation, and state relevant results on the geometry of their sections. 
We also discuss properties of solutions to the Monge-Amp\`ere equation and its linearization under suitable rescalings 
using the Localization Theorem. In addition, we establish boundary $C^{2,\alpha}$ estimates for solutions to the standard Monge-Amp\`ere equations $\det D^2 w=1$
having the same boundary values as $\phi$ on its rescaled sections at the boundary. In Section~\ref{small_decay_sec}, we derive preliminary power decay estimates for the distribution
function of the second derivatives of solutions to the linearized Monge-Amp\`ere equations \eqref{LMA-eq}.
We also establish the global $W^{2,\delta}$ estimates for solutions to \eqref{LMA-eq}, paving the way for proving the global stability of cofactor 
matrices in Subsection~\ref{cofactor_sec}. Moreover, applying the 
global stability of cofactor matrices, we obtain in Subsection~\ref{convex_sec} global $W^{2, 1+\e}$ estimates for convex solutions to the linearized Monge-Amp\`ere
equations when the Monge-Amp\`ere measure is only assumed to be bounded away from zero and infinity. These estimates can be viewed as affine 
invariant versions of results obtained by De Phillipis-Figalli-Savin and Schmidt.
In Subsection~\ref{holder_sec}, we prove  the global H\"older continuity property of 
solutions to \eqref{LMA-eq}. This property together with the boundary $C^{2, \alpha}$ estimates in Section~\ref{data_sec} will be 
instrumental in the global approximation 
lemmas in Subsection~\ref{appro_sec}. In the last section, Section~\ref{density_sec}, by combining these approximation lemmas with the 
preliminary power decay estimates, we obtain  density estimates, which improve the power decay estimates in 
Section~\ref{small_decay_sec}, when the Monge-Amp\`ere measure $\det D^2\phi$ is close to a constant. The proofs of the main results will be 
given at the end of this section using these density estimates, a covering theorem and a strong-type $p-p$ estimate for the maximal function 
with respect to sections.

\section{\bf The Localization Theorem and Geometry of the Monge-Amp\`ere Equation}
\label{data_sec}

The results in this section hold under 
the following global information on the convex domain $\Omega$ and the convex function $\phi$.
We assume there exists $\rho>0$ such that
\begin{equation}\label{global-tang-int}
\Omega\subset B_{1/\rho}, ~\text{and for each $y\in \partial\Omega$ there is a ball $B_\rho(z)\subset \Omega$ that is  tangent to}~ \p 
\Omega~ \text{at $y$}.
\end{equation}
Let $\phi : \overline \Omega \rightarrow \R$, $\phi \in C^{0,1}(\overline 
\Omega) 
\cap 
C^2(\Omega)$  be a convex function satisfying
\begin{equation}\label{eq_u}
\det D^2 \phi =g, \quad \quad 0 <\lambda \leq g \leq \Lambda \quad \text{in}\quad \Omega.
\end{equation}
Assume further that on $\p \Omega$, $\phi$ 
separates quadratically from its 
tangent planes,  namely
\begin{equation}
\label{global-sep}
 \rho\abs{x-x_{0}}^2 \leq \phi(x)- \phi(x_{0})-\nabla \phi(x_{0}) \cdot (x- x_{0})
 \leq \rho^{-1}\abs{x-x_{0}}^2, ~\forall x, x_{0}\in\p\Omega.
\end{equation}
The section of $\phi$ centered at $x\in \overline \Omega$ with height $h$ is defined by
\begin{equation*}
 S_{\phi} (x, h) :=\Big\{y\in \overline \Omega: \quad \phi(y) < \phi(x) + \nabla \phi(x) \cdot (y- x) +h\Big\}.
\end{equation*}
For  $x\in \Omega$, we denote by $\bar{h}(x)$ the maximal height of all sections of $\phi$ centered at $x$ and contained in $\Omega$, that is,
$$\bar{h}(x): =\sup{\big\{h\geq 0| \quad S_{\phi}(x, h)\subset \Omega\big\}}.$$
In this case, $S_{\phi}(x, \bar{h}(x))$  is called the maximal interior section of $\phi$ with center $x\in\Omega$. 
\begin{remark}
In this paper, we denote by $c, \bar{c}, C, C_{1}, C_{2}, \theta_{0}, \theta_{\ast}, \cdots$, positive constants depending only on $\rho$, $\lambda$, $\Lambda$, 
$n$, and their values may change from line to line whenever 
there is no possibility of confusion. We refer to such constants as {\it universal constants}. Small universal constants decrease when $\lambda$ decreases and/or $\Lambda$ increases. 
Large universal constants increase when $\lambda$ decreases and/or $\Lambda$ increases, etc. Therefore, when $1-\e\leq \det D^2\phi\leq 1+\e$ with $0<\e<1/2$, we can suppress the dependence
of universal constants on $\e$.
\end{remark}

\subsection{The Localization Theorem} 
In this subsection, we recall the main tool to study geometric properties of boundary sections of 
solutions to the Monge-Amp\`ere equation: the Localization theorem at the boundary for solution 
to the Monge-Amp\`ere equation (Theorem \ref{main_loc}). Throughout this subsection, we assume that the convex domain $\Omega$ and the 
convex function $\phi$ satisfy \eqref{global-tang-int}--\eqref{global-sep}.
We now focus on sections centered at a point on the boundary  $\p\Omega$ and describe their geometry. Assume this boundary point to be $0$ and by \eqref{global-tang-int}, we can also assume that
\begin{equation}\label{om_ass}
B_\rho(\rho e_n) \subset \, \Omega \, \subset \{x_n \geq 0\} \cap B_{\frac 1\rho},
\end{equation}
where $\rho>0$ is the constant given by condition \eqref{global-tang-int}. 
After subtracting a linear function, we can assume further that
\begin{equation}\label{0grad}
\phi(0)=0\quad \mbox{and} \quad \nabla \phi(0)=0.
\end{equation}
If the boundary data has quadratic growth near $\{x_n=0\}$ then, as $h \rightarrow 0$, $S_{\phi}(0, h)$ is equivalent to a half-ellipsoid centered at 0. This is the content
of the Localization Theorem proved by Savin in \cite{S1,S2}. Precisely, this theorem reads as follows.

\begin{theorem}[Localization Theorem \cite{S1,S2}]\label{main_loc}
 Assume that $\Omega$ satisfies \eqref{om_ass} and $\phi$ satisfies 
\eqref{eq_u},\eqref{0grad}, and
\begin{equation*}\label{commentstar}\rho |x|^2 \leq \phi(x) \leq \rho^{-1} 
|x|^2 \quad \text{on \, $\p \Omega \cap \{x_n \leq \rho\}.$}\end{equation*}
Then there exists  a constant $k=k(\rho, \lambda, \lambda, n)>0$   such that  for each $h \leq k$ there is an ellipsoid $E_h$ of volume $\omega_{n}h^{n/2}$ 
satisfying
$$kE_h \cap \overline \Omega \, \subset \, S_{\phi}(0, h) \, \subset \, k^{-1}E_h \cap \overline \Omega.$$
Moreover, the ellipsoid $E_h$ is obtained from the ball of radius $h^{1/2}$ by a
linear transformation $A_h^{-1}$ (sliding along the $x_n=0$ plane)
$$A_hE_h= h^{1/2}B_1,\, \det A_{h} =1,\,  A_h(x) = x - \tau_h x_n, \,  \tau_h = (\tau_1, \tau_2, \ldots, 
\tau_{n-1}, 0) \,\mbox{ and }\, |\tau_{h}| \leq k^{-1} |\log h|.$$
\end{theorem}

From Theorem \ref{main_loc} we also control the shape of sections that are tangent to $\p \Omega$ at the origin. 

\begin{proposition}\label{tan_sec}
Let $\phi$ and $\Omega$ satisfy the hypotheses of the Localization Theorem \ref{main_loc} at the 
origin. Assume that for some $y \in \Omega$ the section $S_{\phi}(y, h) \subset \Omega$
is tangent to $\p \Omega$ at $0$, i.e., $\p S_{\phi}(y, h)\cap
\p\Omega =\{0\}$, for some $h \le c$ with $c$ universal. Then there exists a small positive
 constant $k_0<k$ depending on $\lambda$, $\Lambda$, $\rho $ and $n$ such that
$$ \nabla \phi(y)=a e_n 
\quad \mbox{for some} \quad   a \in [k_0 h^{1/2}, k_0^{-1} h^{1/2}],$$
$$k_0 E_h \subset S_{\phi}(y, h) -y\subset k_0^{-1} E_h, \quad \quad k_0 h^{1/2} \le \dist(y,\p \Omega) \le k_0^{-1} h^{1/2}, \quad $$
with $E_h$ and $k$ the ellipsoid and constant defined in Theorem~\ref{main_loc}.
\end{proposition}

Proposition \ref{tan_sec} is a consequence of Theorem \ref{main_loc} and was proved in \cite{S3}. \\

The quadratic separation from tangent planes on the boundary for $\phi$
is a crucial assumption in the Localization Theorem (Theorem~\ref{main_loc}). This is the case for solutions to 
the Monge-Amp\`ere equation with the right hand side bounded away from $0$ and $\infty$ on uniformly convex domains and smooth boundary data 
as proved in \cite[Proposition 3.2]{S2}.

\begin{proposition}\label{pro:quadsep}
Let $\Omega\subset\R^{n}$ be a uniformly convex domain satisfying \eqref{global-tang-int}
and $\|\partial\Omega\|_{C^3}\leq 1/\rho$. Let $\phi: \overline{\Omega}\rightarrow \R$, $\phi\in C^{0,1}(\overline{\Omega})\cap C^{2}(\Omega)$  be a convex function 
satisfying $\phi \mid_{\p\Omega}\in C^{3}$ and
$$0<\lambda\leq \det D^{2} \phi \leq \Lambda<\infty\quad \mbox{in}\quad \Omega.
$$
Then, on $\p \Omega$, 
$\phi$ separates quadratically from its tangent planes, that is,
\begin{equation*}\label{eq_u1}
 \rho_{0}\abs{x-x_{0}}^2 \leq \phi(x)- \phi(x_{0})-\nabla \phi(x_{0}) \cdot (x- x_{0}) \leq 
\rho_{0}^{-1}\abs{x-x_{0}}^2,~\forall x_{0}, x\in\p\Omega
\end{equation*}
for some constant $\rho_{0}>0$ depending only on $n, \rho,\lambda, \Lambda$, $\|\phi\|_{C^{3}(\p\Omega)}$ and the uniform convexity of $\Omega$. 
\end{proposition}

\subsection{Properties of the rescaled functions and boundary regularity estimates}
\label{rescale-sec}
In this subsection, we discuss properties of solutions to the Monge-Amp\`ere equation and its linearization under suitable rescalings and then use these properties to establish a boundary
$C^{2, \alpha}$ estimates for solutions to the standard Monge-Amp\`ere equation $\det D^2 w=1$ in our rescaled setting. \\

Let $\Omega$ and $\phi$  satisfy the hypotheses of the Localization Theorem at the 
origin. We know that for all $h \le k$,
 $S_{\phi}(0, h)$ satisfies 
\begin{equation}k E_h \cap \overline \Omega  \subset S_{\phi}(0, h) \subset k^{-1} E_h\cap \overline{\Omega},
 \label{Sh}
\end{equation}
with $A_h$ being a linear transformation 
and
$$\det A_{h} = 1,\quad E_h=A^{-1}_hB_{h^{1/2}}, \quad  A_hx=x-\tau_hx_n,\quad \tau_h \cdot e_n=0, \quad \|A_h^{-1}\|, \,\|A_h\| \le k^{-1} |\log h|.$$
This gives for all $h\leq k$
\begin{equation}\label{small-sec}
 \overline{\Omega} \cap B^{+}_{h^{2/3}}\subset\overline \Omega \cap B^{+}_{ch^{1/2}/\abs{\log h}}\subset S_{\phi}(0, h) 
\subset \overline{\Omega} \cap B^{+}_{C h^{1/2} \abs{\log h}}\subset B^{+}_{h^{1/3} }.
\end{equation}
We denote the rescaled function of $\phi$ and the rescaled domain of $\Omega$ by \begin{equation}
 \phi_h(x):=\frac{\phi(h^{1/2}A^{-1}_hx)}{h}\quad \mbox{and}\quad \Omega_h:= h^{-1/2}A_h \Omega.
\label{phi-omega-h}
 \end{equation}
The function $\phi_h$, defined in $\overline \Omega_h$, is continuous and solves the Monge-Amp\`ere equation
 $$\det D^2 \phi_h=g_h(x), \quad \quad \lambda \le g_h(x):=g(h^{1/2}A_h^{-1}x) \le \Lambda.$$
By \eqref{Sh}, the section of $\phi_h$   centered at the origin and with height 1 satisfies
\begin{equation}
B_k^{+} \cap \overline \Omega_h \subset S_{\phi_{h}}(0, 1)=h^{-1/2}A_hS_{\phi} (0, h)\subset B_{k^{-1}}^+ \cap \overline{\Omega_{h}}.
\label{size-u-h}
\end{equation}
In what follows, we denote
\begin{equation}
U_{h}= S_{\phi_{h}}(0,1).
\label{u-phi-h}
\end{equation}

Now, we discuss two natural rescalings for the linearized Monge-Amp\`ere equation 
$$\mathcal{L}_{\phi}u:= \Phi^{ij} u_{ij}= f \quad \mbox{in}\quad \Omega.$$
We focus on the boundary section $S_{\phi}(0, h)$ in the present setting of Theorem \ref{main_loc}.\\
{\bf $L^{\infty}$-norm preserving rescaling.}  Consider the following rescaling of functions:
$$u_{h}(x):= u(h^{1/2} A^{-1}_{h}x)\, \mbox{ and }\, f_{h}(x):= h f(h^{1/2} A^{-1}_{h}x),\quad \mbox{for }\, x\in\Omega_h.$$
Simple computation gives
\begin{equation*}
 D^2 \phi_h= (A_{h}^{-1})^{t} D^2 \phi A_{h}^{-1},\quad  D^2 u_h= h(A_{h}^{-1})^{t} D^2 u A_{h}^{-1}, 
\end{equation*}
and
$$\Phi_h: = (\det D^2 \phi_h) (D^2\phi_h)^{-1} = (\det D^2\phi) A_{h} (D^2\phi)^{-1} (A_{h})^{t} = A_h \Phi (A_h)^{t}.$$
Therefore, we find that
$$\mathcal{L}_{\phi_{h}} u_{h}= trace (\Phi_h D^2 u_h)=f_{h}~\text{ in }~\Omega_{h},\quad
\mbox{and}\quad \|u_h\|_{L^{\infty}(\Omega_h)} = \|u\|_{L^{\infty}(\Omega)}.$$
Thus this rescaling preserves the $L^{\infty}$-norm of $u$.
Since $\|f_h\|_{L^{n}(\Omega_h)}= h^{1/2}\|f\|_{L^{n}(\Omega)}$ is small if $f\in L^{n}(\Omega)$ and $h$ small, we can expect that $u_h$ has some nice
second derivative estimates, say their boundedness. Given this and  as
$$D^{2} u (h^{1/2} A^{-1}_{h} x) = h^{-1} (A_{h})^{t} \, D^{2} u_{h}(x) \, A_{h},$$
it is again quite natural to expect that $|D^2u |$ behaves like $\frac{1}{h}$ in some part of the section $S_{\phi}(0, h).$ This is what we will prove in 
Lemma~\ref{lm:improved-density-I}.\\
{\bf Almost $W^{2,\infty}$-norm preserving rescaling.} The next rescaling almost preserves the $L^{\infty}$-norm of $D^2u$. Under the following rescaling of functions
$$\tilde u_{h}(x):= h^{-1}u(h^{1/2} A^{-1}_{h}x) \, \mbox{ and }\, \tilde f_{h}(x):= f(h^{1/2} A^{-1}_{h}x)\quad  \mbox{for }\,x\in\Omega_h,$$
we have
$\mathcal{L}_{\phi_{h}} \tilde u_{h}= \tilde f_{h}~\text{ in }~\Omega_{h}$
with
 $$\fint_{\Omega_h}\abs{\tilde f_{h}}^n = \fint_{\Omega} \abs{f}^n,$$
by changing variables and recalling that $\det A_h =1.$
As
$$D^{2} \tilde{u}_{h}(x) = (A_{h}^{-1})^{t}D^{2} u( h^{1/2} A^{-1}_{h} x) A_{h}^{-1}, $$
the present rescaling almost preserves the $L^{\infty}$-norm of $D^2u$  since
$$\|D^2\tilde u_h\|_{L^{\infty}(\Omega_h)}\leq k^{-2}\abs{\log h}^2 \|D^2 u\|_{L^{\infty}(\Omega)}.$$

In principle, the $L^{\infty}$-norm preserving rescaling allows us to find some good points with controlled second derivatives for $u$. Having found them, we would like to propagate 
them by finding more similar points near by, maybe at the cost of a slightly larger bound on the second derivatives. This is the key technical point of the paper and  almost $W^{2,\infty}$-norm preserving 
rescaling is the means for this; see Lemmas~\ref{lm:improved-density-I} and \ref{lm:furthercriticaldensitytwo}.

  A variant of the {\bf $L^{\infty}$-norm preserving rescaling} is the following which applies to sections tangent to the boundary. \\
\noindent
{\bf $L^{\infty}$-norm preserving rescaling in a section tangent to the boundary.} Consider a prototype section $S_{\phi}(y, h)$
with $h:= \bar{h}(y)\leq c.$ By applying Proposition~\ref{tan_sec}  
to $S_{\phi}(y, h)$, we see that it is equivalent to an ellipsoid $E_h$, i.e.,
$$k_0 E_h \subset S_{\phi}(y, h)-y \subset k_{0}^{-1} E_h,$$
where
$$E_h := h^{1/2}A_{h}^{-1}B_1 \quad \mbox{with} \quad \det A_{h}=1,\quad  \|A_{h}\|, \, \|A_h^{-1} \| \le C |\log h|.$$
We use the following rescalings:
$$\tilde{\Omega}_{h}  := h^{-1/2} A_{h} (\Omega-y),$$
and  for $ x\in \tilde{\Omega}_{h}$
$$\tilde{u}_{h}(x) := u ( y + h^{1/2} A^{-1}_{h} x),\quad  \tilde{\phi}_{h}(x) := 
h^{-1} \left[\phi ( y + h^{1/2} A^{-1}_{h} x)-\phi(y)-\nabla\phi(y)\cdot(h^{1/2}A^{-1}_h x) - h\right].$$
Then
$$B_{k_0}\subset \tilde{U}_h\equiv S_{\tilde{\phi}_{h}} (0, 1)\equiv h^{-1/2} A_{h} \big(S_{\phi}(y, h) -y\big)\subset B_{k_0^{-1}}.$$
We have
$$\det D^{2} \tilde{\phi}_{h} (x)= \tilde{g}_{h}(x):= g ( y + h^{1/2} A^{-1}_{h} x), \quad \tilde{\phi}_{h} =0 \, \mbox{ on }\, \partial  S_{\tilde{\phi}_{h}} (0, 1)$$
and
\[
\min_{ S_{\tilde{\phi}_{h}} (0, 1)}{\tilde{\phi}_{h}}=-1 =\tilde{\phi}_{h}(0).
\]
Also
$$\tilde{\Phi}_{h}^{ij}(\tilde{u}_{h})_{ij} = \tilde{f}_{h}(x) : = hf ( y + h^{1/2} A^{-1}_{h} x).$$

Some properties of the rescaled function $\phi_h$ was established in \cite{S2} and \cite[Lemma~4.2, Lemma~5.4]{LS}. For later use, we record them here.
\begin{lemma}\label{sep-lem}
There exists a small  constant $c=c(n,\rho,\lambda,\Lambda)>0$ such that if $h\leq c$, then

a) for any 
$x,x_{0}\in\p\Omega_{h}\cap B_{2/k}$ we have
\begin{equation}\label{sep-near0}
 \frac{\rho}{4}\abs{x-x_{0}}^2 \leq \phi_h(x) - \phi_h(x_{0}) -\nabla \phi_h(x_{0}) \cdot (x- x_{0})\leq 4\rho^{-1} \abs{x- x_{0}}^2.
\end{equation}

b) if $r \le c$ small, we have $$|\nabla \phi_h| \le 
C r |\log r|^2 \quad \mbox{in} \quad \overline \Omega_h \cap B_r.$$

c) $\p \Omega_h \cap B_{2/k}$ is a graph 
in the $e_n$ direction whose $C^{1,1}$ norm 
is bounded by $C h^{1/2}.$

d) $\phi_{h}$ satisfies in $U_{h}\equiv S_{\phi_{h}}(0,1)$  the hypotheses of Theorem~\ref{main_loc} at all points  on $\p U_{h}\cap B_{c}$.

e) If $y\in U_{h}\cap B_{c^2}$ then the maximal interior section $S_{\phi_h}(y, \bar{h}(y))$ of $\phi_h$ in $U_{h}$ satisfies: $$c\geq \bar{h}(y)\geq k_0^2 \, \dist^2 (y, \p U_h)\quad \mbox{and}\quad 
S_{\phi_h}(y, \bar{h}(y))\subset U_h\cap B_c.$$
\end{lemma}
\begin{proof}
 \cite[Lemma 4.2]{LS} contains (a)--(c) while its proof implies (d). The statement (e) can be proved as in \cite[Lemma 5.4]{LS} and we give a 
 complete proof   here. Let $y\in U_{h}\cap B_{c^2}$. Then it follows from property (d)  and \eqref{small-sec} that $y\in S_{\phi_h}(0, c^3)$. 
 Hence, $\phi_h(y)\leq c^3.$ By \cite[Lemma~4.1]{LN} we obtain
 $S_{\phi_h}(0, c^3) \subset S_{\phi_h}(y, \theta_0 c^3)$ and consequently 
 \begin{equation}\label{barh-is-small}
 \bh(y)\leq \theta_0 c^3.
 \end{equation}
 Thus, 
 $\bar{h}(y)\leq c$
 if $c$ is small.
   Since $S_{\phi_h}(y, \bar{h}(y))$ is balanced around $y$, we can use Theorem~3.3.8 in \cite{G} to conclude that 
  \begin{equation}\label{section-inside-a-ball}
  S_{\phi_h}(y, \bh(y)) \subset B(y, K\,\bh(y)^b)
  \end{equation}
  for some universal constants $K, b>0$.
  
  From \eqref{barh-is-small} and \eqref{section-inside-a-ball} we see that for $c$  small the section $S_{\phi_h}(y, \bar{h}(y))$ is 
  tangent to $\partial \Omega_h$. Let 
 $x_0 \in \p S_{\phi_h}(y, \bar{h}(y)) \cap\p \Omega_h$.
Applying \eqref{sep-near0} to $x_0$ and $0$, and using  property (b) and \eqref{section-inside-a-ball}, we have
\begin{align*}
\frac{\rho}{4} |x_0|^2 
&\leq
\phi_h(x_0)= \phi_h(y)+\nabla \phi_h(y)\cdot (x_0 -y) +\bh(y)\\
&\leq c^3 + CK\abs{y}\bh(y)^b\big|\log{|y|}\big|^2+\bh(y).
\end{align*}
This together with the assumption $|y|< c^2$ and \eqref{barh-is-small} implies that $|x_0|<c$. Now, thanks to (d) we can apply Proposition~\ref{tan_sec} at $x_0$ and obtain
 $$k_0^2 \dist^2 (y, \p U_h)\leq \bar{h}(y)\leq k_{0}^{-2} \dist^2 (y, \p U_h).
 $$
 Since $S_{\phi_h}(y, \bar{h}(y))-y \subset k_{0}^{-1} E_h$, we find from the definition of $E_h$ and  $\bh(y)\leq \theta_0 c^3$
 that 
$$S_{\phi_h}(y, \bar{h}(y))\subset y + k_{0}^{-1} E_h\subset B_{c^2 + k_0^{-1}k^{-1} \abs{\bar{h}(y)}^{1/2}\abs{\log \bar{h}(y)}}\subset B_c$$
if $c$ is universally small. 
\end{proof}
\begin{remark}
\label{choosingc}
 From now on, we fix a universally small constant $c\leq k/2$, $c\ll 1$ depending only on $n, \rho, \lambda, \Lambda$ as in the Lemma \ref{sep-lem}.
\end{remark}

The rest of this subsection is devoted to establishing boundary $C^{2,\alpha}$ estimates for the convex solution $w$ to the standard Monge-Amp\`ere equation
\begin{equation}\label{MA-RHS=1}
\left\{\begin{array}{rl}
\det D^2 w &=1 \qquad \mbox{ in}\quad U_h :=S_{\phi_{h}}(0, 1),\\
w  &=\phi_h\ \ \ \ \ \ \mbox{ on}\quad \partial U_h.
\end{array}\right.
\end{equation}
For this, we first show in the next lemma that $w$ separates quadratically from its tangent planes on  the boundary of $U_h$.

\begin{lemma}\label{lm:quadratic-separation-U}
Let  $\Omega_h$, $\phi_h$ and $U_h$ be as in (\ref{phi-omega-h}) and (\ref{u-phi-h}) with $h\leq c$. Let $w\in C(\overline{U_h})$ be the convex solution to \eqref{MA-RHS=1}.
Then there exist universal constants $\delta, \theta>0$ depending only on $n, \rho, \lambda, \Lambda$ such that for any
 $x_0 \in \partial U_h\cap B_{c}$,
\begin{equation*}
 x_{n+1}= \phi_h(x_0) + \langle \nabla\phi_h(x_0) -2\delta^{1-n} k^{-1} \nu_{x_0}, x-x_0\rangle =: \bar l_{x_0}(x) 
\end{equation*}
 is a supporting hyperplane  in $\overline{U_h}$ to $w$ at $x_0$, and
 \begin{equation}\label{quadratic-separation}
 \theta \, |x - x_0|^2 \leq w(x) - \bar l_{x_0}(x)  \leq \theta^{-1} \, | x- x_0|^2\quad \mbox{for all}\quad x\in \partial U_h. 
 \end{equation}
 Here $\nu_{x_0}$ denotes the unit inner normal to $\partial\Omega_h$ at $x_0$.
 \end{lemma}
\begin{proof}
For $x_0 \in \partial U_h\cap B_{c}$, let $l_{x_0}(x) := \phi_h(x_0) +\nabla \phi_h(x_0) \cdot (x- x_0)$. Then by Lemma~\ref{sep-lem}(a), 
\begin{equation}\label{separation-1}
\frac{\rho}{4} |x-x_0|^2 \leq \phi_h(x) - l_{x_0}(x)\leq \frac{4}{\rho} |x-x_0|^2 \quad \forall x\in \partial U_h \cap \partial\Omega_h.
\end{equation}
By Lemma~\ref{sep-lem}(d) and a consequence of the Localization Theorem \ref{main_loc} (see (\ref{small-sec})), there is $r_0>0$ universally small depending only on 
$n, \rho, \lambda, \Lambda$ such that
\[
S_{\phi_h}(x_0, r_0)\subset B(x_0, \frac{k}{2})\cap \overline{U_h}\subset B_{k}\cap \overline{U_h}.
\]
This gives $\phi_h(x) \geq  l_{x_0}(x) + r_0$ for all $x\in \partial U_h \setminus \partial \Omega_h$, and consequently, by (\ref{size-u-h})
\begin{equation}\label{strict-convexity}
\phi_h(x) \geq l_{x_0}(x) + \frac{k^2 r_0}{4} |x- x_0|^2\quad \forall x\in \partial U_h \setminus \partial \Omega_h.
\end{equation}
Define
\[
w^-(x) :=  l_{x_0}(x) + \delta \, \Big[|x - x_0|^2
- |(x - x_0) \cdot \nu_{x_0}|^2\Big] + \delta^{1-n} \, \Big[ |(x - x_0) \cdot\nu_{x_0}|^2 -2 k^{-1} (x - x_0) \cdot\nu_{x_0}\Big]\quad \forall x\in \overline{U_h},
\]
where 
\[
\delta :=\min{\{\frac{\rho}{4}, \frac{k^2 r_0}{4}\}}.
\]
Then $w^-$ is a convex function in $\overline{U_h}$ satisfying $D^2 w^- = 2\delta \, [ I + (\delta^{-n} -1) \, \nu_{x_0} \otimes \nu_{x_0}]$. Therefore,
\begin{equation}\label{determinant-1}
\det D^2 w^- = (2\delta)^n  \delta^{-n}
= 2^n> 1 =\det D^2 w \quad \mbox{in}\quad U_h.
\end{equation}
For $x\in \partial U_h\cap \partial \Omega_h$,   we obtain  from  $0\leq (x - x_0) \cdot\nu_{x_0} \leq 2 k^{-1}$ and the first inequality in \eqref{separation-1}  that 
\[
w^-(x) \leq   l_{x_0}(x) + \delta |x - x_0|^2
\leq \phi_h(x)-\frac{\rho}{4} |x-x_0|^2   +\delta |x - x_0|^2\leq \phi_h(x)= w(x).
\]
On the other hand, for  $x\in \partial U_h\setminus \partial \Omega_h$ by using  \eqref{strict-convexity} we have
\begin{align*}
w^-(x)\leq  l_{x_0}(x) + \delta |x - x_0|^2
\leq  l_{x_0}(x) + \frac{k^2 r_0}{4} |x- x_0|^2
\leq \phi_h(x)=w(x).
\end{align*}
Therefore, $w\geq w^-$ on $\partial U_h$. It follows from this, \eqref{determinant-1} and the comparison principle that
$w(x)\geq w^-(x)$ in $\overline{U_h}$. Hence,
\begin{align}\label{lower-separation}
w(x)
&\geq \bar l_{x_0}(x) +  \delta \, \Big[|x - x_0|^2
- |(x - x_0) \cdot \nu_{x_0}|^2\Big] + \delta^{1-n} \,  |(x - x_0) \cdot\nu_{x_0}|^2 \\
&\geq \bar l_{x_0}(x) + \delta |x - x_0|^2   \qquad\qquad\mbox{in}\quad \overline{U_h}.\nonumber
\end{align} 
In particular, $w(x)\geq \bar l_{x_0}(x)$ for all $x\in \overline{U_h}$. Since $\bar l_{x_0}(x_0) = \phi_h(x_0) = w(x_0)$, we then conclude 
that $x_{n+1}=\bar l_{x_0}(x)$ is a supporting hyperplane in $ \overline{U_h}$ to $w$ at $x_0$.

We now show the second inequality in \eqref{quadratic-separation}. For this, we first recall that $0\leq \phi_h\leq 1$ in $U_{h}$ and by 
Lemma~\ref{sep-lem}(b), we find that for $M:= 1 + 2k^{-1} C \,c\abs{\log c}^2,$
\begin{equation}\label{U-is-a-big-section}
\phi_h(x)\leq  \phi_h(x_0) +\nabla \phi_h(x_0) \cdot (x- x_0) + M\equiv l_{x_0}(x) + M\quad \forall x\in \overline{U_h}.
\end{equation}
We now compare $w$ with $w^{+}$ defined by
\[
w^+(x) :=  l_{x_0}(x)+ 2\Theta k^{-1}\, (x - x_0)\cdot \nu_{x_0}  + \Theta \, \Big[|x - x_0|^2
 - | (x- x_0)\cdot \nu_{x_0}|^2\Big]  
\quad \forall x\in \overline{U_h},
\]
where 
\[
\Theta := \max{\Big\{\frac{4}{\rho}, \, \frac{4 M}{k^2}\Big\}}.
\]
Clearly, $w^+$ is a convex function in $\overline{U_h}$ satisfying
\begin{equation}\label{determinant-2}
\det D^2 w^+ 
=0< 1 =\det D^2 w \quad \mbox{in}\quad U_h.
\end{equation}
For $x\in \partial U_h\cap \partial \Omega_h$, we obtain from the second inequality in \eqref{separation-1} and $\Theta\geq \frac{4}{\rho}$ that 
\begin{align*}
w^+(x) 
&=  l_{x_0}(x) + \Theta |x - x_0|^2
+  \Theta  \Big[2 k^{-1} \, (x - x_0) \cdot\nu_{x_0} - |(x - x_0) \cdot\nu_{x_0}|^2\Big]\\
&\geq \phi_h(x)-\frac{4}{\rho} |x-x_0|^2   +\Theta |x - x_0|^2\geq \phi_h(x)= w(x).
\end{align*}
For  $x\in \partial U_h\setminus \partial \Omega_h$, we have $\abs{x-x_0}\geq k/2$ and thus, by using  \eqref{U-is-a-big-section} we obtain
\begin{align*}
w^+(x)
\geq  l_{x_0}(x)+  \Theta |x - x_0|^2
\geq \phi_h(x) - M + \frac{k^2 \Theta}{4}
\geq \phi_h(x)=w(x).
\end{align*}
Therefore, $w\leq w^+$ on $\partial U_h$. It follows from this, \eqref{determinant-2} and the comparison principle that $w\leq w^+$ in $\overline{U_h}$. In particular,
\begin{align*}\label{w-upper-tangent-est}
w(x)
&\leq l_{x_0}(x)+ 2\Theta k^{-1}\, (x - x_0)\cdot \nu_{x_0}  + \Theta \, |x - x_0|^2\\
&=
\bar l_{x_0}(x)+ 2 k^{-1} \big(\delta^{1-n} +\Theta \big)\, (x - x_0)\cdot \nu_{x_0}  + \Theta \, |x - x_0|^2  \quad 
  \forall x\in\overline{U_h}.\nonumber
\end{align*}
 We then use Lemma~\ref{sep-lem}(c) for $x\in \partial U_h \cap \partial\Omega_h$ and the fact that $k/2\leq |x - x_0| \leq 2/k$ for $x\in \partial U_h\setminus 
\partial\Omega_h,$ to conclude that 
\begin{align*}
w(x)
\leq \bar l_{x_0}(x) + C|x - x_0|^2  \quad
 \forall x\in \partial U_h.
\end{align*} 
This together with \eqref{lower-separation} 
gives the quadratic separation in \eqref{quadratic-separation}.
\end{proof}

Thanks to the quadratic separation property of $w$ in Lemma~\ref{lm:quadratic-separation-U}, we can now 
apply Savin's boundary $C^{2,\alpha}$ estimates
for solutions to the Monge-Amp\`ere equations \cite{S2} to get boundary $C^{2,\alpha}$ estimates for $w$ when
$\p\Omega\cap B_{\rho}$ and $\phi\mid_{\p\Omega\cap B_{\rho}}$ are $C^{2,\alpha}$ and $h$ is small. 

\begin{proposition}\label{D^2w-est}
Let $\Omega$ and $\phi$  satisfy the hypotheses of the Localization Theorem~\ref{main_loc} at the 
origin.
 Assume in addition that $\p\Omega\cap B_{\rho}$ is $ C^{2,\alpha}$ and $\phi\in C^{2,\alpha}(\p\Omega\cap B_{\rho})$ for some $\alpha\in (0,1)$. Let  $\Omega_h$, $\phi_h$, $U_h$ and $w$
be as in Lemma~\ref{lm:quadratic-separation-U}. Then there exists $h_0>0$ depending on $n, \lambda,
\Lambda, \rho, \alpha$, $ \|\p\Omega\cap B_{\rho}\|_{C^{2,\alpha}}$ and $ \| \phi\|_{C^{2,\alpha}(\p\Omega\cap B_{\rho})}$ such that for any  $h\leq h_0$, we have
\begin{equation}\label{eq:Boundary-Pogorelov}
\|w\|_{C^{2,\alpha}(\overline{B_{c}\cap U_h})}\leq c_0^{-1}
\quad \mbox{and}\quad 
c_{0}I_{n}\leq D^2 w\leq c_{0}^{-1} I_{n}~\text{ in }~ 
B_{c}\cap U_h
\end{equation}
for some $c_{0}>0$ depending only on $n, \lambda, \Lambda, \alpha$ and $\rho$. 
\end{proposition}
Now,  let us assume in addition  that $\p\Omega$ and $\phi|_{\p\Omega}$ are $C^{2,\alpha}$ at the origin for some $\alpha\in (0,1)$, that is, we  assume that for $x=(x', x_n)\in \p\Omega\cap B_{\rho}$, we have
\begin{equation*}
\abs{x_{n}- q(x')}\leq M\abs{x'}^{2+\alpha} \quad\mbox{and}\quad
\abs{\phi- p(x')}\leq M\abs{x'}^{2+\alpha},
\end{equation*}
where $p(x')$ and $q(x')$ are homogeneous quadratic polynomials. \\

If $h$ is sufficiently small, then the corresponding rescaling $\phi_{h}$ satisfies the hypotheses of $\phi$ in which the constant $M$ is replaced by an arbitrary small constant $\sigma$.
\begin{lemma}(\cite[Lemma 7.4]{S2})
\label{small-M}
Given any $\sigma>0$, there exists a small positive constant $h = h_{0}(M,\sigma, \alpha, n,\lambda,\Lambda,\rho)$ such that on $\p\Omega_{h}\cap B_{k^{-1}}$, we have
$$\abs{x_{n}- q_{h}(x')}\leq \sigma\abs{x'}^{2+\alpha},\quad \abs{q_{h}(x')}\leq \sigma\quad\text{and}\quad
\abs{\phi_{h}- p(x')}\leq \sigma\abs{x'}^{2+\alpha},$$
where $q_{h}(x') := h^{1/2} q(x')$  is a homogeneous quadratic polynomial.
\end{lemma}
\begin{remark} 
\label{sig-rem}
By inspecting the proof of Lemma~7.4 in \cite{S2}, we see that the following more precise statement holds true: 
There exists $C= C(M, n, \lambda,\Lambda, \rho)>0$ such that for any $h\leq c$, on $\p\Omega_{h}\cap B_{k^{-1}}$ we have 
\begin{align*}
\abs{x_{n}- q_{h}(x')}\leq C h^{\frac{1+\alpha}{2}}\abs{x'}^{2+\alpha},\quad \abs{q_{h}(x')}\leq C h^{\frac12}~ \text{ and } ~
\abs{\phi_{h}- p(x')}\leq C h^{\frac{\alpha}{2}}\abs{x'}^{2+\alpha}.
\end{align*}
\end{remark}

\begin{proof}[Proof of Proposition \ref{D^2w-est}]
Let $M :=max{\big\{ \|\p\Omega\cap B_{\rho}\|_{C^{2,\alpha}}, \| \phi\|_{C^{2,\alpha}(\p\Omega\cap B_{\rho})}\big\}}$ and
let $h_0$ be the small constant  in Lemma~\ref{small-M} corresponding to $M$ and  $\sigma=1$. Then by our assumptions, Lemma~\ref{small-M}, Remark~\ref{sig-rem} and  
Lemma~\ref{lm:quadratic-separation-U}, we can apply   \cite[Corollary~7.2]{S2} to conclude that there exist $C, \delta>0$ depending on $n,\lambda, \Lambda, \alpha$ and $\rho$  such that
\[
\|w\|_{C^{2,\alpha}(\mathcal{C}_0 \cap B_\delta(0))}\leq C,
\]
where $\mathcal{C}_0 := \{x\in \R^n_+ :\, |x'|\leq x_n\}$ is  the cone at the origin with opening $\theta= \pi/4$. 

By varying the point under consideration, we then conclude in the similar fashion that
\begin{equation}\label{C^{2,alpha}-at-a-point}
\|w\|_{C^{2,\alpha}(\mathcal{C}_{x_0} \cap B_\delta(x_0))}\leq C \quad \forall x_0\in \partial \Omega_h \cap \overline{B_{c}}.
\end{equation}
Here  $\mathcal{C}_{x_0} := \{x\in \R^n_+ :\, |x - x_0|^2\leq 2 |(x- x_0)\cdot \nu_{x_0}|^2\}$ is the cone at $x_0$ with 
opening $\theta= \pi/4$ and in the direction of $\nu_{x_0}$, the unit inner normal to $\partial\Omega_h$ at $x_0$. As a consequence of \eqref{C^{2,alpha}-at-a-point} and 
Caffarelli's interior $C^{2,\alpha}$ estimates \cite{C3}, we obtain  the first estimate in \eqref{eq:Boundary-Pogorelov} from which the second estimate in \eqref{eq:Boundary-Pogorelov} follows.
\end{proof}

\subsection{The classes $\mathcal{P}_{\lambda,\Lambda, \rho, \kappa,\alpha}$ and  
$\mathcal{P}_{\lambda,\Lambda, \rho, \kappa, \ast}$ } Fix $n, \rho, \lambda, \Lambda, \kappa$ and $\alpha$.
We define the classes $\mathcal{P}_{\lambda,\Lambda, \rho, \kappa, \alpha}$ and $\mathcal{P}_{\lambda,\Lambda, \rho, \kappa, \ast}$ consisting of the 
triples $(\Omega, \phi, U)$ satisfying the following sets of conditions $(i)-(vii)$ and $(i)-(vi)$, respectively:
\begin{myindentpar}{1cm}
(i) $0 \in \p\Omega, U\subset\Omega\subset \R^{n}$ are bounded convex domains such that
$$B_{k}^{+}\cap \overline{\Omega}\subset\overline{U}\subset B_{k^{-1}}^{+}\cap\overline{\Omega}.$$
(ii) $\phi:\overline{\Omega}\rightarrow \R^{+}$ is convex satisfying $\phi=1$ on $\p U\cap\Omega$ and
$$\phi(0)=0, \quad \nabla \phi(0)=0, \quad \lambda \leq \det D^{2} \phi \leq \Lambda \,\mbox{ in }\, \Omega,\quad 
\p\Omega\cap\{\phi<1\}= \p U\cap \{\phi<1\}.$$
(iii) (quadratic separation) $$\frac \rho 4\abs{x-x_{0}}^2 \leq \phi(x)- \phi(x_{0})-\nabla \phi(x_{0}) 
\cdot (x- x_{0}) \leq \frac 4\rho \abs{x-x_{0}}^2 \quad \quad \forall x, x_{0}\in  \p \Omega\cap B_{\frac{2}{k}}.$$
(iv) (flatness)  $$\p \Omega \cap \{\phi<1\} \subset G \subset  \{x_n \le \kappa \}$$
 where $G\subset B_{2/k}$ is a graph in the $e_n$ direction  and its $C^{1,1}$ norm is bounded by $\kappa$. \\
 (v) (localization and gradient estimates)  $\phi$ satisfies in $U$ the hypotheses of the Localization Theorem \ref{main_loc} at all points  on $\p U\cap B_{c}$ and
$$\abs{\nabla \phi}\leq C_{0}~\text{in}~U\cap B_{c}.$$
(vi) (Maximal sections around the origin)  If $y\in U\cap B_{c^2}$ then the maximal interior section of $\phi$ in $U$ satisfies:$$c\geq \bar{h}(y)\geq k_0^2\, \dist^2 (y, \p U) \quad\mbox{and}\quad  
S_{\phi_h}(y, \bar{h}(y))\subset U\cap B_c.$$
 (vii) (Pogorelov estimates)
 $$\|\p U\cap B_c\|_{C^{2,\alpha}}\leq c_{0}^{-1}$$
 and if $w$ is  the convex solution to 
\begin{equation}
\label{1_MA}
\left\{\begin{array}{rl}
\det D^2 w &=1 \qquad \mbox{ in}\quad U\\
w  &=\phi \ \ \ \ \ \ \ \mbox{ on}\quad \partial U,
\end{array}\right.
\end{equation}
then
\begin{equation*}
\|w\|_{C^{2,\alpha}(\overline{B_{c}\cap U})}\leq c_0^{-1}
\quad \mbox{and}\quad 
c_{0}I_{n}\leq D^2 w\leq c_{0}^{-1} I_{n}~\text{ in }~ 
B_{c}\cap U.
\end{equation*}
\end{myindentpar}
The constants $k, k_0, c, C_0$ above depend only on $n, \rho, \lambda, \Lambda$ and $c_0$ depends also on $\alpha$.
\begin{remark}
\label{quad-rk}
 If $(\Omega, \phi, U)\in \mathcal{P}_{\lambda, \Lambda, \rho, \kappa, \ast}$ then the Pogorelov estimates in $(vii)$ might not hold. However, $\phi$ satisfies in $U$ the hypotheses
 of the Localization Theorem \ref{main_loc} at all points on $\p U\cap B_{c}$. Thus, if $w$ is the convex solution to \eqref{1_MA}, then by inspecting
 the proof of Lemma \ref{lm:quadratic-separation-U}, we see that $w$ separates quadratically from its tangent planes at any point $x_0\in \p U\cap B_c$, that is,
 $$\theta \, |x - x_0|^2 \leq w(x) - w(x_0)-\nabla w(x_0)\cdot (x-x_0)  \leq \theta^{-1} \, | x- x_0|^2\quad \mbox{for all}\quad x\in \partial U.$$
\end{remark}

We summarize the discussion at the end of Subsection~\ref{rescale-sec}, Lemma~\ref{sep-lem}, Lemma~\ref{small-M} and Proposition~\ref{D^2w-est}
in the following proposition.

\begin{proposition}\label{P-prop}
Let $\Omega$ and $\phi$  satisfy the hypotheses of 
Theorem~\ref{main_loc} at the 
origin. Assume in addition that $\p\Omega\cap B_{\rho}$ is $ C^{2,\alpha}$ and $\phi\in C^{2,\alpha}(\p\Omega\cap B_{\rho})$ for some $\alpha\in (0,1)$. 
Then there exists $h_0>0$ depending only on  $n, \lambda, \Lambda, \rho, \alpha,$ $ \|\p\Omega\cap B_{\rho}\|_{C^{2,\alpha}}$ and $ \| \phi\|_{C^{2,\alpha}(\p\Omega\cap B_{\rho})}$ such that for  $h\leq h_0$ we have
\[
\big(\Omega_{h}, \phi_{h}, S_{\phi_{h}}(0, 1)\big)\in \mathcal{P}_{\lambda,\Lambda, \rho, C h^{1/2}, \alpha}\quad\text{and}\quad\|\p\Omega_{h}\cap B_{1/k}\|_{C^{2,\alpha}}\leq C'\,  h^{1/2}.
\]
Here $C$ depends only on $n, \lambda,\Lambda$ and $\rho$; $C'$ depends only on $n, \lambda,\Lambda, \rho,$ 
$ \|\p\Omega\cap B_{\rho}\|_{C^{2,\alpha}},$ and $ \| \phi\|_{C^{2,\alpha}(\p\Omega\cap B_{\rho})}$.
\end{proposition}

\subsection{Geometric properties of 
boundary sections of solutions to Monge--Amp\`ere equation}
In this subsection, we recall some important properties of boundary 
sections of solutions to the Monge-Amp\`ere equations established in \cite{LN}: the engulfing and dichotomic  properties, volume 
estimates, a covering theorem and
strong type $p-p$ estimates for the maximal functions corresponding to small sections including boundary ones. 

The engulfing property and volume estimates are summarized in the following theorem. 
\begin{theorem}
\label{engulfing2}
Assume that $\Omega$ and $\phi$ 
satisfy \eqref{global-tang-int}--\eqref{global-sep}. Then,
\begin{myindentpar}{0.5cm}
a. (Engulfing property) There exists $\theta_{\ast}>0$ depending only 
on $\rho, \lambda, \Lambda$ and $n$ such that if $y\in S_{\phi}(x, t)$ with $x\in\overline{\Omega}$ and $t>0$, then $S_{\phi}(x, t)\subset S_{\phi}(y, \theta_{\ast}t).$\\
b. (Volume estimates) There exist constants $c_{\ast}, C_{1}, C_{2}$ depending only on $\rho, \lambda, \Lambda$ and $n$ 
such that for any section $S_{\phi}(x, t)$ with $x\in\overline{\Omega}$ and $t\leq c_{\ast}$, we have
\begin{equation*}
\label{big_sec_vol}
C_{1} t^{n/2}\leq |S_{\phi}(x, t)|\leq C_{2} t^{n/2}.
\end{equation*}
\end{myindentpar}
\end{theorem}

Our next property is a dichotomy for sections of solutions to the Monge-Amp\`ere equations: any section  is either an interior section or 
included in a boundary section with a comparable height.
\begin{proposition} (Dichotomy)
\label{dicho}
Assume that $\Omega$ and $\phi$ satisfy \eqref{global-tang-int}--\eqref{global-sep}. 
Let $S_{\phi}(x, t)$ be a section of $\phi$ with $x\in\overline{\Omega}$ and $t>0$. Then  one of the following is true:
\begin{myindentpar}{1cm}
(i) $S_{\phi}(x, 2t)$ is an interior section, that is, $S_{\phi}(x, 2t)\subset \Omega$;\\
(ii) $S_{\phi}(x, 2 t)$ is included in a boundary section with comparable height, that is, 
there exists $z\in\partial\Omega$ such that
 $ S_{\phi}(x, 2t)\subset S_{\phi}(z, \bar{c}t)$ for some constant $\bar{c}=\bar{c}(\rho,\lambda, \Lambda, n)>0.$
\end{myindentpar}
\end{proposition}
Our covering theorem states as follows.
\begin{theorem}(Covering theorem)\label{thm:covering}
Assume  $\Omega$ and $\phi$ satisfy \eqref{global-tang-int}--\eqref{global-sep}.
Let $\calO\subset \overline{\Omega}$ be a Lebesgue measurable set and $\e>0$ small. Suppose that for each $x\in \calO$ a section $S_{\phi}(x,t_x)$ is 
given with 
\[
\frac{|S_{\phi}(x,t_x) \cap \calO|}{|S_{\phi}(x,t_x)|} = \e.
\]
Then if $\sup\{t_{x}: x\in\mathcal{O}\}<\infty$, there exists a countable subfamily of 
sections $\{S_{\phi}(x_k,t_k)\}_{k=1}^\infty$ satisfying
$$
\calO \subset \bigcup_{k=1}^\infty{S_{\phi}(x_k, t_k)}\quad \text{and}\quad |\calO|\leq \sqrt{\e} \, \big| \bigcup_{k=1}^\infty{S_{\phi}(x_k, t_k)}\big|.$$
\end{theorem}
Finally, we have the following global strong-type $p-p$ estimates for the maximal function corresponding to small sections. 
\begin{theorem}(Strong-type p-p estimates)\label{strongtype}
Assume that $\Omega$ and $\phi$ satisfy \eqref{global-tang-int}--\eqref{global-sep}.
For $f\in L^1(\Omega)$,  define
\begin{equation*}\label{maximalfunction}
\mathcal M (f)(x)=\sup_{t\leq c}
\dfrac{1}{|S_\phi(x,t)|}\int_{S_\phi(x,t)}|f(y)|\, dy\quad \forall x\in \Omega.
\end{equation*}
Then, for any $1<p<\infty$, there exists  $C_p>0$ depending on $p$, $\rho$, $\lambda$, $\Lambda$ and $n$ such that $$\|\mathcal{M}(f)\|_{L^{p}(\Omega)}\leq C_p\, \|f\|_{L^{p}(\Omega)}.$$
\end{theorem}

\section{\bf Global Power Decay and $W^{2,\delta}$ Estimates}
\label{small_decay_sec}
 In this section, we establish preliminary power decay estimates for the distribution
function of the second derivatives of solutions to the linearized Monge-Amp\`ere equations and also their global $W^{2,\delta}$ estimates. 
We also show under suitable geometric conditions,
the cofactor matrices of the Hessian matrices of two convex 
functions defined on the same domain are close if their Monge-Amp\`ere measures and boundary 
values are close in the $L^{\infty}$ norm.

We begin this section by recalling   the definitions, introduced in \cite{GN2}, of the quasi distance $d(x,x_0)$ generated by a convex function $\phi$ and 
the set $G_M(u, \Omega)$ where
the function $u$ is touched from above and below by ``quasi paraboloids'' generated by this quasi distance.

\begin{definition}\label{def:distance}
Let $\Omega$ be a bounded convex set in $\R^n$ and let $\phi\in
C^1(\Omega)$ be a convex function.
For any  $x\in \Omega$ and $x_0\in \Omega$, we define the quasi distance $d(x, x_0)$ by
\begin{equation*}
d(x,x_0)^2 := \phi(x) - \phi(x_0)
-\nabla \phi(x_0)\cdot (x-x_0).
\end{equation*}
\end{definition}
\begin{definition}\label{def:GMs}
Let $\Omega$ and $\phi$ be as in Definition~\ref{def:distance}.
For  $u\in C(\Omega)$ and $M>0$, we define 
\begin{align*}
G_M(u,\Omega)
= \big\{\bar x\in \Omega: \text{$u$ is differentiable at $\bar x$ and } |u(x)-u(\bar x)-
\nabla u(\bar x)\cdot (x-\bar x)|  \leq
\frac{M}{2} d(x,\bar x)^2\,\forall x\in \Omega\big\}.
\end{align*}
\end{definition}
We call $\frac{M}{2}d(x,\bar{x})^2$ and $-\frac{M}{2}d(x,\bar{x})^2$ quasi paraboloids of opening $M$ generated by $\phi$. 
When we would like to emphasize the dependence of $d(x, x_0)$ on $\phi$, we write $d_{\phi}(x, x_0).$ 
Likewise, we write $G_{M}(u,\Omega, \phi)$ to indicate the dependence on $\phi$ of the set $G_{M}(u,\Omega)$. 
Notice that for $\phi(x) =\abs{x}^2$, we have $d(x, \bar{x})=\abs{x-\bar{x}}$
is the Euclidean distance. 

In the next lemma, we show that if the quasi distance $d(x, x_0)$ is bounded from below by the Euclidean distance $\abs{x-x_0}$ around $x_0$ then it is
also bounded from above by a multiple of this Euclidean distance around $x_0$. This lemma is a slight modification of  \cite[Lemma~6.2.1]{G}.
\begin{lemma}\label{lm:belowimplyabove}
Assume  $\Omega$ satisfies \eqref{global-tang-int} and let $\phi\in C(\overline \Omega)$ be a convex function satisfying $\lambda \leq \det  D^2\phi\leq \Lambda$ in $\Omega$ and $\phi=0$ on $\partial\Omega$.
There exists $c=c(n,\lambda,\Lambda,\rho)>0$  such that
if $x_0\in \Omega$ and
$$
d(x, x_0)^2 \geq \sigma\, |x-x_0|^2~ \text{ in }  ~ B_r(x_0)\subset\Omega\quad \text{for some}\quad r>0,$$
then for all $x$ in a small neighborhood of $x_0$, we have 
\[
d(x,x_0)^2 \leq \frac{1}{c^2 \sigma^{n-1}}\, |x-x_0|^2.
\]

\end{lemma}
\begin{proof}
Let $\varphi(x) := \phi(x)-
\phi(x_0)- \nabla \phi(x_0)\cdot (x-x_0)$. Then the strict convexity of $\phi$ implies that there exists $\delta>0$ such that $S_\phi(x_{0}, \delta) := \{x\in \Omega: \varphi(x) <\delta\}\subset B_r(x_0)$. Therefore
by the proof of Lemma~6.2.1 in \cite{G},  we have $\varphi(x)\leq C(n,\lambda,\Lambda,\rho)\sigma^{-n+1}\, |x-x_0|^2$ for all $x\in\Omega$ satisfying $\varphi(x)\leq \delta$, which gives the conclusion of the lemma.
\end{proof}

The following lemma allows us to estimate the distribution function of $D^2u$.
It is the starting point for our proofs of Theorems~\ref{main-result-constant} and \ref{main-result} and the global version of \cite[Lemma~2.7]{GN2}.
\begin{lemma}\label{distribution-function}
Let $\Omega$, $\phi$ and $c$ be as in Lemma~\ref{lm:belowimplyabove}, and $u\in C^2(\Omega)$. Define
\[
A^{loc}_\sigma := \left\{x_0\in \Omega: d(x, x_0)^2\geq \sigma\, |x-x_0|^2, \mbox{ for all $x$ in some neighborhood of $x_0$}\right\}.
\]
 Then for any $m>1$ and $\beta>0$, we have
\begin{align}\label{eq:distributionsetofsecondderivatives}
\{x\in \Omega: |D_{i j} u(x)| > \beta^m\} 
\subset \big(\Omega \setminus \lA_{(c\beta^{\frac{m -1}{2}})^{\frac{-2}{n-1}}}\big) \cup \big(\Omega
\setminus G_\beta(u,\Omega)\big).
\end{align}
\end{lemma}
\begin{proof}
Let $\gamma := \beta^{\frac{m -1}{2}}$.
If $\bar x\in \lA_{(c\gamma)^{\frac{-2}{n-1}}} \cap G_\beta(u,\Omega)$, then
\[
-\frac{\beta}{2} \,d(x,\bar x)^2 \leq u(x)-u(\bar x) -\nabla u(\bar x)\cdot
(x-\bar x) \leq \frac{\beta}{2} \,d(x,\bar x)^2
\]
for each $x\in \Omega$. Since $\bar x\in
\lA_{(c\gamma)^{\frac{-2}{n-1}}}$, these together with  Lemma~\ref{lm:belowimplyabove} yield
\[
-\frac{\beta \gamma^2}{2} |x-\bar x|^2 \leq u(x)-u(\bar x) -\nabla u(\bar x)\cdot
(x-\bar x) \leq \frac{\beta \gamma^2}{2} |x-\bar x|^2
\]
for all $x$ in a small neighborhood of $\bar x$, and so
$|D_{ij}u(\bar x)|\leq \beta \gamma^2 = \beta^{m}$. Thus we have proved that
\[
 \lA_{(c\gamma)^{\frac{-2}{n-1}}} \cap G_\beta(u,\Omega) \subset  \big\{x\in \Omega:
|D_{ij}u(x)|\leq \beta^m, \text{ for $i,j=1,\dots ,n$}\big\}
\]
and the lemma follows by taking complements.
\end{proof}

\subsection{Power decay estimates} In order to derive global $W^{2,p}$ estimates for solutions $u$  to the linearized Monge-Amp\`ere equation, we will  need to estimate the distribution function     
$$F(\beta) :=\big| \{x\in \Omega: |D_{ij}u(x)| > \beta^m\} \big|$$ for some suitable choice of $m>1$. It follows from Lemma~\ref{distribution-function} that this 
can be done if one can get  appropriate decay estimates for 
$$F_1(\beta) := |\Omega \setminus \lA_{(c\beta^{\frac{m -1}{2}})^{\frac{-2}{n-1}}}|\quad  \text{and}\quad F_2(\beta) := |\Omega
\setminus G_\beta(u,\Omega)|.$$

Notice  that the function $F_1(\beta)$ involves only the solution $\phi$ of the Monge-Amp\`ere equation and its power decay is given in the next  theorem. 

\begin{theorem}\label{thm:powerdecayMA}
Assume  $\Omega$  satisfies  
\eqref{global-tang-int} and $\partial \Omega\in C^{1,1}$. Let $\phi\in C(\overline{\Omega})$ be a convex function such that $1-\epsilon\leq \det D^2\phi\leq 1+\epsilon$ in $\Omega$ and  \eqref{global-sep} holds, where  $0<\epsilon < 1/2$.
Then there exists a positive constant $M$
depending only on  $n$ and $\rho$ such that
\begin{align}\label{eq:exponentialdecay}
\big|\Omega\setminus \lA_{s^{-2}}\big|
\leq C'(\epsilon,n,\rho, \|\partial\Omega\|_{C^{1,1}}) \quad s^{\dfrac{\ln \sqrt{C \epsilon}}{\ln
M}}\quad \mbox{ for all}\quad s>0.
\end{align}
In particular, for $s = (c\beta^{\frac{m-1}{2}})^{\frac{1}{n-1}}$, we get
$$F_{1}(\beta)\leq  C'(\epsilon,n,\rho, \|\partial\Omega\|_{C^{1,1}})
\,  \beta^{-\frac{m-1}{2(n-1) \ln M }\ln\frac{1}{\sqrt{C\e}}}\quad \forall \beta>0.$$
\end{theorem}

The small power decay estimates for $F_2(\beta)$ are given in the following proposition. It is the  boundary version of Proposition~3.4 in \cite{GN2}.  

\begin{proposition}\label{Initial-Estimate} Assume that  $\Omega$ and  $\phi$ satisfy the assumptions \eqref{global-tang-int}--\eqref{global-sep}. Assume in addition that $\p\Omega\in C^{1,1}$.
 Suppose $u\in  C^1(\Omega)\cap W^{2,n}_{loc}(\Omega)$, $|u|\leq 1$ in $\Omega$ and 
$\calL_\phi u=f$ in $\Omega$ with
$\|f\|_{L^n(\Omega)}\leq 1$.
Then there  exist $\tau=\tau(n,\lambda,\Lambda,\rho)\in (0, 1/2)$ and $C= C(\rho, \lambda, \Lambda, n, \|\partial\Omega\|_{C^{1,1}})>0$ such that
\begin{equation*}
F_2(\beta)=\big|\Omega\setminus G_{\beta}(u,\Omega)\big| \leq
\dfrac{C}{\beta^\tau}\quad \mbox{for all}\quad  \beta>0.
\end{equation*}
\end{proposition}

The next result is a variant of Proposition~\ref{Initial-Estimate} which  will be  important for the density and improved power decay estimates in Subsection~\ref{sub:improveddensity}.

\begin{proposition}\label{power-decay-at-boundary}
Let $(\Omega, \phi, U)$ be in the class $\mathcal{P}_{\lambda,\Lambda, \rho, \kappa,\ast}$.
Suppose $u\in  C(\Omega)\cap C^1(U)\cap W^{2,n}_{loc}(U)$, $|u|\leq 1$ in $\Omega$ and 
$\calL_\phi u=f$ in $U$ with
$\|f\|_{L^n(U\cap B_c)}\leq 1$.
Then there exist $\tau=\tau(n,\lambda,\Lambda,\rho)\in (0, 1/2)$ and $C= C(n,\lambda,\Lambda,\rho, \kappa)>0$ 
such that
\begin{equation*}
\big|(U\cap B_{c^2})\setminus G_{\beta}(u,\Omega)\big| \leq
\dfrac{C}{\beta^\tau} \, |U\cap B_{c^2}|\quad \mbox{for all}\quad  \beta>0.
\end{equation*}
The above inequality also holds if  $U\cap B_{c^2}$ is replaced  by  $S_{\phi}(0, r)$ for any universal constant $r$ satisfying $r\leq c^6$.
\end{proposition}

As a consequence of the power decay estimates for $F_1(\beta)$ and $F_2(\beta)$ in Theorem \ref{thm:powerdecayMA} and Proposition \ref{Initial-Estimate}, we find that the decay for $F(\beta)$ when $0<\e< 1/2$  is given by
$$F(\beta)\leq C(\epsilon,n,\rho, \|\partial\Omega\|_{C^{1,1}}) \, \beta^{-\frac{m-1}{2(n-1)\ln M }\ln\frac{1}{\sqrt{C\e}}} + C\beta^{-\tau}.$$
Since $\frac{m-1}{2(n-1)\ln M }\ln\frac{1}{\sqrt{C\e}} \rightarrow \infty$ as $\e\rightarrow 0$, we obtain global $W^{2,\delta}$ estimates for all $\delta<\tau/m<1/2$ for solutions to
the linearized Monge-Amp\`ere equation $\mathcal{L}_{\phi} u = f$ provided that $f\in L^{n}(\Omega)$ and $\e$ is small, that is, 
$\det D^2\phi$ is close to a constant.
However, in the next subsection, we offer a more direct proof of  
global $W^{2,\delta}$ estimates based on interior estimates without resorting to decay estimates of the distribution function of the
second derivatives. Another advantage of this proof is that it works for all Monge-Amp\`ere measures $\det D^2\phi$ bounded away from $0$ and $\infty$.
\begin{remark}
 It is now clear that the obstruction to higher integrability of $|D^2 u|$ is the small exponent $\tau$ in the decay estimates for $\big|\Omega\setminus G_{\beta}(u,\Omega)\big|$ given by Proposition~\ref{Initial-Estimate}.
 Most of the paper is devoted to developing tools to improve the decay estimates for $\big|\Omega\setminus G_{\beta}(u,\Omega)\big|$. In particular, the 
 global stability
 of cofactor matrices and an approximation lemma in the next two sections will be employed for this purpose.
\end{remark}

\subsection{Global $W^{2,\delta}$ estimates}
In this subsection, we obtain global  $W^{2,\delta}(\Omega)$ estimates ($\delta>0$ small) for solutions to the linearized 
Monge-Amp\`ere equation $\calL_{\phi}u=f$ when $\det D^{2}\phi$ is only bounded away from $0$ and $\infty$
and under natural assumptions on the domain $\Omega$ and the boundary data of $\phi$. 

Our main theorem in this subsection is the following.
\begin{theorem}\label{small-d2}
Assume  $\Omega$ and  $\phi$ satisfy the assumptions 
\eqref{global-tang-int}--\eqref{global-sep}. Assume in addition that $\partial \Omega\in C^{1,1}$. 
Let $u \in C(\overline{\Omega})\cap C^1(\Omega)\cap W^{2, n}_{loc}(\Omega)$  be a solution of 
\begin{equation*}
\left\{\begin{array}{rl}
\calL_\phi u &=f \qquad \mbox{ in}\quad \Omega,\\
u  &=0\ \ \ \ \ \ \mbox{ on}\quad \partial \Omega.
\end{array}\right.
\end{equation*}
Then there exist $p=p(\rho,\lambda,\Lambda,n)>0$ and $C= C(\rho, \lambda, \Lambda, n, \|\partial\Omega\|_{C^{1,1}})>0$ such that
\begin{equation*}
\|D^2 u\|_{L^{p}(\Omega)}\leq
C\|f\|_{L^{n}(\Omega)}.
\end{equation*}
\end{theorem}
The rest of this subsection is devoted to proving this theorem. The idea is 
to cover $\Omega$ by maximal interior sections whose shapes are under 
control by Proposition~\ref{tan_sec} and then apply the 
interior $W^{2,\delta}$ estimates of Guti\'errez and Tournier \cite{GT} in these sections. 
Furthermore, since we can control the number of these sections within certain height due to the $C^{1,1}$ regularity of the boundary $\p\Omega$, the global estimates follow by adding
interior ones.\\

For reader's convenience, we recall Guti\'errez-Tournier's $W^{2,\delta}$ estimates.
\begin{theorem}(\cite[Theorem 6.3]{GT})
\label{GT_loc}
 Let $\Omega$ be a convex domain such that 
$B_{k_0}\subset \Omega\subset B_{k_{0}^{-1}}$.
Let $\phi\in C^{2}(\Omega)$ be a convex function satisfying
$\lambda\leq \det D^{2}\phi \leq \Lambda~\text{ in }~\Omega$ and $\phi =0~\text{ on }~\p\Omega.$
Let $u\in C^{1}(\Omega)\cap W^{2,n}_{loc}(\Omega)$ be a solution of $\mathcal{L}_{\phi} u=f$ in $\Omega$. 
Then, given $\alpha_{0}\in (0,1)$, there exist
positive constants $\delta$ and $C$ depending only on $\alpha_{0},  k_0, \lambda, \Lambda$ and $n$ such that
$$\|D^2 u\|_{L^{\delta}\big(S_{\phi}(x_0, -\alpha_0\phi(x_0))\big)}\leq C \Big(\|u\|_{L^{\infty}(\Omega)} + \|f\|_{L^{n}(\Omega)}\Big),$$
where $x_0\in \Omega$ is such that $\min_{\Omega}\phi =\phi(x_0)$.
\end{theorem}

Let $0<p< \min\{\delta, \frac{1}{2}\}$ where $\delta=\delta(\rho,\lambda, \Lambda, n)>0$ is a small number 
appearing in Theorem \ref{GT_loc}  corresponding to  $\alpha_{0} =1/2$ and $k_0=k_0(\rho, n, \lambda, \Lambda)$ given by Proposition \ref{tan_sec}. \\

We will show that the conclusion of Theorem~\ref{small-d2} holds for the above choice of $p$.  To achieve this, we first estimate the $L^{p}$ norm of $D^{2} u$ in the interior
of each maximal interior section.
\begin{lemma}
\label{w2p-small-local}
Assume  $\Omega$ and  $\phi$ satisfy the assumptions 
\eqref{global-tang-int}--\eqref{global-sep}. 
Let $u \in C(\overline{\Omega})\cap C^1(\Omega)\cap W^{2, n}_{loc}(\Omega)$  be a solution of 
\begin{equation*}
\calL_\phi u =f ~ \mbox{ in }~\Omega,\quad \mbox{and}\quad
u  =0 ~\mbox{ on }~ \partial \Omega.
\end{equation*}
Then, there exists a constant $C>0$ depending only on $p, \rho, \lambda, \Lambda$ and $n$ such that
\begin{eqnarray*}
\|D^{2} u\|_{L^{p}\big(S_{\phi}(y, \frac{\bar{h}(y)}{2})\big)}\leq C \bh(y)^{\frac{n}{2p}-1}|\log \bar{h}(y)|^{2} \, \Big(\|u\|_{L^{\infty}\big(S_{\phi}(y, \bar{h}(y))\big)} + 
\bar{h}(y)^{1/2}\|f\|_{L^{n}\big(S_{\phi}(y, \bar{h}(y))\big)} \Big)
\end{eqnarray*}
for all  $y\in\Omega$ satisfying $\bar{h}(y)\leq c$.
\end{lemma}
\begin{proof}
Let $h:= \bar{h}(y)$ with $\bar{h}(y)\leq c.$ 
We now define the rescaled domain $\tilde \Omega_{h}$ and rescaled functions $\tilde \phi_{h}$,
$\tilde u_{h}$ and $\tilde f_h$ as in Subsection~\ref{rescale-sec} that {\it preserve
the $L^{\infty}$-norm in a section tangent to the boundary}.
For simplicity, let us  denote $\tilde{S}_{t}(0) := S_{\tilde{\phi}_{h}} (0, t)$ for $t>0$.
Then by Theorem~\ref{GT_loc}, we have
\begin{equation}
\|D^{2}\tilde{u}_{h}\|_{L^{p}(\tilde{S}_{\frac{1}{2}}(0))}\leq C(p,\rho,\lambda, \Lambda, n) \, \Big(\|\tilde{u}_{h}\|_{L^{\infty}(\tilde{S}_{1}(0))} + 
\|\tilde{f}_{h}\|_{L^{n}(\tilde{S}_{1}(0))}\Big).
\label{res-w2p}
\end{equation}
Using the fact
$$D^{2} u( y + h^{1/2} A^{-1}_{h} x) = h^{-1} (A_{h})^{t} \, D^{2} \tilde{u}_{h}(x) \, A_{h},$$
we obtain
\begin{eqnarray*}
\int_{S_{\phi}(y,\frac{h}{2})}|D^{2} u(z)|^{p} \, dz= h^{\frac{n}{2}-p} 
\int_{\tilde{S}_{\frac{1}{2}}(0)} |A_{h}^{t} \, D^{2} \tilde{u}_{h}(x) \, A_{h}|^{p}\, dx
\leq  C \,h^{\frac{n}{2}-p} \, |\log h|^{2p} \int_{\tilde{S}_{\frac{1}{2}}(0)} |D^{2} \tilde{u}_{h}(x)|^{p}\, dx.
\end{eqnarray*}
It follows that
\begin{equation}
\|D^{2} u\|_{L^{p}\big(S_{\phi}(y, \frac{h}{2})\big)}\leq C h^{\frac{n}{2p} -1}|\log h|^{2}
\, \|D^{2}\tilde{u}_{h}\|_{L^{p}\big(\tilde{S}_{\frac{1}{2}}(0)\big)}.
\label{res-d2u}
\end{equation}
Moreover, we have
\begin{equation}
\|\tilde{f}_{h}\|_{L^{n}(\tilde{S}_{1}(0))} = h^{\frac{1}{2}}\|f\|_{L^{n}(S_{\phi}(y, h))}\quad\mbox{and}\quad
\|\tilde{u}_{h}\|_{L^{\infty}(\tilde{S}_{1}(0))} = \, \|u\|_{L^{\infty}(S_{\phi}(y, h))}.
\label{res-u-infty}
\end{equation}
Combining \eqref{res-w2p}--\eqref{res-u-infty}, we obtain the desired estimate stated in our lemma.
\end{proof}
Finally, we will use the following Vitali covering lemma proved by Savin in \cite{S3}; see also \cite[Lemma~2.5]{LN} for a more general covering result.

\begin{lemma}(\cite[Lemma 2.3]{S3})\label{Vitali-MA}
Assume  $\Omega$ and  $\phi$ satisfy the assumptions 
\eqref{global-tang-int}--\eqref{global-sep}.
Then there exists a sequence of disjoint sections $S_{\phi}(y_{i}, \delta_{0}\bar{h}(y_i))$  with $\delta_0= \delta_0(\lambda, \Lambda, n)>0$ such that
\begin{equation*}\Omega\subset \bigcup_{i=1}^{\infty} S_{\phi}(y_{i}, \frac{\bh(y_i)}{2}).
\end{equation*}
\end{lemma}
\begin{proof}[Proof of Theorem \ref{small-d2}]
It follows from  Proposition~\ref{tan_sec} (see also \cite[Lemma 2.2]{S3}) that 
if  $y\in \Omega$ with $\bar{h}(y)\leq c$  then
$$S_{\phi}(y, \bar{h}(y))\subset y + k_{0}^{-1}E_h\subset D_{C\bar{h}(y)^{1/2}}: =\{x\in \overline{\Omega}: \, 
\dist(x, \partial\Omega)\leq C\bar{h}(y)^{1/2}\},\quad C:= 2k_{0}^{-2}.$$
By Lemma \ref{Vitali-MA}, we have
$$\int_{\Omega} |D^{2} u|^{p} dx \leq \sum_{i=1}^\infty\int_{S_{\phi}(y_{i}, \frac{\bh(y_i)}{2})} |D^2 u|^{p} dx.$$
There is a finite number of sections $S_{\phi}(y_i, \bar{h}(y_i))$ with $\bar{h}(y_i)\geq c$ and, by Theorem~\ref{GT_loc}, we have in each such section
$$ \int_{S_{\phi}(y_{i}, \frac{\bh(y_i)}{2})} |D^{2} u|^{p} \leq C \, \big(\|u\|_{L^{\infty}(\Omega)} + \|f\|_{L^{n}(\Omega)}\big)^{p}.$$

Now, for  $d\leq c$ we consider the family $\mathcal{F}_{d}$ of sections $S_{\phi}(y_{i}, \bar{h}(y_i)/2)$ such that $d/2<\bar{h}(y_i)\leq d$. Let $M_{d}$ be the number 
of sections in $\mathcal{F}_{d}$.  We claim that
\begin{equation}\label{Md-est}
 M_{d}\leq C_{b} d^{\frac{1}{2}-\frac{n}{2}}
\end{equation}
for some constant $C_{b}$ depending only on $\rho, n, \lambda, \Lambda$ and $\|\partial\Omega\|_{C^{1,1}}.$
Indeed, we first note that, by \cite[Corollary 3.2.4]{G} (see also Theorem \ref{engulfing2}(b)), there exists a constant 
$C= C(n,\lambda,\Lambda, \rho)>0$ such that
$$|S_{\phi}(y_{i}, \delta_{0}\bar{h}(y_i))| \geq C \bar{h}(y_i)^{n/2} \geq C d ^{n/2}.$$
Since $S_{\phi}(y_{i}, \delta_{0}\bar{h}(y_i))\subset D_{Cd^{1/2}}$ are disjoint, we find that
$$ M_d Cd^{n/2}\leq \sum_{i\in\mathcal{F}_{d}} |S_{\phi}(y_{i},\delta_{0} \bar{h}(y_i))|\leq |D_{Cd^{1/2}}|\leq C_{\ast}d^{1/2}$$
for some constant $C_{\ast}$ depending only on $n$ and $\|\partial\Omega\|_{C^{1,1}}$.
Thus \eqref{Md-est} holds.

It follows from 
Lemma~\ref{w2p-small-local} and \eqref{Md-est} that
\begin{eqnarray*}
\sum_{i\in \mathcal{F}_{d}} \int_{S_{\phi}(y_{i}, \frac{\bh(y_i)}{2})} |D^{2} u|^{p}& \leq& C M_{d}d^{\frac{n}{2} -p}|\log d|^{2p} \, \Big(\|u\|_{L^{\infty}(\Omega)} + \|f\|_{L^{n}(\Omega)}\Big)^{p}\\ &\leq&
C d^{\frac{1}{2} -p}|\log d|^{2p}\, \Big(\|u\|_{L^{\infty}(\Omega)} + \|f\|_{L^{n}(\Omega)}\Big)^{p}. 
\end{eqnarray*}
Adding these inequalities for the sequence $d = c2^{-k}, k= 0, 1, 2, \cdots, $  and noting that
$$\|u\|_{L^{\infty}(\Omega)} \leq C(n, \rho, \lambda, \Lambda) \|f\|_{L^{n}(\Omega)},$$ by the ABP estimate,
we obtain the desired global $L^p$ estimate for $D^2 u$.
\end{proof}

\subsection{Proofs of the power decay estimates}
\begin{proof}[Proof of Theorem \ref{thm:powerdecayMA}]
Let $\{S_{\phi}(y_i, \bar{h}(y_i)/2)\}$ be the sequence of sections covering $\Omega$ given by Lemma~\ref{Vitali-MA}. In what follows we will use the notations as in the proof of Lemma~\ref{w2p-small-local}. We then have 
\begin{eqnarray}\label{eq:covering1}
\big|\Omega\setminus \lA_{s^{-2}}\big|
&\leq &\sum_{i=1}^\infty{\big|S_{\phi}(y_i, \bar{h}(y_i)/2)\setminus \lA_{s^{-2}}\big| }\nonumber\\
&\leq & \sum_{k=0}^\infty \sum_{i\in \calF_{c2^{-k}}}{\big|S_{\phi}(y_i, \bar{h}(y_i)/2)\setminus \lA_{s^{-2}}\big| } + \sum_{i: \, \bh({y_i})> c}{\big|S_{\phi}(y_i, \bar{h}(y_i)/2)\setminus 
\lA_{s^{-2}}\big| }=: I + II.
\end{eqnarray}
Let us first estimate the summation $I$ corresponding to sections with $\bh(y_i)\leq c$. Consider a prototype section $S_{\phi}(y, h)$ with $h:=\bar{h}(y)\leq c$. 
Proposition \ref{tan_sec} 
tells us that $S_{\phi}(y, h)$ is equivalent to an ellipsoid $E_h$, i.e.,
$$k_0 E_h \subset S_{\phi}(y, h)-y \subset k_{0}^{-1} E_h,$$
where
$$E_h := h^{1/2}A_{h}^{-1}B_1, \quad \mbox{with} \quad \det A_{h}=1,\quad  \|A_{h}\|, \, \|A_{h}^{-1} \| \le k^{-1} |\log  h|.$$
Here $k, k_{0}$ depend only on $n$ and $\rho$.
Let $T(x) := h^{-1/2} A_h(x-y)$.  Define $\tilde U_h:=T(S_{\phi}(y, h))$ and
\[\tilde \phi_h(z) := h^{-1} \Big[\phi(T^{-1}z) -\phi(y) -\nabla\phi(y)\cdot (T^{-1}z -y)-h\Big]
\quad \mbox{for}\quad z\in \tilde U_h.
\]
Then $B_{k_0} \subset \tilde U_h\equiv S_{\tilde\phi_h}(0,1)\subset B_{k_0^{-1}} $, $1-\eps \leq \det D^2\tilde\phi_h \leq 1+\eps$ in $\tilde U_h$ and 
$\tilde\phi_h =0$ on $\partial \tilde U_h$.
By \cite[Theorem 3.3.10]{G}, there exists $\eta_{0}= \eta_{0}(n,\rho)>0$ such that
$$S_{\tilde\phi_h}(x, t)\Subset \tilde U_h~\text{ for all}~ x\in S_{\tilde\phi_h}(0, 1/2)~\text{and}~ t\leq \eta_{0}.$$
Now, let
$$\tilde{D}_{s}^{\frac{1}{2}} :=\{x\in S_{\tilde\phi_h}(0, 1/2): S_{\tilde\phi_h}(x, t)\subset B(x, s\sqrt{t}), ~\forall t\leq \eta_{0}\}.$$
Then, by \cite[Theorem~2.8]{GN2}, we obtain
\begin{equation*}
|S_{\tilde\phi_h}(0, 1/2)\setminus \tilde D^{\frac12}_{s}|
\leq \frac{|\tilde U_h|}{(C\epsilon)^2}\, s^{-p_\eps},
\end{equation*}
where $p_\eps := -\dfrac{\ln \sqrt{C \epsilon}}{\ln
M}$ with $C, M>0$ is a constant depending only on $n$ and $\rho$. Let 
\[
\tilde A_\sigma := \left\{\bar z\in \tilde U_h: \tilde\phi_h(z)\geq
\tilde\phi_h(\bar z) + \nabla \tilde\phi_h(\bar z)\cdot (z-\bar z) + \sigma\, |z-\bar z|^2, \quad \forall z\in \tilde U_h\right\}.
\]
Since $\tilde D^{\frac12}_{s} = S_{\tilde\phi_h}(0, 1/2) \cap \tilde A_{s^{-2}}$ by \cite[Theorem 6.2.2]{G}, we can rewrite the above inequality as 
\begin{equation}\label{interior-est}
|T\big(S_{\phi}(y, h/2)\big)\setminus \tilde A_{s^{-2}}|
\leq C(\eps,n,\rho) \, s^{-p_\eps}.
\end{equation}

Let us relate $\tilde A_{s^{-2}}$ to $\lA_\sigma$. Since $|x-\bar x|\leq \|A_h^{-1}\|\, | A_h (x-\bar x)|\leq k^{-1} h^{1/2} |\log h|\,  |Tx - T\bar x|$, 
we have
\begin{align*} 
\tilde A_{s^{-2}} 
&= T \left\{\bar x\in S_\phi(y, h): \tilde\phi_h(Tx)\geq
\tilde\phi_h(T\bar x) + \nabla \tilde\phi_h(T\bar x)\cdot (T x-T\bar x) + s^{-2}\, |T x-T\bar x|^2, \, \forall x\in S_\phi(y, h)\right\}\\
&\subset T \left\{\bar x\in S_\phi(y, h): \phi(x)\geq
\phi(\bar x) + \nabla \phi(\bar x)\cdot ( x-\bar x) + (k^{-1} s |\log h|)^{-2}\, | x-\bar x|^2, \, \forall x\in S_\phi(y, h)\right\}\\
&\subset T\big(S_\phi(y, h) \cap \lA_{(k^{-1} s |\log h|)^{-2}}\big). 
\end{align*}
We infer from this and  \eqref{interior-est} that
\begin{equation*}
|S_\phi(y, h/2)\setminus   \lA_{(k^{-1} s |\log h|)^{-2}}|
\leq C(\eps,n,\rho)  |\det T|^{-1} \, s^{-p_\eps}= C(\eps,n,\rho) \, h^{n/2} \, s^{-p_\eps}\quad \forall s>0,
\end{equation*}
or equivalently,
\begin{equation*}
|S_\phi(y, h/2)\setminus   \lA_{s^{-2}}|
\leq C(\eps,n,\rho) \, h^{n/2} |\log h|^{p_\eps} \, s^{-p_\eps}\quad \forall s>0.
\end{equation*}
Thus the summation  $I$ in  \eqref{eq:covering1} can be estimated as follows
\begin{align}\label{est-item-I}
I
\leq  C(\eps, n, \rho) s^{-p_\eps} \sum_{k=0}^\infty \sum_{i\in \calF_{c 2^{-k}}}{\bar{h}(y_i)^{n/2} |\log \bar{h}(y_i)|^{p_\eps} }
&\leq  C(\eps,n,\rho)  s^{-p_\eps} \sum_{k=0}^\infty {(c 2^{-k})^{n/2} |\log (c 2^{-k-1})|^{p_\eps} M_{c 2^{-k}}}\nonumber\\
&\leq  C \, s^{-p_\eps} \sum_{k=0}^\infty {(c 2^{-k})^{1/2} |\log (c 2^{-k-1})|^{p_\eps} }
\leq  C \, s^{-p_\eps}.
\end{align}
Note that $C$ depends on $\e, n, \rho$ and $\|\partial\Omega\|_{C^{1,1}}$, and we have used the bound \eqref{Md-est} for 
$M_d$ to obtain the third inequality. 

Next let us estimate the summation $II$ corresponding to  sections $S_{\phi}(y_i, \bar{h}(y_i)/2)$ with $\bar h (y_i) >c$. Since 
the family $\{S_\phi(y_i, \delta_0 \bh(y_i))\}$ is disjoint, we infer from the lower bound on volume of sections and $\Omega\subset B_{1/\rho}$ that 
\[
\#\{i: \bh(y_i)> c\} \leq C(n,\rho).
\]
Also, by using the standard normalization for interior sections and  \cite[Theorem~2.8]{GN2} we get
\begin{equation*}
|S_\phi(y_i, \bar{h}(y_i)/2)\setminus   \lA_{s^{-2}}|
\leq C(\eps,n,\rho)  \, s^{-p_\eps} \quad \mbox{for all $i$ with}\quad \bh(y_i)> c.
\end{equation*}
Therefore,
\begin{equation}\label{est-item-II}
II \leq  \#\{i: \bh(y_i)> c\} \, \big[ C(\eps,n,\rho)  \, s^{-p_\eps} \big]\leq  C(\eps,n,\rho)  \, s^{-p_\eps} \quad \forall s>0.
\end{equation}
By combining  \eqref{eq:covering1}, \eqref{est-item-I} and \eqref{est-item-II} we obtain
\begin{align*}
\big|\Omega\setminus \lA_{s^{-2}}\big|
&\leq  I + II
\leq  C(\eps,n,\rho, \|\partial\Omega\|_{C^{1,1}}) \, s^{-p_\eps}=  C(\epsilon,n,\rho, \|\partial\Omega\|_{C^{1,1}}) \quad s^{\dfrac{\ln \sqrt{C \epsilon}}{\ln
M}}.
\end{align*}
\end{proof}
The proof of Theorem~\ref{thm:powerdecayMA}  can also be employed to give the proof of Proposition \ref{Initial-Estimate}.
\begin{proof}[Proof of Proposition \ref{Initial-Estimate}]
Let $\{S_{\phi}(y_{i},\bar{h}(y_i)/2)\}$ be the sequence of sections covering $\Omega$ given by Lemma~\ref{Vitali-MA}. Then we have
\begin{multline}\label{eq:covering2} 
\big|\Omega\setminus G_{\beta}(u,\Omega)\big| 
\leq \sum_{i:\, \bh(y_i) >c}\big|S_{\phi}(y_{i},\frac{\bh(y_i)}{2})\setminus G_{\beta}(u,\Omega)\big|
+ \sum_{k=0}^\infty \sum_{i\in \calF_{c 2^{-k}}}{\big|S_{\phi}(y_{i},\frac{\bh(y_i)}{2})\setminus G_{\beta}(u,\Omega)\big| }.
\end{multline}
By using \cite[Proposition 3.4]{GN2} and arguing as in estimating the term $II$ in the proof of Theorem~\ref{thm:powerdecayMA}, we see that there exist constants $C,\tau>0$ 
depending only on $n, \lambda, \Lambda$ and $\rho$ with $\tau<1/2$ such that
\begin{equation}\label{term1-small-decay}
\sum_{i:\, \bh(y_i) >c}\big|S_{\phi}(y_{i},\frac{\bh(y_i)}{2})\setminus G_{\beta}(u,\Omega)\big|
\leq \sum_{i:\, \bh(y_i) >c}{\frac{C}{\beta^\tau}} 
= \frac{C}{\beta^\tau}  \, \#\{i: \bh(y_i)> c\} 
 \leq \frac{C}{\beta^\tau}.
\end{equation}
To estimate the last expression in \eqref{eq:covering2}, let us  consider a prototype section $S_{\phi}(y, h)$ with $h:=\bar{h}(y)\leq c$. 
We now define the rescaled domains $\tilde \Omega_{h}, \tilde{U}_h$ and rescaled functions $\tilde \phi_{h}$,
$\tilde u_{h}$ and $\tilde f_h$ as in Subsection~\ref{rescale-sec} that {\it preserve
the $L^{\infty}$-norm in a section tangent to the boundary}.
Then
 \begin{equation}
 \|\tilde f_h\|_{L^n(\tilde U_h)}
 =h^{1/2} \|f\|_{L^n(S_{\phi}(y, h))}\leq  h^{1/2} \|f\|_{L^n(\Omega)} \leq 1.
 \label{c_function}
 \end{equation}
 Therefore, we can apply \cite[Proposition 3.4]{GN2} to obtain for $T(x): = h^{-1/2} A_h(x-y)$
 \begin{equation*}
 \big|T(S_{\phi}(y, h/2)) \setminus G_{\beta}(\tilde u_h,\tilde \Omega_h, \tilde\phi_h)\big|=
\big| S_{\tilde\phi_h}(0,1/2) \setminus G_{\beta}(\tilde u_h,\tilde \Omega_h, \tilde\phi_h)\big| \leq
\dfrac{C}{\beta^\tau}\quad \mbox{for all}\quad  \beta>0.
\end{equation*}
But as $\tilde u_h \in C^1(\tilde U_h)$ and 
$d_{\tilde \phi_h}(Tx, T \bar x)^2 =h^{-1} d(x, \bar x)^2$
for all $x,\bar x\in \Omega$, we get
\begin{align*}
T(S_{\phi}(y, h/2)) \cap G_{\beta}(\tilde u_h,\tilde \Omega_h, \tilde\phi_h)=
 T\Big( S_{\phi}(y, h/2)\cap G_{\beta h^{-1}}( u, \Omega) \Big).
\end{align*} 
Thus we infer from the above inequality that
\begin{equation*}
 \big|S_{\phi}(y, h/2) \setminus G_{\beta h^{-1}}( u, \Omega) \big| \leq \dfrac{C}{\beta^\tau} |\det T|^{-1} = \dfrac{C}{\beta^\tau} h^{\frac{n}{2}},
\end{equation*}
or equivalently,
\begin{equation*}
 \big|S_{\phi}(y,h/2) \setminus G_{\beta }( u, \Omega) \big| 
\leq  \dfrac{C}{\beta^\tau} h^{\frac{n}{2} -\tau} \quad \mbox{for all}\quad  \beta>0.
\end{equation*}
This together with  the estimate \eqref{Md-est} for $M_d$  yields
\begin{eqnarray}\label{term2-small-decay}
\sum_{k=0}^\infty \sum_{i\in \calF_{c 2^{-k}}}{\big|S_{\phi}(y_{i}, \bh(y_i)/2)\setminus G_{\beta}(u,\Omega)\big| } 
\leq \frac{C}{\beta^\tau} \sum_{k=0}^\infty \sum_{i\in \calF_{c2^{-k}}}{\bh(y_i)^{\frac{n}{2} -\tau} }
&\leq& \frac{C}{\beta^\tau} \sum_{k=0}^\infty (c2^{-k})^{\frac{n}{2} -\tau} M_{c2^{-k}}
\nonumber\\ &\leq&  \frac{C'}{\beta^\tau} \sum_{k=0}^\infty (c2^{-k})^{\frac{1}{2} -\tau} \leq \frac{C'}{\beta^\tau}
\end{eqnarray}
provided that $\tau <1/2$. Here $C'$ also depends on $\|\partial\Omega\|_{C^{1,1}}$. The desired estimate is now obtained by combining \eqref{eq:covering2}, \eqref{term1-small-decay} and \eqref{term2-small-decay}.
\end{proof}

To prove Proposition \ref{power-decay-at-boundary}, we use the following localized version at the boundary of 
Lemma~\ref{Vitali-MA}.  
\begin{lemma}
 \label{covering_rk}
Assume $(\Omega, \phi, U)\in \mathcal{P}_{\lambda, \Lambda, \rho, \kappa, \ast}$ and let $w$ be the solution to (\ref{1_MA}).
Let $\psi$ denote one of the functions $\phi$ and $w$. Then
 there exists a sequence of disjoint sections $\{S_{\psi}(y_{i}, \delta_{0}\bar{h}(y_i))\}_{i=1}^\infty$, where $\delta_0=\delta_0(n,\lambda,\Lambda)$, $y_i\in U\cap B_{c^2}$ and $S_{\psi}(y_{i},\bar{h}(y_i)) $ is the maximal interior
 section of $\psi$ in $U$, such that
\begin{equation}U\cap B_{c^2}\subset \bigcup_{i=1}^{\infty} S_{\psi}(y_{i}, \frac{\bh(y_i)}{2}).
\label{ccovering}
\end{equation}
Moreover, we have
\begin{equation}S_{\psi}(y_{i},\bar{h}(y_i))\subset U\cap B_c,\quad \bar{h}(y_i)\leq c.
 \label{csection}
\end{equation}
If we let $M^{loc}_{d}$ denote the number 
of sections $S_{\psi}(y_{i}, \bar{h}(y_i)/2)$ such that $d/2<\bar{h}(y_i)\leq d\leq c$, then
\begin{equation}\label{Md-est-loc}
 M^{loc}_{d}\leq C_{b} d^{\frac{1}{2}-\frac{n}{2}}
\end{equation}
for some constant $C_{b}$ depending only on $\rho, n, \lambda, \Lambda$ and $\|\partial\Omega\cap B_{\rho}\|_{C^{1,1}}.$
\end{lemma}
\begin{proof}
 By Remark~\ref{quad-rk}, we can use Proposition~\ref{tan_sec} to get the same conclusion as in Lemma~\ref{sep-lem}(e) for 
 sections of $\psi$ with centers in $U\cap B_{c^2}$. All these sections thus satisfy \eqref{csection} and are equivalent to ellipsoids. In particular, $\psi$ is strictly convex in
 $U\cap B_c.$
 Furthermore,
 $$S_{\psi}(y_{i},\bar{h}(y_i)) \subset \Big\{x\in B_c\cap U: \dist(x, \p\Omega\cap \p U)\leq 2k_0^{-1} \bar{h}(y_i)^{1/2}\Big\}.$$
 With this in mind and assuming that the sequence  $\{S_{\psi}(y_{i}, \delta_{0}\bar{h}(y_i))\}_{i=1}^\infty$ is disjoint and satisfies \eqref{ccovering}, we argue similarly 
as in deriving the estimate \eqref{Md-est} for $M_d$  to obtain \eqref{Md-est-loc}.

 It remains to establish the covering \eqref{ccovering}. The crucial point 
 in the proof of Lemma~\ref{Vitali-MA} is the engulfing property of interior sections which hold for strictly convex solution to the Monge-Amp\`ere
 equation with bounded right hand side. By our discussion above, $\psi$ is strictly convex in $U\cap B_c$ and thus we obtain \eqref{ccovering}. For 
 completeness, we include the proof here, taken almost verbatim from \cite{S3}. By the engulfing property of interior sections of strictly convex solution to the Monge-Amp\`ere
 equation with bounded right hand side, we can choose $\delta_0$ depending only on $n, \lambda, \Lambda$ with the following property. If $y, z\in B_{c^2}\cap U$
 with $$S_{\psi}(y, \delta_0\bar{h}(y))\cap S_{\psi}(z, \delta_0\bar{h}(z))\neq \emptyset~\text{and}~ 2\bar{h}(y) \geq \bar{h}(z)$$
 then
 $$S_{\psi}(z, \delta_0\bar{h}(z)) \subset S_{\psi}(y, \bar{h}(y)/2).$$
 We choose $S_{\psi}(y_1, \delta_0\bar{h}(y_1))$ from all sections $S_{\psi}(y, \delta_0\bar{h}(y))$, $y\in U\cap B_{c^2}$ such that
 $$\bar{h}(y_1)\geq \frac{1}{2}\sup_{y}\bar{h}(y)$$
 then choose $S_{\psi}(y_2, \delta_0\bar{h}(y_2))$ as above but only from the remaining sections $S_{\psi}(y, \delta_0\bar{h}(y))$ that are disjoint from
 $S_{\psi}(y_1, \delta_0\bar{h}(y_1))$, then $S_{\psi}(y_3, \delta_0\bar{h}(y_3))$, etc. Consequently, we easily obtain
 $$U\cap B_{c^2}\subset \bigcup_{y\in U\cap B_{c^2}}S_{\psi}(y, \delta_0\bar{h}(y))\subset  \bigcup_{i=1}^{\infty}S_{\psi}(y_i, \delta_0\bar{h}(y_i)).$$
\end{proof}
\begin{proof}[Proof of Proposition \ref{power-decay-at-boundary}] 
Our proof is similar to that of Proposition~\ref{Initial-Estimate} using Lemma~\ref{covering_rk}. In the proof of Proposition~
\ref{Initial-Estimate}, we replace $\Omega\setminus G_{\beta}(u, \Omega)$ by 
$(U\cap B_{c^2})\setminus G_{\beta}(u, \Omega)$, the covering of $\Omega$ using Lemma~\ref{Vitali-MA} by the covering of $U\cap B_{c^2}$ using Lemma~\ref{covering_rk}.
By \eqref{csection}, the first term of the right hand side of \eqref{eq:covering2} disappears. For the second term of the right hand side of \eqref{eq:covering2}, we estimate
as in the rest of the proof of Proposition~\ref{Initial-Estimate}. Note that, since all sections in the covering for 
$U\cap B_{c^2}$ satisfy $S_{\phi}(y_i, \bar{h}(y_i))\subset B_c\cap U$, instead of \eqref{c_function}, we now have
$$\|\tilde f\|_{L^n(T(S_{\phi}(y, h)))}
 =h^{1/2} \|f\|_{L^n(S_{\phi}(y, h))}\leq  h^{1/2} \|f\|_{L^n(U\cap B_c)} \leq 1.$$
 In \eqref{term2-small-decay}, we replace $M_d$ by $M^{loc}_d$ and use \eqref{Md-est-loc} to estimate it. The conclusion  of Proposition~\ref{power-decay-at-boundary} follows. Note that
 by \eqref{small-sec}, we have $S_{\phi}(0, r)\subset U\cap B_{c^2}$ if $r\leq c^6$ and the last remark of the proposition follows.
\end{proof}
\subsection{Global stability of cofactor matrices}
\label{cofactor_sec}
In this subsection, we prove that, under suitable geometric conditions,
the cofactor matrices of the Hessian matrices of two convex 
functions defined on the same domain are close if their Monge-Amp\`ere measures and boundary 
values are close in the $L^{\infty}
$ norm. 

We first start with a stability result at the boundary
for the second derivatives and the cofactor matrices of functions in the class $\mathcal{P}$.
\begin{proposition}\label{global-convergence}
Assume $(\Omega, \phi, U)\in \mathcal{P}_{1-\e, 1+ \e, \rho, \kappa, \ast}$. Let $w\in C(\overline{U})$ be  the convex solution to 
\begin{equation*}
\left\{\begin{array}{rl}
\det D^2 w &=1 \qquad \mbox{ in}\quad U\\
w  &=\phi \ \ \ \ \ \ \ \mbox{ on}\quad \partial U.
\end{array}\right.
\end{equation*}
 Then the following statements hold.
\begin{myindentpar}{1cm}
(i) For any $ p>1$, there exist $\epsilon_0=\epsilon_0(p, n,\rho)>0$  
and $C =C(p,n,\rho,\kappa)>0$  such that
$$\|D^{2}\phi-D^{2} w\|_{L^{p}(B_{c^2}\cap U)}\leq  C \e^{\frac{\delta}{n (2p - \delta)}}\quad \mbox{for all}\quad \e\leq \e_0.$$
(ii)  Assume  in addition that $(\Omega, \phi, U)\in \mathcal{P}_{1-\e, 1+ \e, \rho, \kappa, \alpha}$.  Then for any $q\geq 1$, there exist
 $\epsilon_0=\epsilon_0(q, n,\rho)>0$  
and $C =C(q,n,\rho,\kappa,\alpha)>0$   such that
$$\|\Phi-\calW\|_{L^{q}(B_{c^2}\cap U)}\leq  C \e^{\frac{(n-1)\delta}{n (2nq - \delta)}}\quad \mbox{for all}\quad \e\leq \e_0.$$
\end{myindentpar}
Here $\delta=\delta(n, \rho)>0$, and $\Phi, \calW$ are the matrices of cofactors of $D^2\phi$ and $D^2 w$, respectively.
\end{proposition}

\begin{proof} 
(i) Our conclusion follows from the following claims.\\
{\bf Claim 1.} There exist $\epsilon_0=\epsilon_0(p, n,\rho)>0$ small 
and $C_0 =C_0(p,n,\rho,\kappa)>0$  such that
$$\|D^{2}\phi\|_{L^{2p}(B_{c^2}\cap U)} + \|D^{2} w\|_{L^{2p}(B_{c^2}\cap U)}\leq C_{0}\quad \mbox{whenever}\quad \e\leq \e_0.
$$
{\bf Claim 2.} There exist $\delta=\delta(n,\rho)\in (0,1/2)$ and $C=C(n,\rho,\kappa)>0$
such that
\begin{equation}\|D^{2}\phi-D^{2} w\|_{L^{\delta}(B_{c^2}\cap U)}\leq C \e^{1/n}\quad \mbox{for all}\quad \e<\frac{1}{2}.
 \label{small-exp}
\end{equation}
Indeed, let $\theta\in (0, 1)$ be such that
$$\frac{1}{p} =\frac{\theta}{2p} + \frac{1-\theta}{\delta}.$$
Then $1-\theta = \delta/ (2p -\delta)$ and by the interpolation inequality we get
$$ \|D^{2}\phi-D^{2} w\|_{L^{p}(B_{c^2}\cap U)} \leq \|D^{2}\phi-D^{2} w\|^{\theta}_{L^{2p}
(B_{c^2}\cap U)}
\|D^{2}\phi-D^{2} w\|^{1-\theta}_{L^{\delta}(B_{c^2}\cap U)} \leq C \e^{\frac{1-\theta}{n}} =  C \e^{\frac{\delta}{n (2p - \delta)}}.$$ 
We now turn to the proofs of the claims. 

{\bf Claim 1} is essentially Savin's global $W^{2, p}$ estimates for the Monge-Amp\`ere equations \cite{S3}. For the proof in our setting, we 
use Lemma~\ref{covering_rk} and follow his arguments. For completeness, we include
the proof here. Let $\psi$ denote one of the functions $\phi$ and $w$. Then by  Lemma~\ref{covering_rk}, 
 there exists a sequence of disjoint sections $\{S_{\psi}(y_{i}, \delta_{0}\bar{h}(y_i))\}_{i=1}^\infty$, where $y_i\in U\cap B_{c^2}$ and $S_{\psi}(y_{i},\bar{h}(y_i)) $ is the maximal interior
 section of $\psi$ in $U$, such that
\begin{equation*}U\cap B_{c^2}\subset \bigcup_{i=1}^{\infty} S_{\psi}(y_{i}, \frac{\bh(y_i)}{2}).
\label{c_covering}
\end{equation*}
Moreover, we have
\begin{equation*}S_{\psi}(y_{i},\bar{h}(y_i))\subset U\cap B_c,\quad\bar{h}(y_i)\leq c.
 \label{c_section}
\end{equation*}
We will prove that:
There exist $\epsilon_0=\epsilon_{0}(p, \rho, n)>0$ small and $C= C(p, \rho, n)>0$ such that for $\e\leq \e_{0}$, we have
\begin{equation}
\int_{S_{\psi}(y, \frac{\bar{h}(y)}{2})} \abs{D^2 \psi}^{2p}\leq C \,\bh(y)^{\frac{n}{2}}|\log \bar{h}(y)|^{4p}  \quad \forall y\in U\cap B_{c^2}.
\label{MA_section}
\end{equation}
Given this, we can complete the proof of {\bf Claim 1} as follows. We have
\begin{eqnarray}
 \int_{U\cap B_{c^2}}\abs{D^2\psi}^{p}\leq \sum_{i=1}^\infty \int_{S_{\psi}(y_{i}, \frac{\bar{h}(y_{i})}{2})} \abs{D^2 \psi}^{2p} 
 =\sum_{k=0}^\infty \sum_{i\in \calF_{c2^{-k}}}\int_{S_{\psi}(y_{i}, \frac{\bar{h}(y_{i})}{2})} \abs{D^2 \psi}^{2p},
\label{MA_section2}
 \end{eqnarray}
where $\mathcal{F}_{d}$ is the family of sections $S_{\psi}(y_{i}, \frac{\bar{h}(y_{i})}{2})$ such that $d/2<\bar{h}(y_i)\leq d\leq c.$ 
By \eqref{MA_section}, we have for each $S_{\psi}(y_{i}, \frac{\bar{h}(y_{i})}{2})\in \mathcal{F}_{d}$,
$$\int_{S_{\psi}(y, \frac{\bar{h}(y)}{2})} \abs{D^2 \psi}^{2p}\leq C \abs{\log d}^{4p} \abs{S_{\psi} (y_i, \delta_0 \bar{h}(y_i))}$$
and since
$$S_{\psi} (y_i, \delta_0 \bar{h}(y_i))\subset \{x\in B_c\cap U: \dist(x, \p\Omega\cap \p U)\leq 2k_0^{-1} d^{1/2}\} $$
are disjoint, we find
$$\sum_{i\in \calF_{d}}\int_{S_{\psi}(y_{i}, \frac{\bar{h}(y_{i})}{2})} \abs{D^2 \psi}^{2p}\leq C_1\abs{\log d}^{4p} d^{1/2}$$
where $C_1$ now depends also on $\kappa$ which is the upper bound for $\|\p\Omega\cap B_{c}\|_{C^{1,1}}$.
Therefore, {\bf Claim~1} easily follows from \eqref{MA_section2} by adding these inequalities for $d= c2^{-k}$, $k=0, 1, \dots$

It remains to prove (\ref{MA_section}). Let $h:= \bar{h}(y)$. Then $h\leq c.$
By applying Proposition~\ref{tan_sec}  
to $S_{\psi}(y, h)$, we find that it is equivalent to an ellipsoid $E_h$, i.e.,
$$k_0 E_h \subset S_{\psi}(y, h)-y \subset k_{0}^{-1} E_h,$$
where
$E_h := h^{1/2}A_{h}^{-1}B_1$ with $\det A_{h}=1$ and  $\|A_{h}\|, \, \|A_h^{-1} \| \le C |\log h|.$
We use the following rescalings similar to those in Subsection~\ref{rescale-sec}:
$$\tilde{\Omega}_{h}  := h^{-1/2} A_{h} (\Omega-y),$$
and  for $ x\in \tilde{\Omega}_{h}$
$$\tilde{\psi}_{h}(x) := 
h^{-1} \left[\psi ( y + h^{1/2} A^{-1}_{h} x)-\psi(y)-\nabla\psi(y)\cdot(h^{1/2}A^{-1}_h x)-h\right].$$
Then
$$B_{k_0}\subset S_{\tilde{\psi}_{h}} (0, 1)\equiv h^{-1/2} A_{h} \big(S_{\phi}(y, h) -y\big)\subset B_{k_0^{-1}}.$$
We have
$$\det D^{2} \tilde{\psi}_{h} (x)= \det D^2 \psi( y + h^{1/2} A^{-1}_{h} x)\quad \mbox{and}\quad \tilde{\psi}_{h}= 0 \,\mbox{ on }\, \partial S_{\tilde{\psi}_{h}} (0, 1).$$

For simplicity, we denote $\tilde{S}_{t}(0) := S_{\tilde{\psi}_{h}} (0, t)$ for $t>0.$
If $\psi=\phi$ then by Caffarelli's interior $W^{2,p}$ estimates for the Monge-Amp\`ere equation \cite{C3}, we have
$$\int_{\tilde{S}_{\frac{1}{2}}(0))}\abs{D^2 \tilde\psi_{h}}^{2p}\leq C$$
if $\e\leq \e_{0}$ small depending only on $p, \rho$ and $n$. If $\psi=w$ then as $\det D^2 w=1$, the above inequality obviously holds.
Using the fact
$$D^{2} \psi( y + h^{1/2} A^{-1}_{h} x) = A_{h}^{t} \, D^{2} \tilde{\psi}_{h}(x) \, A_{h},$$
we obtain (\ref{MA_section}) from
\begin{eqnarray*}
\int_{S_{\phi}(y,\frac{h}{2})}|D^{2} \psi(z)|^{2p} \, dz&=& h^{\frac{n}{2}} 
\int_{\tilde{S}_{\frac{1}{2}}(0)} |A_{h}^{t} \, D^{2} \tilde{\psi}_{h}(x) \, A_{h}|^{2p}\, dx\\
&\leq &  C \,h^{\frac{n}{2}} \, |\log h|^{4p} \int_{\tilde{S}_{\frac{1}{2}}(0)} |D^{2} \tilde{\psi}_{h}(x)|^{p}\ \leq C \,h^{\frac{n}{2}} \, |\log h|^{4p}.
\end{eqnarray*}

Finally, we verify {\bf Claim 2} by proving (\ref{small-exp}). As in \cite[Lemma 3.4]{GN1}, we note that the difference $v:= \phi-w$ is a subsolution
(supersolution) of linearized Monge-Amp\`ere equations with bounded right hand side, corresponding to the potentials $w$ and 
$\phi$ respectively. We cover $U\cap B_{c^2}$ by sections of $w$ and $\phi$ using Lemma \ref{covering_rk}. In each of these sections, we can use the
one-sided $W^{2,\delta}$ estimates of Guti\'errez-Tournier \cite{GT}. Then, adding these estimates as in the proof of 
Theorem~\ref{small-d2}, we get \eqref{small-exp}. The details are as follows.

 Consider the operator $\mathcal{M}u:= (\det D^{2}u)^{1/n}$ and its linearized operator
$$\hat{\mathcal{L}}_{u} v:= \frac{1}{n} (\det D^{2}u)^{1/n} \trace \big((D^{2} u)^{-1} D^{2} v\big).$$
Notice that $\hat{\mathcal{L}}_{u} v$ and the operator $\mathcal{L}_{u} v$ defined in \eqref{LMA-eq} are related by
$$ \mathcal{L}_{u} v= n (\det D^{2}u)^{\frac{n-1}{n}} \hat{\mathcal{L}}_{u}v.$$
Let $v:= \phi -w$ and $g:= \det D^2\phi.$ Since $\mathcal{M}$ is concave, we obtain
$$g^{1/n}-1 = \mathcal{M}\phi -\mathcal{M}w \leq \hat{\mathcal{L}}_{w} v$$
and hence
\begin{equation}
\label{sub-w}
\mathcal{L}_{w} v = n (\det D^{2}w)^{\frac{n-1}{n}} \hat{\mathcal{L}}_{w}v\geq -n |g^{1/n}-1|.
\end{equation}
We also have
$\hat{\mathcal{L}}_{\phi} v \leq \mathcal{M}\phi - \mathcal{M} w \leq |g^{1/n}-1|$
and thus
\begin{equation}
\label{sup-phi}
\mathcal{L}_{\phi} v = n (\det D^{2}\phi)^{\frac{n-1}{n}} \hat{\mathcal{L}}_{\phi}v\leq n (1+\e)^{\frac{n-1}{n}} |g^{1/n}-1|\leq 2n
 |g^{1/n}-1|.
\end{equation}
On the other hand, it follows from the maximum principle (\cite[Lemma 3.1]{H09}) that
\begin{eqnarray}
\label{v-max}
\|v\|_{L^{\infty}(U)}\leq C_{n} \diam(U)
\|g^{1/n}-1\|_{L^{n}(U)}.
\end{eqnarray}
We cover $U\cap B_{c^2}$ by sections of $w$ using Lemma~\ref{covering_rk}. From \eqref{sub-w} and \eqref{v-max},  we can use  Guti\'errez-Tournier's
one-sided $W^{2,\delta}$ estimates \cite{GT} instead of Theorem~\ref{GT_loc} in each of these sections to estimate the $L^{\delta}$ norm
of $(D^2 v)^{+}$.  After that, we argue as in the proof of Theorem~\ref{small-d2}, and taking into account Lemma~\ref{covering_rk} again
 to obtain $\delta_{1}=\delta_1(n,\rho)\in (0,1/2)$ and $C_{1}=C_1(n,\rho,\kappa)>0$ such that
\begin{multline}
\label{d2-plus}
\|(D^2 v)^{+}\|_{L^{\delta_{1}}(U\cap B_{c^2})} \leq C_{1}\Big( \|v\|_{L^{\infty}(U\cap B_c)} + 
\|(\mathcal{L}_{w}v)^{-}\|_{L^{n}(U\cap B_c)}\Big) \leq C_{1} \|g^{1/n}-1\|_{L^{n}(U)}.
\end{multline}
Similarly, from \eqref{sup-phi}, \eqref{v-max} and by covering $U\cap B_{c^2}$ by sections of $\phi$, we obtain
\begin{multline}
\label{d2-minus}
\|(D^2 v)^{-}\|_{L^{\delta_{2}}(U\cap B_{c^2})} \leq C_{2}\Big(\|v\|_{L^{\infty}(U\cap B_c)} +
\|(\mathcal{L}_{\phi}v)^{+}\|_{L^{n}(U\cap B_c)}\Big) \leq C_{2} \|g^{1/n}-1\|_{L^{n}( U)}.
\end{multline}
Let 
$\delta:=\min\{\delta_{1}, \delta_{2}\}.$ Then, from \eqref{d2-plus} and \eqref{d2-minus}, we obtain \eqref{small-exp} as desired since
\begin{equation*}
\|D^2 v\|_{L^{\delta}(U\cap B_{c^2})} \leq C \|g^{1/n}-1\|_{L^{n}(U)} \leq C \e^{1/n}.
\end{equation*}
(ii) The key to the proof is the following estimate
\begin{multline}
\label{cof-key}
\|\Phi-\calW\|_{L^{q}(U\cap B_{c^2})}  \leq C_{n}\left(\e + \|D^2 w\|^{n-1}_{L^{\infty}(U\cap B_{c^2})}
 \|D^2\phi-D^2 w\|^{n-1}_{L^{qn}(U\cap B_{c^2})}\right)
\|D^2 \phi\|^{n-1}_{L^{qn}(U\cap B_{c^2})}
\end{multline}
which can be deduced from the proof of Lemma 3.5 in \cite{GN1}. 

As in {\bf Claim 1} in the proof of part (i), we have
\begin{equation}
\label{global-phi}
\|D^2 \phi\|_{L^{q n}(U\cap B_{c^2})} \leq C_{0}\quad\mbox{for all}\quad \e\leq \e_0,
\end{equation}
where $\e_{0}=\e_0(q,n,\rho)>0$ small and $C_0 =C_0(q,n,\rho,\kappa)>0$.\\
On the other hand, by (vii) in the definition of the class $\mathcal{P}$, we have
\begin{equation}
\label{global-w}
\|D^2 w\|_{L^{\infty}(U\cap B_{c^2})} \leq C_{1}(n,\alpha,\rho).
\end{equation}
Putting \eqref{cof-key}--\eqref{global-w} together, we obtain for $\e\leq\e_0$
\begin{equation*}
\label{cof-final}
\|\Phi-\calW\|_{L^{q}(U\cap B_{c^2})} \leq C_{n} (\e + C_{1}^{n-1}  \|D^2\phi-D^2 w\|^{n-1}_{L^{qn}(U\cap B_{c^2})})
 C_{0}^{n-1}.
\end{equation*}
By applying  part (i) of this proposition to $p= qn,$ we then get the desired conclusion.
\end{proof}
We also obtain the following global 
stability of matrices of cofactors.
\begin{lemma} (Global stability of cofactor matrices)\label{convergence-of-cofactors}
Let $\Omega\subset\R^{n}$ be a uniformly convex domain satisfying (\ref{global-tang-int}) and $\|\partial\Omega\|_{C^3}\leq 1/\rho$.
For any  $q\geq 1$, there exist 
$C, \e_0>0$ depending only on $q$, $n$ and $\rho$ with the following property.
If $\phi,\,w\in C(\overline \Omega)$ are convex functions satisfying
\begin{equation*}
\left\{\begin{array}{rl}
1-\epsilon\!\!\leq\! \det D^2 \phi \!\!\!\!\!&\leq 1+\epsilon\quad \mbox{in}\quad \Omega\\
\!\!\phi\!\!\!\!\!&=0 \ \ \quad \quad\mbox{on}\quad\partial\Omega
\end{array}\right.
\qquad \mbox{and}\qquad \left\{\begin{array}{rl}
 \det D^2 w \!\!\!\!\!&= 1 \quad
 \mbox{in}\quad \Omega\\
\!\!w\!\!\!\!\!&=0 \quad
\mbox{on}\quad\partial\Omega,
\end{array}\right.
\end{equation*}  
then for some small constant $\delta>0$ depending only on $n$ and $\rho$, we have
\[
\|\Phi-\calW\|_{L^q(\Omega)}\leq C \e^{\frac{(n-1)\delta}{n (2nq - \delta)}}\quad \mbox{for all}\quad \e\leq \e_0.
\]

\end{lemma}
\begin{proof}
The proof follows the lines of the proof of Proposition~\ref{global-convergence} using Proposition \ref{pro:quadsep}. Here we choose $U=\Omega$, replace $U\cap B_{c^2}$
by $\Omega$ and use the covering Lemma~\ref{Vitali-MA}. The estimate \eqref{global-w} is now a classical result of Caffarelli-Nirenberg-Spruck \cite{CNS}.
\end{proof}

\subsection{Global $W^{2, 1 +\eps}$ estimates for convex solutions}
\label{convex_sec}
In this subsection, we establish the global $W^{2, 1 +\eps}$ estimates for convex solutions to the linearized Monge-Amp\`ere equations. These estimates are simple consequence
of the global stability of cofactor matrices in Subsection \ref{cofactor_sec}.
\begin{theorem} Let $\Omega$ be a uniformly convex domain satisfying \eqref{global-tang-int} with $\p\Omega\in C^{3}$. Let $\phi\in
 C(\overline{\Omega})\cap C^{2}(\Omega)$ be a convex function satisfying
$$0<\lambda\leq \det D^{2}\phi \leq\Lambda~\text{ in }~\Omega\quad\mbox{and}\quad \phi =0 ~\text{ on }~\p\Omega.$$
Let $v$ be the convex solution to
\begin{equation*}
\left\{\begin{array}{rl}
\Phi^{ij} v_{ij} &=f \qquad \mbox{ in}\quad \Omega,\\
v  &=0\ \ \ \ \ \ \ \mbox{ on}\quad \partial \Omega,
\end{array}\right.
\end{equation*}
where $f\in L^\infty(\Omega)$.
Then, there exist $\gamma>1$ and $C>0$ depending only on $\lambda, \Lambda, n$ and $\Omega$ such that
\begin{equation}\|D^2 v\|_{L^{\gamma}(\Omega)}\leq C\|f\|_{L^{\infty}(\Omega)}.
\label{21LMA-eq}
\end{equation}
\label{21LMA}
\end{theorem}
\begin{remark}
(i) De Phillipis-Figalli-Savin \cite{DPFS} and Schmidt \cite{Sch} discovered the 
interior $W^{2, 1+\eps}$ estimates for convex solution $\phi$ to the Monge-Amp\`ere equation
\begin{equation*}
 \det D^2 \phi =g ~\mbox{ in}\quad \Omega,\quad\text{and}\quad
\phi  =0~\mbox{ on}\quad \partial \Omega,
\end{equation*}
with $0<\lambda\leq g\leq \Lambda$. In these works, the convexity of $\phi$ plays a crucial role, especially in giving a bound for $|D^2 \phi|$ by 
$\Delta \phi.$  Since
$\Phi^{ij}\phi_{ij} = n \det D^2\phi= ng,$
our theorem is a natural extension of De Phillipis-Figalli-Savin's and Schmidt's estimates. \\
(ii) The convexity of $v$ and standard arithmetic-geometric inequality give
$$f = \Phi^{ij}v_{ij} =\trace(\Phi \, D^2v)\geq n (\det \Phi)^{1/n}(\det D^2 v)^{1/n}\geq 0. $$ 
(iii) It would be interesting to remove the convexity of $v$ in the statement of Theorem \ref{21LMA}.
\end{remark}
Now, we proceed with the proof of Theorem \ref{21LMA}. To do this, we first establish the following Sobolev stability result.
\begin{proposition} (Sobolev stability estimates)
Let $\Omega$ be a uniformly convex domain with $\p\Omega\in C^{3}$. Let $\phi_{k}\in C(\overline{\Omega})\cap C^{2}(\Omega)$ 
($k=1,2$) be convex Aleksandrov solutions of
\begin{equation*}
\det D^2 \phi_{k} =g_{k}~\mbox{ in}\quad \Omega,\quad \mbox{and}\quad
\phi_{k}  =0~\mbox{ on}\quad \partial \Omega,
\end{equation*}
with $0<\lambda\leq g_{k}\leq \Lambda$ in $\Omega$. 
Then there exist $\gamma>1$, $\alpha\in (0,1)$ and $C>0$ depending only on $n, \lambda, \Lambda$ and $\Omega$ such that
\begin{equation}
\label{stab-eq}
\|D^2\phi_{1}-D^2\phi_2\|_{L^{\gamma}(\Omega)}\leq C \|g_{1}- g_{2}\|^{\frac{\alpha}{n}}_{L^{1}(\Omega)}.
\end{equation}
\label{stab-sobolev}
\end{proposition}
\begin{proof}
The interior counterpart of our proposition was established by De Phillipis-Figalli \cite{DPF2}. Here, we will prove the boundary version with a different method. Our proof relies on 
the $W^{2,\delta}$ estimates of Guti\'errez-Tournier \cite{GT} for solutions to the linearized Monge-Amp\`ere equation. \\
First, using Proposition~\ref{pro:quadsep}, \cite[Lemma~3.1]{H09} and arguing as in the proof of \eqref{small-exp} in 
Proposition~\ref{global-convergence}, we find a small $\delta>0$ and $C_{1}>0$ depending 
only on $n, \lambda, \Lambda$ and $\Omega$ such that
\begin{equation}
\label{stab-eq-2}
\|D^2\phi_{1}-D^2\phi_2\|_{L^{\delta}(\Omega)}\leq 
C_{1} \|g^{\frac{1}{n}}_{1}- g^{\frac{1}{n}}_{2}\|_{L^{n}(\Omega)} \leq C_{1} \|g_{1}- g_{2}\|^{\frac{1}{n}}_{L^{1}(\Omega)}.
\end{equation}
Second, using De Phillipis-Figalli-Savin's and Schmidt's interior $W^{2, 1+\eps}$ estimates for solutions to the Monge-Amp\`ere equation
 \cite{DPFS, Sch} and arguing as in \cite{S3}, we obtain the following global $W^{2, 1+\eps}$ estimates
\begin{equation}
\label{bdr-21}
\|D^2\phi_{1}\|_{L^{\gamma_{1}}(\Omega)} + \|D^2\phi_{2}\|_{L^{\gamma_{1}}(\Omega)} \leq C_{2},
\end{equation}
where $\gamma_{1}>1$ and $C_{2}>0$  depend only on $n, \lambda, \Lambda,$ and $\Omega$. \\
 We now choose $\alpha\in (0,1)$ sufficiently close to $0$ so that
$$\frac{1}{\gamma}:= \frac{\alpha}{\delta} + \frac{1-\alpha}{\gamma_{1}}<1,$$
i.e., $\gamma>1.$ 
Then by the interpolation inequality, we obtain
$$\|D^2\phi_{1}-D^2\phi_2\|_{L^{\gamma}(\Omega)}\leq \|D^2\phi_{1}-D^2\phi_2\|^{\alpha}_{L^{\delta}(\Omega)} 
\|D^2\phi_{1}-D^2\phi_2\|^{1-\alpha}_{L^{\gamma_{1}}(\Omega)}$$
which together with \eqref{stab-eq-2} and \eqref{bdr-21} yields  the estimate \eqref{stab-eq}. 
\end{proof}
\begin{proof}[Proof of Theorem \ref{21LMA}] 
For any $t\in (0, \|f\|^{-1}_{L^{\infty}(\Omega)}]$, we have 
$\phi =\phi + tv~\text{on}~\p\Omega$
and, by the convexity of $v$, $\phi\geq \phi + tv~\text{in}~\Omega.$
Thus
$$\lambda\leq \det D^2 \phi \leq \det D^{2} (\phi + tv).$$
Moreover by the concavity of the map $\phi\mapsto \log \det D^2 \phi,$ we obtain
$$\log \det D^2 (\phi + tv)\leq log \det D^2 \phi + t \phi^{ij}v_{ij} = log \det D^2\phi + t\frac{f}{det D^2 \phi}.$$
Therefore,
\begin{eqnarray*}
0\leq \det D^2 (\phi + tv) - \det D^2 \phi \leq (\det D^2 \phi) \,  \big(e^{\frac{tf}{\det D^2\phi}}-1\big)
\leq  \Lambda (e^{\frac{1}{\lambda}}-1) \quad \mbox{in}\quad \Omega.
\end{eqnarray*}
Applying the stability result in Proposition \ref{stab-sobolev}, we can find $\alpha, C>0, \gamma>1$ depending only 
on $n, \lambda, \Lambda$ and $\Omega$ such that
$$\|t D^2 v\|_{L^{\gamma}(\Omega)}\leq C  \|\det D^2 (\phi + tv) - \det D^2 \phi\|^{\frac{\alpha}{n}}_{L^{1}(\Omega)}
\leq C  \| \Lambda (e^{\frac{1}{\lambda}}-1)\|^{\frac{\alpha}{n}}_{L^{1}(\Omega)}.$$
The estimate \eqref{21LMA-eq} follows by taking $t= \|f\|^{-1}_{L^{\infty}(\Omega)}.$\end{proof}

\section{\bf Global H\"older  Estimates and Approximation Lemma}
In this section, we establish global H\"older continuity estimates for solutions to the linearized Monge-Amp\`ere equation under natural assumptions on the domain, Monge-Amp\`ere
measure and H\"older continuous boundary data. We then use these H\"older estimates to prove approximation lemmas allowing us 
to approximate the solution $u$ to 
$\mathcal{L}_{\phi} u =f$
by smooth solutions of linearized Monge-Amp\`ere equations
associated with convex functions $w$ whose Monge-Amp\`ere measures are close to that of $\phi$. 
\subsection{Global H\"older  estimates}
\label{holder_sec}
In this subsection, we derive global H\"older estimates for solutions to the linearized Monge-Amp\`ere equation in convex domains when the right hand side is 
assumed to be in $L^{n}$ and the boundary data is H\"older continuous. These estimates extend the  global H\"older estimates in \cite{L}  where 
the domains are assumed to be uniformly convex.\\

Our first main theorem is concerned with H\"older estimates in a neighborhood of a boundary point. Its precise statement is as follows.
\begin{theorem}\label{global-H}
Assume $\Omega$ and $\phi$  satisfy 
\eqref{global-tang-int}, \eqref{eq_u},\eqref{om_ass} and if $x_0\in \p\Omega\cap B_\rho$ then
\begin{equation}\label{eq:sep-half-domain}
 \rho\abs{x-x_{0}}^2 \leq \phi(x)- \phi(x_{0})-\nabla \phi(x_{0}) \cdot (x- x_{0})
 \leq \rho^{-1}\abs{x-x_{0}}^2, ~\forall x\in\p\Omega.
\end{equation}
Let $u \in C\big(B_{\rho}\cap 
\overline{\Omega}\big) \cap W^{2,n}_{loc}(B_{\rho}\cap 
\Omega)$  be a  solution to 
\begin{equation*}
 \left\{
 \begin{alignedat}{2}
   \Phi^{ij}u_{ij} ~& = f ~&&\text{ in } ~ B_{\rho}\cap \Omega, \\\
u &= \varphi~&&\text{ on }~\p \Omega \cap B_{\rho},
 \end{alignedat} 
  \right.
\end{equation*} 
where $\varphi\in C^{\alpha}(\partial\Omega\cap B_{\rho})$ for some $\alpha\in (0,1)$. Then, there exist constants $\beta, C >0 $ depending only on $\lambda, \Lambda, n, \alpha$ and $\rho$ 
such that 
$$|u(x)-u(y)|\leq C|x-y|^{\beta}\Big(
\|u\|_{L^{\infty}(\Omega\cap B_{\rho})} + \|\varphi\|_{C^\alpha(\partial\Omega\cap B_{\rho})}  + \|f\|_{L^{n}(\Omega\cap B_{\rho})} \Big),~\forall x, y\in \Omega\cap B_{\rho/2}. $$
\end{theorem}
As an immediate consequence of Theorem~\ref{global-H}, we obtain the following estimates which are 
the global counterparts of Caffarelli-Guti\'errez's interior H\"older estimates for solutions to the linearized Monge-Amp\`ere equation  \cite{CG2}.

\begin{theorem}\label{global-h}
Assume $\Omega$ and $\phi$  satisfy
\eqref{global-tang-int}--\eqref{global-sep}. 
Let $u \in C(\overline{\Omega})\cap W^{2, n}_{loc}(\Omega)$   be a  solution to 
\begin{equation*}
   \Phi^{ij}u_{ij}= f ~\text{ in } ~  \Omega,\quad\mbox{and}\quad
u = \varphi~\text{ on }~\p \Omega,
\end{equation*} 
where $\varphi\in C^{\alpha}(\partial\Omega)$ for some $\alpha\in (0,1)$. Then, there exist constants $\beta, C >0 $ depending only on $\lambda, \Lambda, n, \alpha$ and $\rho$ 
such that 
$$|u(x)-u(y)|\leq C|x-y|^{\beta}\Big(
\|u\|_{L^{\infty}(\Omega)} + \|\varphi\|_{C^\alpha(\partial\Omega)}  + \|f\|_{L^{n}(\Omega)} \Big),~\forall x, y\in \Omega. $$
\end{theorem}

\bigskip

The key to the proof of Theorem \ref{global-H} is the following boundary H\"older estimates.
\begin{proposition}
Let $\phi$ and $u$ be as in Theorem \ref{global-H}. Then, there exist $\delta, C$ depending only on $\lambda, \Lambda, n,  \alpha, \rho$ 
such that, for any $x_{0}\in\partial\Omega\cap B_{\rho/2}$, we have
$$|u(x)-u(x_{0})|\leq C|x-x_{0}|^{\frac{\alpha}{\alpha +3n}} \Big(
\|u\|_{L^{\infty}(\Omega\cap B_{\rho})} + \|\varphi\|_{C^\alpha(\partial\Omega\cap B_{\rho})}  + \|f\|_{L^{n}(\Omega\cap B_{\rho})} \Big),~\forall x\in \Omega\cap B_{\delta}(x_{0}). $$
 \label{local-H}
\end{proposition}
\begin{proof}[Proof of Theorem \ref{global-H}]
The boundary H\"older estimates in Proposition \ref{local-H} combined with the interior H\"older continuity estimates of Caffarelli-Guti\'errez \cite{CG2} and Savin's Localization Theorem
\cite{S1, S2, S3} gives the global H\"older 
estimates in  Theorem \ref{global-H}. The precise arguments are almost the same as the proof  of \cite[Theorem~1.4]{L}. Since \cite[Theorem~1.4]{L} is a global result and our theorem is local,
we indicate some differences in the arguments. It suffices to prove the theorem for $x, y\in B_{c^2}\cap \Omega$. We use the quadratic separation (\ref{eq:sep-half-domain}) and Proposition \ref{tan_sec} to show that 
if $y\in \Omega\cap B_{c^2}$ then the maximal interior section $S_{\phi}(y, \bar{h}(y))$ is contained in $\Omega\cap B_c$ and so tangent to $\p\Omega$ at $x_{0}\in\p\Omega\cap B_{c}$ (see Lemma \ref{sep-lem}(e)).
Using this fact, Caffarelli-Guti\'errez's interior H\"older estimates \cite{CG2} and Proposition \ref{local-H}, we obtain as in \cite{L}
$$\abs{u(z_1)-u(z_2)}\leq \abs{z_1-z_2}^{\beta}\Big(
\|u\|_{L^{\infty}(\Omega\cap B_{\rho})} + \|\varphi\|_{C^\alpha(\partial\Omega\cap B_{\rho})}  + \|f\|_{L^{n}(\Omega\cap B_{\rho})} \Big)~\forall z_1, z_2\in S_{\phi}(y, \frac{\bar{h}(y)}{2}).$$
The rest of the argument is the same as in \cite{L}.
 \end{proof}

The proof of Proposition~\ref{local-H} is based on a construction of suitable barriers. Assume $\phi$ and $\Omega$ satisfy the assumptions 
in the proposition. 
We also assume for simplicity that
 $\phi(0)=0\, \mbox{ and } \, \nabla \phi(0)=0.$ 
We now construct a supersolution as in \cite[Lemma 6.2]{LS}.
 \begin{lemma}[Supersolution] 
Given $\delta$ universally small ($\delta\leq \rho$), define
$$\tilde \delta:= \frac{\delta^3}{2}\quad \mbox{and} \quad M_{\delta}:= \frac{2^{n-1} \Lambda^n}{\lambda^{n-1}}\frac{1}{\delta^{3n-3}}
\equiv \frac{\Lambda^{n}}{(\lambda \tilde \delta)^{n-1}}.$$
Then the function
$$w_{\delta}(x', x_n):= M_{\delta} x_{n} +  \phi - \tilde \delta |x'|^2 - \frac{\Lambda^n}{( \lambda \tilde \delta)^{n-1}} 
 x_{n}^2\quad\mbox{for}\quad (x', x_n)\in \overline{\Omega}$$ satisfies 
 $$\mathcal{L}_{\phi} (w_{\delta}):= \Phi^{ij}(w_{\delta})_{ij} \leq - n\Lambda\quad\mbox{in}\quad \Omega,$$
and
$$w_{\delta} \ge 0 ~\text{ on }~ \p(\Omega \cap B_\delta), \quad w_{\delta} \ge \frac{\delta^3}{2} ~\text{ on }~  \Omega\cap \p B_{\delta}.$$
\label{key-lem_sup}
\end{lemma}

\begin{proof}
We recall from \eqref{small-sec} that
$$ \overline \Omega \cap B^{+}_{ch^{1/2}/\abs{\log h}}\subset S_{\phi}(0, h) 
\subset \overline{\Omega} \cap B^{+}_{C h^{1/2} \abs{\log h}}.$$
The first inclusion gives
$\phi\leq h~\text{ in }~ \overline \Omega \cap B^{+}_{ch^{1/2}/\abs{\log h}}$ and hence for $x$ close to the origin,
$$\phi(x)\leq C\abs{x}^2\abs{\log \abs{x}}^2.$$
Similarly, the second inclusion gives
$$\phi(x)\geq  c\abs{x}^2\abs{\log \abs{x}}^{-2}\geq \abs{x}^3$$ for $x$ close to the origin.
In conclusion, we have
\begin{equation}\label{est-for-phi}
 \abs{x}^3 \leq \phi(x) \leq C\abs{x}^2 \abs{\log{|x|}}^2
\end{equation}
if $\abs{x}\leq \delta$ for $\delta$ universally small. Therefore, the choice of $\tilde\delta$ gives
$$\phi(x)- \tilde \delta |x'|^2 \geq \abs{x}^3-\tilde\delta \abs{x}^2\geq \frac{1}{2}\abs{x}^3
 = \tilde \delta \quad \mbox{on }\, \Omega\cap \p B_{\delta}.$$
On the other hand, the choice of $M_{\delta}$ implies that 
\begin{equation*}
 M_\delta x_n - \frac{\Lambda^n}{( \lambda \tilde \delta)^{n-1}} 
 x_{n}^2  \ge 0  \quad \quad \text{on}\quad\overline{ \Omega\cap B_{\delta}}.
 \end{equation*}
Hence, $w_\delta \geq \tilde\delta$ on $\Omega\cap \p B_{\delta}$ while
on  $\p \Omega\cap B_{\delta}$, the quadratic separation \eqref{eq:sep-half-domain} and $\delta\leq \rho$ give
 $$w_{\delta} \ge \phi -  \tilde \delta  |x'|^2 \ge 0.$$
As a consequence, we obtain the desired inequalities for $w_{\delta}$ on $\p(\Omega\cap B_{\delta})$. 

It remains to prove that $\mathcal{L}_{\phi} (w_{\delta})\leq - n\Lambda$ in $\Omega$. If we denote $$q(x):=\frac 12 \left( \tilde \delta  |x'|^2 + \frac{\Lambda^n}{( \lambda \tilde \delta)^{n-1}} 
 x_{n}^2 \right),$$
then $$\det D^2 q=\frac{\Lambda^n}{\lambda^{n-1}}, \quad D^2q \ge \tilde \delta I.$$
Using the matrix inequality
\begin{equation*}
\trace (AB) \geq n (\det A \det B)^{1/n}~\text{ for }~ A, B~\text{ symmetric }~\geq 0,
\end{equation*}
we get
\begin{equation*}
 \mathcal{L}_\phi q=\trace(\Phi\, D^2 q) \geq n (\det (\Phi) \det D^{2} q)^{1/n} = n ((\det D^{2}\phi)^{n-1}\frac{\Lambda^n}{\lambda^{n-1}})^{1/n}\geq n\Lambda.
\end{equation*}
Since $\mathcal{L}_\phi x_n=0,$ we find
$$\mathcal{L}_\phi w_{\delta} =L_\phi(M_{\delta}x_n + \phi - 2 q) = \Phi^{ij}\phi_{ij}- 2 \mathcal{L}_\phi q
= n\det D^2\phi-  2\mathcal{L}_\phi q\le  -n\Lambda\quad \mbox{in}\quad \Omega.$$
\end{proof}

\begin{proof}[Proof of Proposition \ref{local-H}] Our proof follows closely the proof of Proposition 2.1 in \cite{L}.
We can suppose  that  $K:= \|u\|_{L^{\infty}(\Omega\cap B_{\rho})} + \|\varphi\|_{C^\alpha(\partial\Omega\cap B_{\rho})}  + \|f\|_{L^{n}(\Omega\cap B_{\rho})}$ is finite. By  working with the function $v:= u/K$ instead of $u$, we can assume in addition that
\[
\|u\|_{L^{\infty}(\Omega\cap B_{\rho})} + \|\varphi\|_{C^\alpha(\partial\Omega\cap B_{\rho})}  + \|f\|_{L^{n}(\Omega\cap B_{\rho})}\leq 1
\]
and need to show that the inequality
\begin{equation}\label{equivalent-conclusion}
|u(x)-u(x_{0})|\leq C|x-x_{0}|^{\frac{\alpha}{\alpha +3n}}\quad\forall x\in \Omega\cap B_{\delta}(x_{0})
\end{equation}
holds for all $x_0\in \Omega\cap B_{\rho/2}$,  where  $\delta$ and $C$ depends only on $\lambda, \Lambda, n,  \alpha$ and $\rho$. 

We prove \eqref{equivalent-conclusion} for $x_{0} =0$. However, our arguments apply to all points $x_0\in \Omega\cap B_{\rho/2}$ with
obvious modifications.
For any $\varepsilon \in (0,1)$, we consider the functions
$$h_{\pm}(x) := u(x)- u(0)\pm \e\pm \frac{6}{\delta_2^3} w_{\delta_2}$$
in the region
 $$A:= \Omega\cap B_{\delta_{2}}(0),$$ where $\delta_{2}$ is small to be chosen later and the function $w_{\delta_2}$ is as in Lemma~\ref{key-lem_sup}.
We remark that $w_{\delta_2}\geq 0$ in $A$ by the maximum principle. Observe that if $x\in\partial\Omega$
with $|x|\leq \delta_{1}(\e):= \e^{1/\alpha}$ then, 
\begin{equation}
\label{bdr-ineq}|u(x)-u(0)| =|\varphi(x)-\varphi(0)| \leq |x|^{\alpha} \leq \e.
\end{equation}
On the other hand, if $x\in \Omega \cap \p B_{\delta_2}$ then from Lemma~\ref{key-lem_sup}, we obtain
$\frac{6}{\delta_2^3} w_{\delta_2}(x)\geq 3.$
It follows that, if we choose $\delta_{2}\leq \delta_{1}$ then from \eqref{bdr-ineq} and $|u(x)-u(0)\pm \e|\leq 3$, we get
$$h_{-}\leq 0, \, h_{+}\geq 0~\text{ on }~\partial A.$$
Also from Lemma~\ref{key-lem_sup}, we have 
$$\mathcal{L}_{\phi} h_{+}\leq f, \,  \mathcal{L}_{\phi} h_{-}\geq f~\text{ in }~A.$$
Hence the ABP estimate applied in $A$ gives
\begin{equation}\label{h-from-above}
h_{-}\leq  C_{1}(n, \lambda)diam (A) \|f\|_{L^{n}(A)}\leq C_{1}(n, \lambda)\delta_{2}~\text{ in }~ A
\end{equation}
and 
\begin{equation}\label{h-from-below}
h_{+}\geq - C_{1}(n, \lambda)diam (A) \|f\|_{L^{n}(A)}\geq  -C_{1}(n,\lambda)\delta_{2}~\text{ in }~ A.
\end{equation}
By restricting $\e\leq C_{1}(n, \lambda)^{\frac{-\alpha}{1-\alpha}}$, we can assume that
$$\delta_{1} = \e^{1/\alpha}\leq \frac{\e}{C_{1}(n,\lambda)}.$$
Then, for $\delta_{2}\leq \delta_{1}$, we have $C_{1}(n,\lambda)\delta_{2}\leq \e$ and thus, for all $x\in A$, we obtain from \eqref{h-from-above} and \eqref{h-from-below} that
$$|u(x)-u(0)|\leq 2\e + \frac{6}{\delta_2^3} w_{\delta_2}(x).$$
Note that, by construction and the estimate \eqref{est-for-phi} for the function $\phi$, we have in $A$
$$w_{\delta_2}(x)\leq M_{\delta_2} x_n + \phi(x)\leq M_{\delta_2}\abs{x} + C\abs{x}^2\abs{\log \abs{x}}^2\leq 2M_{\delta_2}\abs{x}.$$

Therefore, choosing $\delta_{2}= \delta_{1}$ and recalling the choice of $M_{\delta_2}$, we get
\begin{equation}
\label{op-ineq}
|u(x)-u(0)|\leq 2\e + \frac{12 M_{\delta_2}}{\delta_2^3}\abs{x}= 2\e + \frac{C_{2}(n,\lambda, \Lambda)}{\delta_{2}^{3n}}|x| = 2\e + C_{2} \e^{-3n/\alpha}|x|
\end{equation}
for all $x,\e$ satisfying the following conditions
\begin{equation}
\label{xe-ineq}
|x|\leq \delta_{1}(\e):= \e^{1/\alpha},\quad \e\leq C_1(n,\lambda)^{\frac{-\alpha}{1-\alpha}} =: c_{1}.
\end{equation}
Finally, let us choose $\e = |x|^{\frac{\alpha}{\alpha + 3n}}.$
It satisfies the conditions in \eqref{xe-ineq}  if
$$|x|\leq \min\{c_{1}^{\frac{\alpha +3n}{\alpha}}, 1\} =: \delta.$$
Then, by \eqref{op-ineq}, we have 
$|u(x)-u(0)| \leq (2 + C_2) |x|^{\frac{\alpha}{\alpha + 3n}}\,$ for all $\, x\in \Omega\cap B_{\delta}(0)$.
\end{proof}
\subsection{Global approximation lemma}
\label{appro_sec}
In this subsection, we prove an approximation lemma that allows us to compare the solution $u$ to the linearized Monge-Amp\`ere equation 
$\mathcal{L}_{\phi} u =f$
to smooth solutions $h$ of linearized Monge-Amp\`ere equations $\mathcal{L}_{w}h =0$
associated with convex functions $w$ satisfying $\det D^2 w =1$.
 We will estimate the difference $u-h$ in terms of the
$L^{n}$-norms of $f$ and $\Phi-\calW$ where $\Phi =(\Phi^{ij})$ and   $\calW = (\calW^{ij})$ are 
the matrices of cofactors of $D^2\phi$ and $D^2w$, respectively. Therefore, in light of the global stability of cofactor matrices in Subsection~\ref{cofactor_sec}, $u$ is 
well-approximated by $h$ provided that $\det D^2\phi$ is close to $1$. This approximation lemma will play a key role in Section~\ref{density_sec}
where we use it to get power decay estimates for the distribution function of the second derivatives of $u$ that are more refined than those provided by
Proposition~\ref{power-decay-at-boundary}. 

Our approximation lemma, relevant for data of the type $\big(\Omega_{h}, \phi_{h}, S_{\phi_{h}}(0, 1)\big)$, states as follows. 

\begin{lemma}\label{explest} 
Assume $(\Omega, \phi, U)\in \mathcal{P}_{\frac{1}{2}, \frac{3}{2}, \rho, \kappa, \alpha}.$ Let $r := c^2/4$ where $c$ is as in Remark \ref{choosingc}.
Suppose that $u\in C(\overline U)\cap W^{2,n}_{loc}(U)$ is a  solution of 
$\,\Phi^{i j} u_{i j} = f\,$ in $U\cap B_{4r}$
with $$\|u\|_{L^{\infty}(U\cap B_{4r})} +\|u\|_{C^{2,\alpha}(\p U\cap B_{4r})} \leq 1. $$ 
Let $w$ be defined as in $(vii)$ of the definition of the class $\mathcal{P}$. 
Assume  $h$  is a solution of
\begin{equation*}
\left\{\begin{array}{rl}
\calW^{i j} h_{i j}  &= 0 \quad \mbox{ in }\quad  B_{2r}\cap U\\
h  &=u\quad\mbox{ on } \quad\partial (B_{2r}\cap U).
\end{array}\right.
\end{equation*}
Then, there exist $C>0$ and  $\gamma \in (0,1)$ depending only on $n, \rho$ and $\alpha$ such that  
\begin{equation}\label{boundary-C^2}
\|h\|_{C^{1,1}(\overline{ B_{r} \cap U})} 
\leq C,
\end{equation}
and if $\| \Phi - \calW\|_{L^n( B_{2r}\cap U)}\leq r^4$ then
\begin{align*}
\|u-h\|_{ L^\infty (B_{r}\cap U)}
&+\|f-\trace([\Phi-\calW ]D^2h)\|_{ L^n( B_{r}\cap U)}\\
&\leq C \left\{\big(1+ \|u\|_{C^{1/2}(\p U\cap B_{4r})} \big)\, \| \Phi - \calW\|_{L^n( B_{2r}\cap U)}^\gamma + \|f\|_{L^n(U\cap B_{4r})} \right\}.
\end{align*}
\end{lemma}
\begin{proof}
Observe first that by  $(vii)$ in the definition of the class $\mathcal{P}$, the 
following $C^{2,\alpha}$ and Pogorelov estimates hold
\begin{equation}
\label{Pogo-1}
\|\p U\cap B_{4r}\|_{C^{2,\alpha}} \leq c_{0}^{-1},~ \|w\|_{C^{2,\alpha}(\overline{U\cap B_{4r}})} \leq c_{0}^{-1},~c_{0}I_{n}\leq D^2 w\leq c_{0}^{-1} I_{n}~\text{ in }~ B_{4r}\cap U.
\end{equation}
Therefore,  $W^{ij} \p_{ij}$ is a uniformly elliptic differential operator with $C^{\alpha}$ coefficients. Hence, we can employ
 the standard boundary $C^{2,\alpha}$-estimates for linear uniformly elliptic equation and obtain \eqref{boundary-C^2} since
 \begin{equation*}
\|h\|_{C^{1,1}(\overline{ B_{r} \cap U})} \leq \|h\|_{C^{2,\alpha}(\overline{ B_{r} \cap U})}
\leq C(n, \rho, \alpha) \big( \|u\|_{L^\infty( B_{2r} \cap U)} + \|u\|_{C^{2,\alpha}(\p U\cap B_{2r})}\big) \leq  C(n, \rho,\alpha).
\end{equation*}
Next, since  $(\Omega, \phi, U)\in \mathcal{P}_{\frac{1}{2}, \frac{3}{2}, \rho, \kappa, \alpha}$, by Remark \ref{quad-rk}, the domain
$U$ and function $\phi$ satisfy \eqref{global-tang-int}, \eqref{eq_u} and \eqref{om_ass} and (\ref{eq:sep-half-domain}). Therefore, 
 it follows from Theorem~\ref{global-H} with $C^{1/2}$ boundary data that there exist constants $C>0$ and  $\beta\in (0,1)$ depending only on $n$ and $\rho$ such that
\begin{equation}\label{eq:boundary-holder}
\|u\|_{C^\beta(\overline{ B_{2r} \cap U})}
\leq C\Big(\|u\|_{L^\infty( B_{4r} \cap U)} 
+ \|u\|_{C^{1/2}(\partial U \cap B_{4r})} + \|f\|_{L^n(B_{4r} \cap U)}\Big)
\leq C \Theta,
\end{equation}
where
$$\Theta:=1 +  \|u\|_{C^{1/2}(\partial U \cap B_{4r})} + \|f\|_{L^n(B_{4r} \cap U)}.$$
 In view of \eqref{Pogo-1}, \eqref{eq:boundary-holder},
 and the standard global H\"older estimates for linear  uniformly 
elliptic equations (see  \cite[Corollary~9.29]{GiT}, \cite[Proposition~4.13]{CC} and \cite[Theorem~1.10]{Wi}), 
we can find constants $C>0$ and  $\beta' \in (0, \beta)$ depending only on $n, \rho$ and $\alpha$ such that
\begin{equation}\label{eq:global-holder-for-h}
\|h\|_{C^{\beta'}(\overline{ B_{2r} \cap U})}
\leq C\Big(\|h\|_{L^\infty(B_{2r} \cap U)} + \|u\|_{C^{\beta}(\partial (U\cap B_{2r}))} \Big)
\leq C \Theta.
\end{equation}
Now let $0<\delta <r$. Then we claim that 
\begin{equation}\label{dest}
\|u-h\|_{L^\infty(\partial (B_{2r-\delta}\cap U)) } \leq C\delta^{\beta'} \Theta,
\end{equation}
and
\begin{equation}\label{2orderest}
\|D^2 h\|_{L^\infty(B_{2r-\delta}\cap U) } \leq C \delta^{\beta' -2-\alpha} \Theta.
\end{equation}
To prove \eqref{dest}, we verify that $|(u-h)(x)|\leq C\delta^{\beta'} \Theta$ for all $x\in \partial (B_{2r-\delta}\cap U)$. Indeed, if 
$x\in \partial (B_{2r-\delta}\cap U)$ then we can find $y\in \partial(B_{2r}\cap U)$ such that $|x-y| \leq \delta$. 
Since $u-h =0$ on $ \partial (B_{2r}\cap U)$, we get from \eqref{eq:boundary-holder} and \eqref{eq:global-holder-for-h} that
\begin{align*}
|(u-h)(x)| 
= |(u-h)(x) - (u-h)(y)| 
\leq |u(x)-u(y)| + |h(x)-h(y)|
\leq C \delta^{\beta'} \Theta.
\end{align*}

To prove (\ref{2orderest}), let $x_0\in B_{2r-\delta}\cap U$. If $B_{\delta/2}(x_0)\subset B_{2r}\cap U$, then we can  apply interior $C^{1,1}$-estimates 
to $h-h(x_0)$ in $B_{\delta/2}(x_0)$ and use \eqref{eq:global-holder-for-h} to get
\begin{equation*}
\|D^2 h(x_0)\| \leq C \delta^{-2} \|h-h(x_0)\|_{L^\infty( B_{\frac{\delta}{2}}(x_0))} \leq C \delta^{\beta'-2} \Theta.
\end{equation*}

In case $B_{\delta/2}(x_0)\not\subset B_{2r}\cap U$, then there exists $z_0\in B_{2r-\delta}\cap \partial U\subset \partial \Omega$ such that $x_0\in B_{\delta/2}(z_0)$. Hence since $B_\delta(z_0)\cap U \subset B_{2r}\cap U$ and by applying boundary $C^{2,\alpha}$-estimates to $h-h(x_0)$ in $B_{\delta/2}(z_0)\cap U$ we obtain
\begin{align*}
 &\|D^2(h-h(x_0)) \|_{C^{\alpha}(B_{\frac{\delta}{2}}(z_0)\cap U)} \\
 &\leq C \delta^{-(2+\alpha)} \Big( \|h-h(x_0)\|_{L^\infty(B_{\delta}(z_0)\cap U)}
+\sum_{k=1}^{2} \delta^{k+\alpha}\|D^k( u-h(x_0))\|_{C^{\alpha}(\partial U\cap  B_{\delta}(z_0))}\Big)\\
&\leq C \delta^{-(2+\alpha)} \Big(\delta^{\beta'} \Theta
+\delta^{1+\alpha}\Big)\leq C \delta^{\beta' -2-\alpha} \Theta.
\end{align*}
It follows that 
$\|D^2 h(x_0)\|
\leq  C \delta^{\beta'-2-\alpha} \Theta$, and thus
\eqref{2orderest} is proved.

Having \eqref{dest} and \eqref{2orderest}, we now complete the proof of the lemma. Observe that $u-h\in W^{2,n}_{loc}( B_{2r}\cap U)$ is a  solution of
\begin{align*}
\Phi^{ij} (u-h)_{ij}=f - \Phi^{ij} h_{i j} = f- [ \Phi^{i j} - \calW^{ij}] h_{i j} =: F \quad \mbox{in}\quad  B_{2r}\cap U.
\end{align*}
The ABP estimate together with \eqref{dest} and \eqref{2orderest} gives 
\begin{align*}
&\|u-h\|_{L^\infty( B_{2r-\delta}\cap U) } +\|F\|_{L^n(B_{2r-\delta}\cap U)}\leq \|u-h\|_{L^\infty(\partial (B_{2r-\delta}\cap U)) }  + C_n \|F\|_{L^n(B_{2r-\delta}\cap U)} \\
&\leq \|u-h\|_{L^\infty(\partial(B_{2r-\delta}\cap U)) } +C_n \|D^2 h\|_{L^\infty(B_{2r-\delta}\cap U) } \| \Phi - \calW\|_{L^n(B_{2r}\cap U)}  + C_n \|f\|_{L^n(U\cap B_{2r})} \\ 
&\leq C  \Big( \delta^{\beta'} + \delta^{\beta'-2-\alpha}\| \Phi - \calW\|_{L^n(B_{2r}\cap U)}\Big) \, \Theta + C_n \|f\|_{L^n(U\cap B_{2r})}.
\end{align*}
If $\| \Phi - \calW\|_{L^n(B_{2r}\cap U)}\leq r^4$ then by taking $\delta = \| \Phi - \calW\|_{L^n(B_{2r}\cap U)}^{\frac{1}{2 +\alpha}},$ we obtain the desired inequality with
$\gamma =\beta'/(2+\alpha)$ since
\begin{align*}
\|u-h\|_{L^\infty(B_{r}\cap U) } +\|F\|_{L^n(B_{r}\cap U)}
&\leq C \| \Phi - \calW\|_{L^n(B_{2r}\cap U)}^{\frac{\beta'}{2+\alpha}} \Theta  + C_n \|f\|_{L^n(U\cap B_{2r})}\\
&\leq C  \Big\{ \big( 1 + \|u\|_{C^{1/2}(\p U\cap B_{4r})}\big)\, \| \Phi - \calW\|_{L^n(B_{2r}\cap U)}^{\frac{\beta'}{2+\alpha}} + \|f\|_{L^n( U\cap B_{4r})}\Big\}.
\end{align*}

\end{proof}

We end this subsection with a result allowing us to estimate the measure of the set where the quasi distance generated by $\phi$ is bounded from below by certain multiple of the Euclidean
distance. This set, when restricted to sections of $\phi$, has almost full measure if the Monge-Amp\`ere measure
$\det D^2\phi$ is close to a constant. Its precise statement is as follows.

\begin{lemma}\label{lm:tangent-paraboloid}
Assume that $(\Omega, \phi, U)\in \mathcal{P}_{1-\e, 1+ \e, \rho, \kappa, \alpha}$ where $0<\e<1/2$. Define
\begin{equation}\label{Def-A-sigma}
A_{\sigma} := \left\{\tilde x\in U: \phi(x)\geq
\phi(\tilde x) + \nabla \phi(\tilde x)\cdot (x-\tilde x) + \frac{\sigma}{2}\, |x-\tilde x|^2, \quad \forall x\in B_{c^2} \cap U\right\}.
\end{equation}
Then there exist $\sigma =\sigma(n,\rho,\alpha)>0$ and $C= C(n, \rho, \alpha, \kappa)>0$ such that
\begin{equation*}
\big|  S_{\phi}(0, c^9)\setminus A_\sigma\big| \leq C \eps^{1/3n}\, | S_{\phi}(0, c^9)|.
\end{equation*}
\end{lemma}
\begin{proof}
We first note that (\ref{small-sec}) implies
\begin{equation}| S_{\phi}(0, c^9)| \geq |B_{c^6}\cap U|\geq C.
 \label{volc}
\end{equation}
Let $w$ be defined as in $(vii)$ in the definition of the class $\mathcal{P}$. Then the following
boundary Pogorelov estimates hold
\begin{equation}
\label{Pogo-2}
c_{0}I_{n}\leq D^2 w\leq c_{0}^{-1}I_{n}~\text{ in }~ B_{c^2}\cap U.
\end{equation}
Let $\Gamma$ be the convex envelope of $\phi-\frac{w}{2}$ in $U\cap B_{c^2}$. We claim that there exists 
$C>0$ depending only on $n$, $\rho, \alpha$ and $\kappa$ such that 
\begin{equation}
\label{contact-set}
\big| \{\Gamma = \phi-\frac{w}{2}\}\cap  S_{\phi}(0, c^9) \big| \geq (1-C\eps^{1/3n})| S_{\phi}(0, c^9)|.
\end{equation}
Assume this claim for a moment. Then by using \eqref{Pogo-2} and arguing as in the proof of \cite[Theorem 6.1.1]{G}, we 
obtain the desired conclusion. For completeness, we include the proof.

Let the contact set be $$\mathcal{C} :=\Big\{x\in U\cap B_{c^2}: \Gamma(x)= \phi(x)-\frac{w(x)}{2}\Big\}.$$ We assert that for $\sigma := c_0/2$, we have
$$\mathcal{C}\cap S_{\phi}(0, c^9)\subset A_{\sigma}\cap S_{\phi}(0, c^9).$$
It then follows from \eqref{contact-set} that
$$|S_{\phi}(0, c^9)\setminus A_{\sigma}|\leq |S_{\phi}(0, c^9)\setminus \mathcal{C}|\leq C \e^{1/3n}|S_{\phi}(0, c^9)|.$$
We now proceed with the proof of the claim. Let $x_0\in \mathcal{C}\cap S_{\phi}(0, c^9)$, and let $l_{x_0}$ be a supporting hyperplane to
$\Gamma$ at $x_0$.  Since $x_0\in\mathcal{C}$, we have 
$l_{x_0}(x_0) =\phi(x_0) - \frac{1}{2}w(x_0)$
and
\begin{equation}
 \label{contact1}
 \phi(x) \geq l_{x_0}(x) +\frac{w(x)}{2}\quad \mbox{for all}\quad x\in U\cap B_{c^2}.
\end{equation}
On the other hand, if $x\in U\cap B_{c^2}$ then the Taylor formula and the first inequality in the Pogorelov estimates \eqref{Pogo-2} give
\begin{eqnarray*}
 w(x)-w(x_0)-\nabla w(x_0)\cdot (x-x_0) &=&\int_{0}^{1} t \int_{0}^{1} \langle D^2 w(x_0 +\theta t (x-x_0))\cdot (x-x_0), x- x_0\rangle d\theta dt\\
 &\geq& \int_{0}^{1} t c_0\abs{x-x_0}^2 dt = \frac{c_0}{2}\abs{x-x_0}^2.
\end{eqnarray*}
Combining this with \eqref{contact1}, we deduce that 
$$\phi(x)\geq l(x) + \frac{c_0}{4}\abs{x-x_0}^2 \quad \forall x\in U\cap B_{c^2},$$
where $l(x)$ is the supporting hyperplane to $\phi$ at $x_0$ in $U\cap B_{c^2}$ given by
$$l(x) := l_{x_0}(x) + \frac{1}{2} w(x_0) + \frac{1}{2} \nabla w(x_0) \cdot(x-x_0).$$
Therefore $x_0\in A_{\sigma}$ with $\sigma = c_0/2$, proving the assertion.

It remains to prove \eqref{contact-set}. The idea is to compare the image of the gradient 
mappings of convex functions which are close in $L^{\infty}$-norms. This idea goes back to Caffarelli (see \cite[Lemma 2]{C3} and also \cite[Lemma 6.2]{LTW}). 
Since our setting near the boundary is a bit different, we sketch the proof. 

By \eqref{volc}, it suffices to consider the case $\e \ll 1$. By the maximum principle (\cite[Lemma 3.1]{H09}), we have
\begin{eqnarray*}
\|\phi-w\|_{L^{\infty}(U)}\leq C_{n} \diam(U)
\|(\det D^2\phi )^{1/n}-1\|_{L^{n}(U)} \leq C\, \e^{1/n}\equiv \bar{\e}.
\end{eqnarray*}
Therefore,
$$\frac{1}{2}w-\bar{\e}\leq\phi-\frac{w}{2}\leq\frac{1}{2} w +\bar{\e}\quad \text{in}\quad U\cap B_{c^2}$$
and since $w$ is convex, we have
$$\frac{1}{2}w-\bar{\e}\leq\Gamma\leq\frac{1}{2} w +\bar{\e}\quad \text{in}\quad U\cap B_{c^2}.$$
Let
$$V_{1}=\Big\{x\in U\cap B_{c^2}: \,  \dist \big(x, \partial(U\cap B_{c^2})\big)>\delta\Big\}.$$ Then, 
using \eqref{Pogo-2}, we will show that for $1\gg \delta> \bar{\e}$ to be chosen later, we have
\begin{equation}
\big| \{\Gamma = \phi-\frac{w}{2}\}\cap V_1 \big| \geq (1-C\delta)|V_1|
\label{contact2}
\end{equation}
for some
$C$ depends on $n$, $\rho, \alpha$ and $\kappa$.
Indeed, let
$$V_2=\Big\{x\in U\cap B_{c^2}: \,  \dist \big(x, \partial(U\cap B_{c^2})\big)>2\delta\Big\}.$$
For $x_0\in V_2$, consider 
$$v^*(x)  :=\frac{1}{2} w(x) -\bar{\e} + \delta (r^2- \abs{x-x_0}^2)$$
where
$$\delta= 2\sqrt[3]{\bar{\e}}\sim \e^{1/3n},~\delta>r>\sqrt{\frac{2\bar{\e}}{\delta}}.$$
Then $B(x_0, r)\subset V_1$ and 
\begin{equation*}v^*\leq \Gamma~\text{ on }~ \p B(x_0, r) \quad \text{and}\quad v^* \geq \Gamma~\text{ in }~ B(x_0, r- \frac{2\bar{\e}}{\delta r}).
\end{equation*}
It follows that $\nabla v^*(B(x_0, r- \frac{2\bar{\e}}{\delta r}) \subset \nabla \Gamma (B(x_0, r)).$
Hence
\begin{equation}\label{eq:subdiff-comparison}
 \nabla v^*(V_2)\subset \nabla \Gamma (V_1).
\end{equation}
From the $C^2$ bound on $w$ in (\ref{Pogo-2}), we have 
$$D^2 v^*= \frac{1}{2}D^2 w -2\delta I_n = \frac{1}{2} \big(1 - 4 c_0^{-1}\delta\big) D^2 w + 2\delta \big(c_0^{-1} D^2 w- I_n\big)\geq 
\frac{1}{2} \big(1 - 4 c_0^{-1}\delta\big) D^2 w.$$
Therefore, using $\det D^2 w=1$, we obtain
\begin{equation}
 \label{vol1}
 |\nabla v^* (V_2)|=\int_{V_2}\det D^2 v^* \geq (\frac{1}{2^{n}}- C_{1}\delta)|V_2|.
\end{equation}
Next, as $\Gamma$ is convex with $\Gamma\in C^{1,1}(U\cap B_{c^2})$ and $\det D^2\Gamma =0$  a.e. outside $\mathcal{C}$, we have
\begin{equation}\label{eq:subdffGamma}
 |\nabla\Gamma (V_1)|= |\nabla\Gamma (V_1\cap\mathcal{C})| = \int_{V_1\cap\mathcal{C}}\det D^2\Gamma.
\end{equation}
We now estimate $\det D^2 \Gamma$ from above. For this, observe that 
for any $x\in\mathcal{C}$, the function $\phi-\frac{1}{2}w-\Gamma$
attains its local minimum value $0$ at $x$. Hence, 
$$D^2\Gamma (x) \leq D^2 (\phi-\frac{1}{2}w)(x)$$
at any twice differentiable point of $\Gamma$ and $\phi$. Therefore, this inequality holds for a.e $x\in \mathcal{C}$ by Aleksandrov theorem. 
 Note that for symmetric, nonnegative matrices $A$ and $B$, we have
$$(\det (A+B))^{1/n}\geq (\det A)^{1/n} + (\det B)^{1/n}.$$
Thus, for a.e $x\in\mathcal{C}$, we have
\begin{eqnarray*}(\det D^2\Gamma(x))^{1/n} &\leq& (\det D^2 (\phi-\frac{1}{2}w)(x))^{1/n} \leq (\det D^2\phi)^{1/n}-(\det D^2 (\frac{1}{2}w)(x))^{1/n}\\
&\leq& (1+\e)^{1/n}-\frac{1}{2}\leq 
\frac{1}{2} +\frac{\e}{n}.
\end{eqnarray*}
Combining with \eqref{eq:subdffGamma} gives
$$|\nabla \Gamma(V_1)|\leq(\frac{1}{2^n} + C_2\e)|V_1\cap \mathcal{C}|.$$
We infer from this, \eqref{eq:subdiff-comparison} and \eqref{vol1} that
\begin{equation}|V_1\cap\mathcal{C}|\geq \frac{1-2^n C_1\delta}{1 + 2^{n}C_2\e}\abs{V_2}\geq (1-C_3 \delta)\abs{V_2}
 \label{vol2}
\end{equation}
for $\e\leq \e_0$ with $\e_0$ is a small universal constant.

By $(iv)$ in the definition of the class $\mathcal{P}$, we have $\|B_{c^2} \cap \partial U\|_{C^{1,1}}\leq \kappa.$ Consequently, 
 \begin{align*}
  \big|(U\cap B_{c^2}) \backslash V_2 \big| \leq C_4\delta\quad \text{and}\quad \abs{V_2}\geq \abs{V_1}- C_4\delta
 \end{align*}
for some $C_4>0$ depending only on $n, \rho$ and $\kappa$. Combining the above inequalities with \eqref{vol2}, we easily obtain \eqref{contact2}.

It follows from \eqref{contact2}, the inclusion  $\{\Gamma<\phi-\frac{w}{2}\}\subset U\cap B_{c^2}$  and  \eqref{volc} that
\begin{eqnarray*}
|\{\Gamma<\phi-\frac{w}{2}\}\cap  S_{\phi}(0, c^9)| &\leq & |\{\Gamma<\phi-\frac{w}{2}\}\cap V_1| + |(U\cap B_{c^2}) \backslash V_1|
\leq  C\delta|V_1| + C_4\delta \\ &\leq& C_5\delta\leq C\delta | S_{\phi}(0, c^9)| \leq C \e^{1/3n} | S_{\phi}(0, c^9)|. \nonumber
\end{eqnarray*}
This gives the claim \eqref{contact-set} and the proof  is complete.
\end{proof}

\section{\bf Density and Global $W^{2,p}$ Estimates }
\label{density_sec}

In this section we will prove global $W^{2,p}$ estimates for solutions to the linearized Monge-Amp\`ere equations as stated in the introduction. The key tools are density estimates
and a covering lemma.
\subsection{Density estimates}\label{sub:improveddensity}
In this subsection, by using the approximation lemma
in Subsection \ref{appro_sec} together with the stability of cofactor matrices established in Subsection \ref{cofactor_sec}, we improve density estimates obtained in 
Section \ref{small_decay_sec} when the Monge-Amp\`ere measure
$\det D^{2}\phi$ is close to $1$.

Our first lemma improves the power decay estimates in Proposition \ref{power-decay-at-boundary} which say that for $(\Omega, \phi, U)\in \mathcal{P}_{\lambda, \Lambda, \rho, \kappa, \ast}$, the quantity 
$|S_{\phi}(0, r)\setminus G_{N}(u, \Omega) |$ decays like $CN^{-\tau}$. Here, we improve $C$ by roughly a factor of
$\| \Phi - \calW\|_{L^n(U)}
+ \Big(\fint_{U} |f|^n ~dx\Big)^{\frac{1}{ n}}$ when $\ast$ is replaced by $\alpha$, $\lambda$ and $\Lambda$ are close to $1$, and $W$ is the matrix of cofactors
of $D^2 w$ of the solution to the Monge-Amp\`ere equation $\det D^2 w=1$ with the same boundary values as $\phi$. The precise statement is as follows.

\begin{lemma}\label{lm:acceleration}
Assume $(\Omega, \phi, U)\in \mathcal{P}_{1-\e, 1+ \e, \rho, \kappa, \alpha}$ where $0<\epsilon<1/2$. Let $r= c^2/4$.
Suppose $u\in  C(\Omega)\cap  C^1(U)\cap W^{2, n}_{loc}(U)$
is a solution of  
$\calL_\phi u=f$ in $U$ that satisfies
$$\|u\|_{L^{\infty}(U)} + \|u\|_{C^{2,\alpha}(\p U\cap B_{4r})} \leq 1, $$
and has at most quadratic growth  in the sense that
\begin{equation}
 \label{u-far}
|u(x)|\leq C^*\, \big[ 1+ d(x, x_0)^2\big] ~\text{ in }~ \Omega\setminus U~ \text{ for some }~ x_0\in B_{r/2}\cap U.
\end{equation}
Then there exist $\tau =\tau(n,\rho)>0$ 
and $N_0 = N_0(C^*, n,\rho, \alpha)>0$  such that for $N\geq N_0$ we have
\begin{equation*}
|G_{N}(u,\Omega) \cap   S_{\phi}(0, c^9)|\geq \left\{1- C\big( N^{-\tau} \delta_0^\tau +\epsilon^{1/3n} \big)\right\} \, | S_{\phi}(0, c^9)|
\end{equation*}
provided that $\| \Phi - \calW\|_{L^n(B_{2r} \cap U)} \leq r^4$. Here $C= C(n, \rho, \alpha,\kappa)>0$, $\calW,\, \gamma$ are from 
Lemma~\ref{explest}, and
\[
\delta_0 := \big( 1 + \|u\|_{C^{1/2}(\p U\cap B_{4r})}\big)\, \| \Phi - \calW\|_{L^n(B_{2r} \cap U)}^\gamma 
+ \Big(\fint_{U} |f|^n ~dx\Big)^{\frac{1}{ n}}.
\]
\end{lemma}
\begin{proof}
Let $h$ be the solution of
\begin{equation*}
\calW^{i j} h_{i j}  = 0~ \mbox{ in} \quad B_{2r} \cap U,\quad\mbox{and}\quad
h  =u~ \mbox{ on} \quad\partial ( B_{2r} \cap U).
\end{equation*}
By Lemma~\ref{explest} and since $U\subset B_{k^{-1}}$, there exists $C_0$ depending only on $n, \rho$ and $\alpha$ such that
\begin{align}
&\|h\|_{C^{1,1}(\overline{ B_{r} \cap U})}
\leq  C_0,\label{regularity-alpha}\\
&\|u-h\|_{L^\infty(\overline{ B_{r} \cap U})} + 
\| f-\trace([\Phi-\calW] D^2h)\|_{L^n(\overline{ B_{r} \cap U})}
\leq C_{0} \,\delta_0 =:\delta_0'.\label{closeness-alpha}
\end{align}
We now consider $h|_{B_{r} \cap U}$ and then extend $h$ outside $B_{r} \cap U$ continuously such that
\begin{equation*}
\left\{\begin{array}{rl}
&h(x)  = u(x) \quad \forall  x\in \Omega\setminus (B_{2r} \cap U),\\
&\|u-h\|_{L^\infty(\Omega)}  =\|u-h\|_{L^\infty(B_{r} \cap U)}.
\end{array}
\right.
\end{equation*}
The maximum principle gives
$\|h\|_{L^\infty(B_{r} \cap U)} \leq \|u\|_{L^\infty(U)}\leq 1,$ and thus 
\begin{equation}\label{h-above-below}
u(x) -2 \leq h(x) \leq u(x) +2 \quad \mbox{for all}\quad x\in \Omega.
\end{equation}
We claim that if $N\geq N_0$, then
\begin{equation}\label{h-is-good}
(B_{\frac{r}{2}} \cap U)\cap A_{\sigma}\subset G_N(h,\Omega)
\end{equation}
 where $\sigma=\sigma(n,\rho,\alpha)>0$ is the constant given by Lemma~\ref{lm:tangent-paraboloid}  and   the set $A_{\sigma} $ is defined by \eqref{Def-A-sigma}.\\

 Indeed, let $\bar x\in (B_{\frac{r}{2}} \cap U)\cap A_{\sigma}$.
By \eqref{regularity-alpha} we have  $$ \left | h(x) -[ h (\bar x) +
\nabla  h (\bar x)\cdot ( x- \bar x)]\right|\leq C_0 |x- \bar x|^2
\text{ for all } x\in \overline{B_{r} \cap U},$$ and since 
$\bar x\in  A_{\sigma}$ 
\begin{equation}\label{d-below}
d(x, \bar x)^2 = \phi(x) -[\phi(\bar x) +\nabla \phi(\bar x)\cdot( x-\bar x)]
\geq \frac{\sigma}{2} |x-\bar x|^2\quad\forall x\in B_{4r} \cap U.
\end{equation}
Therefore,
\begin{equation}\label{h-step2}
\big| h(x) -[ h (\bar x) +\nabla h (\bar x)\cdot ( x- \bar x)] \big|\leq \frac{2 C_0}{\sigma} \, d(x, \bar x)^2\quad\forall x \in \overline{B_{r} \cap U}.
\end{equation}
We next show that by increasing the constant on the right hand side of \eqref{h-step2}, that the resulting inequality holds for all $x$ in $\Omega$. 

To see this,  we first observe that by the maximum principle
$\max_{U}\phi = \max_{\p U}\phi =1$
and by the gradient estimates $(v)$ in the definition of the class $\mathcal{P}$ and $x_0\in U\cap B_{r/2}$, we have 
\begin{align}\label{d-via-d}
d(x, x_0)^2  &= d(x, \bar x)^2 
+ [\phi(\bar x) -\phi(x_0) -\langle  \nabla \phi(x_0), \bar x -x_0\rangle] +  \langle  \nabla \phi(\bar x) -\nabla \phi(x_0), x-\bar x\rangle\\
&\leq d(x, \bar x)^2 + C_1 (1 +  |x -\bar x| ) \qquad \mbox{for all }\quad x\in\Omega\nonumber
\end{align}
for some universal $C_1$ depending only on $n$ and $\rho$.

Next, we observe that if $c_{1}= \sigma r/4$ then 
\begin{equation}\label{d-below-linear}
 d(x, \bar x)^2 \geq c_1 |x-\bar x| \quad \forall x\in \overline{\Omega}\setminus \overline{ B_{r} \cap U}.
\end{equation}
Indeed, by \eqref{d-below}  and the fact that $\bar x\in B_{\frac{r}{2}} \cap U$, the above inequality holds for all 
$x\in  U\cap \partial B_{r}$. Now for $x\in \overline{\Omega}\setminus \overline{B_{r} \cap U}$ we can choose 
$\hat x \in  U \cap \partial B_{r}$ and $\lambda\in (0,1)$ satisfying $\hat x = \lambda x + (1-\lambda) \bar x$. Then since $d(\hat x, \bar x)^2 \geq 
c_1 |\hat x -\bar x|$ and the function $z\mapsto d(z,\bar x)^2$ is convex, we obtain 
\[
\lambda d(x, \bar x)^2 + (1-\lambda) d(\bar x, \bar x)^2
\geq c_1 |\lambda x + (1-\lambda) \bar x -\bar x| = c_1\lambda |x-\bar x|
\]
which gives $d(x, \bar x)^2 \geq c_1 |x-\bar x|$ and hence \eqref{d-below-linear} is proved.

We are ready to show that \eqref{h-step2} holds for all $x\in\Omega$ but with a bigger constant on 
the right hand side. Let $x\in \Omega \setminus \overline{B_{r} \cap U}$. Then, recalling $\bar x\in B_{\frac{r}{2}} \cap U$ and by (\ref{d-below-linear}),
we have
$$d(x, \bar{x})^2 \geq c_{1}r/2 =: c_2.$$
We can estimate using \eqref{regularity-alpha} and  \eqref{h-above-below},
\begin{equation} 
\left | h(x) -[ h (\bar x) +
\nabla  h (\bar x)\cdot ( x- \bar x)]\right|
\leq |h(x) - h(\bar x)| + C_0 |x-\bar x|
\leq|u(x)| + C_0 ( |x-\bar x|+ 1).
\label{dh-est}
\end{equation}
Consider the following cases:

{\bf Case 1:} $x\in \overline{U}\setminus \overline{B_{r} \cap U}$. 
Using (\ref{dh-est}) and the above lower bound for $d(x, \bar{x})^2$, we obtain
$$\left| h(x) -[ h (\bar x) + \nabla  h (\bar x)\cdot ( x- \bar x) ]\right| \leq 1 + C_0 ( |x-\bar x|+ 1)\leq 1 + C_0(2k^{-1}+ 1)\leq C_{2} \, d(x,\bar{x})^2.$$

{\bf Case 2:} $x\in \Omega\setminus  \overline{U}$. Using \eqref{dh-est}, \eqref{u-far}, \eqref{d-via-d}, \eqref{d-below-linear} and the bound  $d(x, \bar{x})^2\geq c_2$, we find that 
\begin{eqnarray*}
\left | h(x) -[ h (\bar x) +
\nabla  h (\bar x)\cdot ( x- \bar x)]\right|
&\leq&  C^*\,[ 1+ d(x, x_0)^2] +C_0 ( |x-\bar x|+ 1) \\
&\leq &  C^*\,d(x, \bar x)^2 +C_3 ( |x-\bar x|+ 1)\leq C_4 \, d(x, \bar x)^2.
\end{eqnarray*}
Therefore if we choose 
\[N_0 := \max{\Big\{\frac{ 4 C_0}{\sigma}, 2 C_2, 2 C_4 \Big\}},
\]
then it follows from the above considerations and \eqref{h-step2} that 
\[
\left | h(x) -[ h (\bar x) +
\nabla  h (\bar x)\cdot ( x- \bar x)]\right|\leq \frac{N_0}{2} \, d(x, \bar x)^2
\quad\mbox{for all}\quad x \in  \Omega.
\]
This means $\bar x\in G_{N_0}(h,\Omega)\subset G_{N}(h,\Omega)$ for all $N\geq N_0$. Thus claim \eqref{h-is-good} is proved.

Next let
\[
u'(x) := \frac{(u-h)(x)}{\delta_0'}, \quad \mbox{for}\quad x\in\Omega.
\]
We infer from \eqref{closeness-alpha} and the way $h$ was initially defined and extended that
\begin{align*}
& \|u'\|_{L^\infty(\Omega)} =\frac{1}{\delta_0'} \|u-h\|_{L^\infty(B_{r} \cap U)}\leq 1,\\
& \calL_\phi u' =\frac{1}{\delta_0'} [\calL_\phi u -\calL_\phi h ]
 =\frac{1}{\delta_0'} \Big[f-\trace([\Phi -\calW] D^2 h) \Big]
=: f'(x) 
\qquad \mbox{ in}\quad B_{r} \cap U.
\end{align*}
Notice that $\|f'\|_{L^n(B_{r} \cap U)}\leq 1$ by \eqref{closeness-alpha}.  Thus   we can  apply Proposition~\ref{power-decay-at-boundary} to get
\begin{equation*}
\big| S_{\phi}(0, c^9)\setminus G_{\frac{N}{\delta_0'}}(u',\Omega)\big| \leq
C \big(\frac{\delta_0'}{N}\Big)^\tau \, | S_{\phi}(0, c^9)|.
\end{equation*}
As 
$ G_{\frac{N}{\delta_0'}}(u',\Omega) =  G_{N}(u-h,\Omega)$, we then conclude 
\begin{align*}
| S_{\phi}(0, c^9)| - |G_{N}(u -h,\Omega) \cap   S_{\phi}(0, c^9)|\leq 
 C \left(\frac{\delta_0}{N} \right)^{\tau}  ~ | S_{\phi}(0, c^9)|
\end{align*}
yielding
\begin{align*}
 \left\{ 1- 
 C  \big(\frac{\delta_0}{N} \big)^{\tau}  \right\}~ | S_{\phi}(0, c^9)|
 &\leq |G_{N}(u -h,\Omega) \cap   S_{\phi}(0, c^9)|\\
 & \leq |G_{N}(u -h,\Omega) \cap   S_{\phi}(0, c^9)\cap A_{\sigma}| +\left| S_{\phi}(0, c^9)\setminus A_{\sigma}\right|\\
 & \leq |G_{N}(u -h,\Omega) \cap   S_{\phi}(0, c^9)\cap A_{\sigma}| + C \epsilon^{1/3n}   \, | S_{\phi}(0, c^9)|,
\end{align*}
where the last inequality is by Lemma~\ref{lm:tangent-paraboloid}. Consequently,
\begin{equation}\label{Est-12}
 |G_{N}(u -h,\Omega) \cap   S_{\phi}(0, c^9) \cap A_{\sigma}| \geq \left\{ 1- 
 C \Big[ \big(\frac{\delta_0}{N} \big)^{\tau} +  \epsilon^{1/3n} \Big] \right\}~ | S_{\phi}(0, c^9)|.
\end{equation}
Next observe that $ G_{N}(u -h,\Omega) \cap  S_{\phi}(0, c^9)\cap A_{\sigma}\subset G_{N}(u -h,\Omega) \cap  G_N(h,\Omega)$ by \eqref{h-is-good}. Therefore,
\begin{equation*}
G_{N}(u -h,\Omega) \cap   S_{\phi}(0, c^9) \cap A_{\sigma} \subset 
G_{2 N}(u,\Omega ) \cap   S_{\phi}(0, c^9)
\end{equation*}
which together with \eqref{Est-12}  gives the conclusion of the lemma. 
\end{proof}

Having the improved decay estimates in Lemma \ref{lm:acceleration}, we can now proceed with density estimates when 
$\det D^2\phi$ is close to a constant. Our next lemma is concerned
with second derivative estimates for solutions to $\mathcal{L}_{\phi} u=f$. It roughly says that in each section $S_{\phi}(x, t)$ with small height $t$, we can find
a very large portion (as close to the full measure as we want) where $u$ has second derivatives bounded in a controllable manner. The bound on $D^2 u$ is made more precise by using
the openings of the quasi paraboloids that touch $u$ from below and above. So far, we have no a priori information on the boundedness of $D^2u$.
However, we can still hope for a bound of order $\frac{1}{t}$ for $|D^2u|$ in $S_{\phi}(x, t)$ as explained in Subsection~\ref{rescale-sec} using an $L^{\infty}$-norm
rescaling of our solution. This heuristic idea explains the factor $\frac{N}{t}$ in the estimate \label{localized-est} of Lemma~\ref{lm:improved-density-I} and the way the solution is rescaled
in the proof.

\begin{lemma}\label{lm:improved-density-I}
Assume $\Omega$ satisfies \eqref{global-tang-int} and $\phi\in C^{0,1}(\overline{\Omega})$ is a convex function satisfying \eqref{global-sep} and
\[1-\epsilon \leq \det D^2\phi \leq
1+\epsilon \quad\mbox{in}\quad \Omega.
\]
Assume in addition that $\partial \Omega\in C^{2,\alpha}$ and $\phi \in C^{2,\alpha}(\partial\Omega)$ for some $\alpha\in (0,1)$. Let $u\in C^1(\Omega) \cap W^{2,n}_{loc}(\Omega)$ be a solution of
$\calL_\phi u=f$ in $\Omega$ with  $u=0$ on $\partial\Omega$ and $\|u\|_{L^\infty(\Omega)}\leq 1$. 
Let $0<\epsilon_0<1$. 
Then there exists $\epsilon>0$ depending only on $\epsilon_0, n, \rho$ and $\alpha$ such that 
for any $x\in \overline{\Omega}$ and $t\leq c_1$  we have
\begin{equation}\label{localized-est}
 \left| G_{\frac{N}{t}}(u,\Omega) \cap S_\phi(x, t)\right|
 \geq \left\{1-\epsilon_0 -C \big(\frac{\sqrt{t}}{N}\big)^\tau \|f\|_{L^n(\Omega)}^\tau\right\}~ \left|S_\phi(x, t)\right|\quad \forall N\geq N_1.
 \end{equation}
Here $\tau =\tau(n,\rho)$;  $C$ and $N_1$ depend only on $n, \rho$ and $\alpha$; $c_1>0$ is  small depending only on  $n, \rho, \alpha$,
$\|\partial\Omega\|_{C^{2,\alpha}}$ and $\|\phi\|_{C^{2,\alpha}(\partial\Omega)}$.
\end{lemma}
\begin{proof}
If $\e$ is small then by the global $W^{2,p}$ estimates for solutions to the Monge-Amp\`ere equations \cite[Theorem 1.2]{S3}, we have $\phi\in W^{2, 2n}(\Omega)$ and hence 
$\phi\in C^{1}(\overline{\Omega})$.

Let us first consider the case $x\in \partial\Omega$. We can assume that $x=0$, $\phi(0)=0$ and $\nabla \phi(0)=0.$
By the Localization Theorem \ref{main_loc}, we have 
\[
k E_{t} \cap \overline\Omega\subset S_{\phi}(0, t) \subset k^{-1} E_{t}\cap \overline\Omega, 
\]
where $E_{t} := A_{t}^{-1} B_{t^{1/2}}$ with $ A_{t} x= x-\tau_{t} x_n$ and
\[
\tau_{t}\cdot e_n =0, \quad  \quad\|A_{t}^{-1}\|,\, \|A_{t}\|\leq k^{-1} |\log{t}|.
\]
We now define the rescaled domains $\Omega_{t}, U_{t}$ and rescaled functions $\phi_{t}$ and
$u_{t}$ as in Subsection~\ref{rescale-sec} {\it that preserve the $L^{\infty}$-norm of $u$}. We have
\begin{align*}
\mathcal{L}_{\phi_{t}} u_{t}(y)=  tf(T^{-1} y) =: f_{t}(y)
\end{align*}
where $T:=t^{-1/2} A_{t}$ and
 $$\| u_{t} \|_{L^\infty(\Omega_{t})} =\|u \|_{L^\infty(\Omega)}\leq 1, \quad u_{t} = 0 \quad\mbox{on}\quad \partial U_{t} \cap B_k.$$
Moreover, we have from Proposition~\ref{P-prop} that
$$(\Omega_{t}, \phi_{t}, U_{t})\in \mathcal{P}_{1-\e, 1+\e, \rho, Ct^{1/2}, \alpha}
\subset  \mathcal{P}_{1-\e, 1+\e, \rho, 1, \alpha}
$$
 if $t\leq \tilde{c}$, where $\tilde{c}>0$ is a small constant depending only on $n, \rho, \alpha$,
$\|\partial\Omega\|_{C^{2,\alpha}}$ and $\|\phi\|_{C^{2,\alpha}(\partial\Omega)}$.

 Now, applying Lemma~\ref{lm:acceleration} with $C^*=1$, we obtain
\begin{equation*}
|G_{N}(u_{t}, \Omega_{t}, \phi_t) \cap  S_{\phi_t}(0, c^9) |\geq \left\{1- C\big(N^{-\tau} \delta_0^{\tau} +
\epsilon^{1/3n} \big)\right\} \, | S_{\phi_t}(0, c^9)|
\end{equation*}
for any $N\geq N_0=N_0(n,\rho, \alpha)$. Here
\begin{equation}\label{delta-zero}
\delta_0 := \| \Phi_t - \calW_t\|_{L^n(B_{\frac{c^2}{2}} \cap U_t)}^\gamma 
+ \Big(\fint_{U_{t}} | f_{t}|^n ~dy\Big)^{\frac{1}{ n}}, 
\end{equation}
$\gamma$ is given by Lemma~\ref{explest}, $w_{t}$ is the function in $(vii)$ in the definition of the class
$\mathcal{P}$ associated with the triple $(\Omega_{t}, \phi_{t}, U_{t})$ and
 $\calW_{t}$ is the cofactor matrix of $D^2 w_{t}$.  This together with the stability of cofactor matrices in Proposition~\ref{global-convergence} 
implies  the existence of $\e= \e(\e_0, n, \rho, \alpha)>0$ 
such that for $r:=c^9$, we have
\begin{align*}\label{G-tilde-u}
 |G_{N}(u_{t}, \Omega_{t}, \phi_t) \cap  S_{\phi_t}(0, r)|
 &\geq \left\{1-\epsilon_0 \beta-C N^{-\tau}\Big(\fint_{U_{t}} |f_{t}|^n ~dy\Big)^{\frac{\tau}{n}}\right\}~ |S_{\phi_t}(0, r)|\\
 &= \left\{1-\epsilon_0\beta -C \big(\frac{t}{N}\big)^\tau \Big(\fint_{S_{\phi}(0, t)} | f|^n ~dx\Big)^{\frac{\tau}{n}} 
\right\}~ |S_{\phi_t}(0, r)|,
 \end{align*}
 where 
 $\beta=\beta(n,\rho)<1$ is a universal  constant  to be chosen later.

As $S_{\phi_t}(0, r) =T(S_{\phi}(0, rt))$, it is easy to see that for $G_{N}(u,\Omega,\phi) = G_{N}(u,\Omega)$,
\begin{align*}
&G_{N}( u_{t}, \Omega_{t}, \phi_t) \cap S_{\phi_t}(0, r)
=  T \Big( G_{\frac{N}{t}}(u, \Omega, \phi) \cap S_{\phi}(0, rt)\Big).
\end{align*}
Therefore  we conclude that
\begin{align*}
 \left| T \Big( G_{\frac{N}{t}}(u,\Omega) \cap S_{\phi}(0, rt)\Big)\right|
 &\geq \left\{1-\beta \epsilon_0 -C \big(\frac{t}{N}\big)^\tau \Big(\fint_{S_{\phi}(0, t)} | f|^n ~dx\Big)^{\frac{\tau}{n}} \right\}~ \left|T (S_{\phi}(0, rt))\right|~\forall t\leq \tilde c.
 \end{align*}
This is equivalent to
\begin{equation}\label{better-localized-est}
 \left| G_{\frac{N^{'}}{t}}(u,\Omega) \cap S_\phi(x, t)\right|
 \geq \left\{1-\epsilon_0\beta -C \big(\frac{t}{N^{'}}\big)^\tau  \Big(\fint_{S_\phi(x, \frac{t}{r} )} | f|^n ~dx\Big)^{\frac{\tau}{n}} \right\}~ 
 \left|S_\phi(x, t)\right|
 \end{equation}
giving \eqref{localized-est} for any $N^{'}\geq N_1\equiv N_0 r$ and $t\leq  r\tilde c$.  

Next we consider the situation that $x\in \Omega$. We then have the following possibilities:
 
 {\bf Case 1:}  $t\leq h/2$, where $h :=\bh(x)$. 
 
 If $h\geq c$ where $c$ is defined in Proposition \ref{tan_sec} then the estimate (\ref{localized-est}) is an easy consequence of
 the interior density estimates \cite[Lemma 4.3]{GN2} which we now recall.
 \begin{lemma}(\cite[Lemma 4.2]{GN2}) Let $0<\alpha_0<1$ and $\Omega$ be a convex domain in $\R^{n}$ satisfying $B_{k_0}\subset \Omega\subset B_{k_0^{-1}}$ and 
 $u\in C^1(\Omega)\cap W^{2, n}_{loc}(\Omega)$ be a solution of $\Phi^{ij} u_{ij} = f$ in $\Omega$ with $\|u\|_{L^{\infty}(\Omega)}\leq 1,$ where $\phi \in C(\overline{\Omega})$
 is a convex function satisfying $\phi =0$ on $\p\Omega$. Let  $0<\e_{0}<1$. There exists $\e>0$ depending only on $\e_0, \alpha_0, k_0$ and $n$ such that if 
 \[1-\e\leq \det D^2 \phi \leq 1+ \e \quad \mbox{in}\quad  \Omega,
 \]
  then
 for any section $S_{\phi}(x_0, \frac{t_0}{\alpha_0})\subset \Omega_{\frac{\alpha_0 +1}{2}}:=\{x\in \Omega: \phi(x) < (1-\frac{\alpha_0 + 1}{2})\min_{\Omega}\phi\}$, we have
  $$|G_{\frac{N}{t_0}}(u,\Omega)\cap S_{\phi}(x_0, t_0)|\geq \left\{1-\e_0 - C(\frac{t_0}{N})^{\tau} \Big(\fint_{S_{\phi}(x_0, \frac{t_0}{\alpha_0})}
  \abs{f}^n\Big)^{\frac{\tau}{n}}\right\}|S_{\phi}(x_0, t_0)|$$
  for every $N\geq N_0$. Here $C, \tau, N_0$ are positive constants depending only on $\alpha_0, n$ and $k_0.$
  \label{interior-density-est}
 \end{lemma}
Now we consider the remaining situation in {\bf Case 1} when $h\leq c.$ 
We define the rescaled domain $\tilde \Omega_{h}$ and rescaled functions $\tilde \phi_{h}$,
$\tilde u_{h}$ and $\tilde f_h$ as in Subsection~\ref{rescale-sec} that {\it preserve
the $L^{\infty}$-norm in a section tangent to the boundary}.
Now, we apply Lemma~\ref{interior-density-est} to the domain $ S_{\tilde{\phi}_{h}} (0, 1)$ with $\alpha_0 =3/4$, $x_0=0$ and $t_0 = t/h\leq 1/2$,
noting that $(S_{\tilde\phi_h}(0,1))_{\alpha}= S_{\tilde\phi_h}(0,\alpha)$ for all $\alpha>0$.
Thus,
 \begin{equation}|G_{\frac{N h}{t}}(\tilde{u}_h, S_{\tilde{\phi}_{h}} (0, 1), \tilde \phi_h)\cap S_{\tilde{\phi}_{h}} (0, \frac{t}{h})|\geq
 \left\{1-\e_0 - C(\frac{t}{hN})^{\tau} \Big(\fint_{S_{\tilde{\phi}_{h}} (0, \frac{4t}{3h})}
  \abs{\tilde{f}_{h}}^n\Big)^{
  \frac{\tau}{n}}\right\}|S_{\tilde{\phi}_{h}} (0, \frac{t}{h})|.
 \label{density-h}
 \end{equation} 
Let $Ty := h^{-1/2}A_{h}(y-x)$. Then
$$G_{\frac{N h}{t}}(\tilde{u}_h, S_{\tilde{\phi}_{h}} (0, 1), \tilde\phi_h)\cap S_{\tilde{\phi}_{h}} (0, \frac{t}{h}) = T\big(G_{\frac{N}{t}}(u,\Omega) \cap S_{\phi}(x, t)\big).$$
Changing variables in \eqref{density-h} gives
$$|G_{\frac{N}{t}}(u, \Omega)\cap S_{\phi}(x, t)|\geq \left\{1-\e_0 - C(\frac{t}{N})^{\tau} \Big(\fint_{S_{\phi}(x, \frac{4t}{3})}
  \abs{f}^n\Big)^{\tau/n}\right\}|S_{\phi}(x, t)|$$
 and hence \eqref{localized-est} holds. 

{\bf Case 2:} $h/2 <t\leq r\tilde c/\bar{c}\equiv c_{1}$ where $\bar{c}>1$ is the constant in Proposition~\ref{dicho}. Then  
by Proposition~\ref{dicho}, we know that 
$S_{\phi}(x, 2 t) \subset S_{\phi}(z, \bar{c} t) $
for some $z\in\p\Omega$, and by Theorem~\ref{engulfing2}(b), 
\begin{equation*}
C_{1} t^{n/2}\leq |S_{\phi}(x, t)|\leq C_{2} t^{n/2}\quad \forall t\leq c_0.
\end{equation*}
Using these inequalities and the estimate \eqref{better-localized-est} in the case of boundary section, we get 
\begin{eqnarray*}
\left| S_\phi(x, t)\setminus G_{\frac{N}{t}}(u,\Omega) \right|
&\leq&   \left| S_\phi(z, \bar{c}t)\setminus G_{\frac{N}{\bar{c}t}}(u,\Omega) \right|\\
&\leq& \left\{ \epsilon_0\beta +C \big(\frac{ \bar{c}t}{N}\big)^\tau  \Big(\fint_{S_\phi(z,  \bar{c}t/r)} | f|^n ~dx\Big)^{\frac{\tau}{n}}\right\} |S_\phi(z, \bar{c} t)| \\
&\leq& \left\{ \epsilon_0\beta +C \big(\frac{\sqrt{t} }{N}\big)^\tau  \|f\|_{L^n(\Omega)}^\tau\right\} |S_\phi(x, t)| \,C_1^{-1} C_2 \bar{c}^{\frac{n}{2}}.
 \end{eqnarray*}
This implies \eqref{localized-est} as desired by choosing $\beta = C_1 C_2^{-1} \bar{c}^{\frac{-n}{2}}$ and $c_1= r\tilde{c}/\bar{c}= c^9\tilde{c}/\bar{c}$.
\end{proof}

The next lemma is a key technical ingredient in our global $W^{2, p}$ estimates. It propagates a point in a given section 
where the solution $u$ of $\mathcal{L}_{\phi} u=f $ has bounded second derivative to almost all points in that section. More precisely, it says that if in a small 
section $S_{\phi}(x, t)$
we can find a point where $u$ is touched from above and below by quasi paraboloids of opening $\gamma$ generated by $\phi$ then on a set of nearly full measure of $S_{\phi}(x, t)$,
$u$ is touched from above and below by quasi paraboloids of opening $N\gamma$  for some controllable constant $N$, provided that
$\det D^2\phi$ is close to a constant.

\begin{lemma}\label{lm:furthercriticaldensitytwo}
Assume $\Omega$ is uniformly convex satisfying \eqref{global-tang-int} and $\phi\in C^{0,1}(\overline{\Omega})$ is a convex function satisfying \eqref{global-sep} and
\[
1-\epsilon \leq \det D^2\phi \leq 1+\epsilon\quad\mbox{in}\quad\Omega.
\]
Assume in addition that $\partial \Omega\in C^{2,\alpha}$ and $\phi \in C^{2,\alpha}(\partial\Omega)$ for some $\alpha\in (0,1)$. 
Let $u\in  C^1(\Omega) \cap W^{2,n}_{loc}(\Omega)$ be a solution of $\calL_\phi u=f$ in $\Omega$ and  $u=0$ on $\partial\Omega$. 
Let $0<\epsilon_0<1$.
Then there exists $\epsilon>0$ depending only on  $\epsilon_0, n, \rho$ and $\alpha$ such that for 
any $x\in \overline{\Omega}$,   $t\leq c_2$ and $S_\phi(x,t)\cap G_\gamma(u,\Omega)\neq \emptyset$ we have
\begin{equation}\label{improve-density-II}
 \left| G_{N \gamma }(u,\Omega) \cap S_\phi(x,  t)\right|
 \geq \left\{1-\epsilon_0 -C (N\gamma)^{-\tau} \Big(\fint_{S_\phi(\tilde x,\Theta t)} | f|^n ~dx\Big)^{\frac{\tau}{n}} \right\}~ \left|S_\phi(x, t)\right|
 \end{equation}
for all $\tilde x\in S_\phi(x,t)$ and $N\geq N_2$.
Here $\tau$ and $\Theta$ depend only on $n$ and $\rho$;  $C$, $c_2$ and $N_2$ depend only on  $n, \rho, \alpha$, the 
uniform convexity of $\Omega$, $\|\partial\Omega\|_{C^{2,\alpha}}$ and $\|\phi\|_{C^{2,\alpha}(\partial\Omega)}$.
\end{lemma}

\begin{proof}
As explained in the proof of Lemma~\ref{lm:improved-density-I}, we have $\phi\in C^{1}(\Omega)\cap W^{2, 2n}(\Omega)$ if $\e$ is small.

Let us  first consider the case $x\in \partial\Omega$. We can assume that $x=0$, $\phi(0)=0$ and $\nabla \phi(0)=0.$
Let $h=\theta t$ where $\theta=\theta(n,\rho)>1$ will be chosen later.
Let $A_{h}$ be the affine transformation as in the Localization Theorem \ref{main_loc}.
We now define the rescaled domains $\Omega_{h}, U_{h}$ and rescaled functions $\phi_{h},$
$\tilde{u}_{h}$  and $\tilde{f}_h$ as in Subsection~\ref{rescale-sec} {\it that almost preserve the $L^{\infty}$-norm of $D^2u$}. 
Let $T= h^{-1/2}A_{h}$. 
  
 Let $\bar x\in S_\phi(0,t)\cap G_\gamma(u,\Omega)$ and $\bar y :=T\bar x$. Then
\[
-\gamma \, d(x,\bar x)^2 \leq u(x)-u(\bar x) - \nabla u(\bar
x)\cdot (x-\bar x) \leq \gamma \, d(x,\bar x)^2,\quad \forall x\in \Omega.
\]
By changing variables and recalling that $\Omega_h= T (\Omega), \tilde{u}_h(y)= h^{-1} u(T^{-1} y)$, we get
\begin{equation}\label{eq:utildeboundedbydistance}
-\gamma \, \dfrac{d(T^{-1}y,T^{-1}\bar y)^2}{ \theta t} \leq \tilde
u_h(y)- \tilde u_h(\bar y) - \nabla \tilde u_h(\bar y)\cdot (y-\bar y)
\leq \gamma \, \dfrac{d(T^{-1}y,T^{-1}\bar y)^2}{ \theta t},~ \forall y\in \Omega_h.
\end{equation}
Since  $\bar x\in S_{\phi}(0, t) \subset S_{\phi}(0, \theta t),$ we have by the engulfing property of sections in Theorem~\ref{engulfing2}(a)
$S_{\phi}(0, \theta t)\subset
S_\phi(\bar x, \theta^2 t).$ It follows that
$d(x,\bar x)^2\leq \theta^2 t$ for $x\in S_{\phi}(0, \theta t)$ yielding $d(T^{-1}y,T^{-1}\bar y)^2\leq \theta^2 t$ for all
$y\in U_h:= T(S_{\phi}(0, h))$.  Consequently, if we define
\begin{equation}
\label{affine-function-l}
v(y) := \dfrac{1 }{\theta \gamma }\left[\tilde u_h(y)- \tilde u_h(\bar y) - \nabla \tilde u_h(\bar
y)\cdot (y-\bar y)\right],\quad y\in \Omega_h,
\end{equation}
then $|v| \leq 1$ in $U_h$. Thanks to Lemma~\ref{lm:Holder-on-bottom-boundary} below we get for $t\leq c_{\alpha}$
\begin{equation}\label{Holder-est-for-L}
\|v\|_{C^{2, \alpha}(\partial{U_h} \cap B_{k})}\leq C_{\alpha},
\end{equation}
where $c_{\alpha}, C_{\alpha}$ depend only on $n, \rho,\alpha$, the 
uniform convexity of $\Omega$, $\|\partial\Omega\|_{C^{2,\alpha}}$ and $\|\phi\|_{C^{2,\alpha}(\partial\Omega)}$.
By \eqref{eq:utildeboundedbydistance} we  have
\begin{equation}
|v(y)|\leq \frac{1}{\theta^2  t}
d(T^{-1} y, T^{-1} \bar y)^2
\leq \frac{1}{\theta } d_{ \phi_h}( y, \bar y)^2\quad \forall y\in T(\Omega),
\label{v-far-off}
\end{equation}
where we recall $\theta t= h$ and $$d_{\phi_h}(y,\bar y)^2
:= \phi_h(y) -\phi_h(\bar y) - \nabla \phi_h (\bar y)\cdot ( y - \bar y) =h^{-1}\, d(T^{-1}y,T^{-1}\bar y)^2.$$
Moreover
\begin{align*}
\mathcal{L}_{\phi_h} v=  (\theta \gamma)^{-1} \mathcal{L}_{\phi_h}\tilde{u_h}=(\theta \gamma)^{-1}\tilde{f}_h\equiv (\theta \gamma)^{-1} f(T^{-1} y)=: \tilde f(y).
\end{align*}
Because $\bar x\in S_{\phi}(0, t)$, we have $ \bar y = T \bar x\in S_{\tilde\phi}(0, \frac{1}{\theta}).$
Hence, we can choose $\theta>1$ depending on $n, \rho, k$ such that $\bar y\in B_{\frac{c^2}{8}}\cap \tilde U$. 
With this choice of $\theta$, we have by Proposition \ref{P-prop} 
$$(\Omega_h, \phi_h, U_h)\in \mathcal{P}_{1-\e, 1+\e, \rho, Ch^{1/2}, \alpha}
\subset  \mathcal{P}_{1-\e, 1+\e, \rho, 1, \alpha}
$$
 if $t\leq \tilde c$, where $\tilde c>0$ is a small constant depending only on  $n, \rho, \alpha$,
$\|\partial\Omega\|_{C^{2,\alpha}}$ and $\|\phi\|_{C^{2,\alpha}(\partial\Omega)}$. Here we can choose $\tilde c\leq c_{\alpha}$, and hence it also depends on the  uniform convexity of $\Omega$.

  Thus, using \eqref{Holder-est-for-L} and (\ref{v-far-off}), we can apply Lemma~\ref{lm:acceleration} to $\bar{v}:= v/C_{\alpha}$ to obtain 
\begin{equation*}
|G_{N}(\bar{v}, \Omega_h, \phi_h) \cap  S_{\tilde \phi}(0, c^9)|\geq \left\{1- C\big( N^{-\tau}\delta_0^{\tau} +\epsilon^{1/3n}\big)\right\} \, | S_{\phi_h}(0, c^9)|
\end{equation*}
for any  $N\geq N_0$,
where
$\delta_0$ is as in \eqref{delta-zero}.
This together with the stability of cofactor matrices in Proposition~\ref{global-convergence}  implies the existence of
$\e= \e(\epsilon_0, n,\rho,\alpha)>0 $ such that
\begin{align*}
 |G_{N}(\bar{v}, \Omega_h, \phi_h) \cap   S_{\phi_h}(0, r)|
 &\geq \left\{1-\epsilon_0\beta -C N^{-\tau} \Big(\fint_{U_h} |\tilde f|^n ~dy\Big)^{\frac{\tau}{n}} \right\}~ | S_{\phi_h}(0, r)|\\
 &=\left\{1-\epsilon_0\beta -C
  \big(\frac{1}{\theta \gamma N}\big)^\tau \Big(\fint_{S_{\phi}(0, \theta t)} | f|^n ~dx\Big)^{\frac{\tau}{n}} \right\}~ | S_{\phi_h}(0, r)|,
 \end{align*}
 where for simplicity we have denoted
 $$r:= c^9$$
 and  $\beta=\beta(n,\rho)<1$ is a universal constant to be chosen later.
 It follows that 
\[
 | S_{\phi_h}(0, r) \setminus G_{N}(\bar{v}, \Omega_h, \phi_h)|\leq\left\{ \epsilon_0\beta +C \big(\frac{1}{\theta \gamma N}\big)^\tau 
 \Big(\fint_{S_{\phi}(0, \theta t)} | f|^n ~dx\Big)^{\frac{\tau}{n}}\right\} | S_{\phi_h}(0, r)|.
\]
 As $ S_{\phi_h}(0, r) =T(S_{\phi}(0, \theta rt))$ and  $\bar{v}(y) =\frac{1}{C_{\alpha}\theta^2 \gamma t} \big[u(T^{-1}y) - u(\bar x) -
 \nabla u(\bar x) \cdot (T^{-1}y -\bar x ) \big]$,   it is easy to see that
\begin{align*}
&G_{N}(\bar{v},\Omega_h, \phi_h) \cap  S_{\phi_h}(0, r)
=T \Big( G_{C_{\alpha}N\theta \gamma}(u,\Omega) \cap S_{\phi}(0, r\theta t)\Big).
\end{align*}
Therefore, by the volume estimates in Theorem \ref{engulfing2}(b),  we conclude that
\begin{eqnarray*}
| S_{\phi}(0, r t)\setminus G_{C_{\alpha}N\theta \gamma}(u,\Omega)|
&\leq&
 | S_{\phi}(0, r\theta t)\setminus G_{C_{\alpha}N\theta \gamma}(u,\Omega)|
 \\&\leq& \left\{ C_1^{-1} C_2 \theta^{\frac{n}{2}}\epsilon_0\beta +C
  \big(\frac{1}{C_{\alpha}\theta \gamma N}\big)^\tau \Big(\fint_{S_{\phi}(0, \theta t)} | f|^n ~dx\Big)^{\frac{\tau}{n}}\right\} \left|S_\phi(0, rt)\right|.
 \end{eqnarray*}
By setting $N'= C_{\alpha}N\theta$, $\beta' = C_1^{-1} C_2 \theta^{n/2}\beta$, we can rewrite this as 
\begin{equation}\label{density-est-boundary}
 \left| G_{ N' \gamma}(u,\Omega) \cap S_\phi(x,  t)\right|
 \geq \left\{1- \epsilon_0\beta' -C
  \big(\frac{1}{ \gamma N'}\big)^\tau
 \Big(\fint_{S_\phi(x, \frac{\theta}{r} t)} | f|^n ~dx\Big)^{\frac{\tau}{n}} \right\}~ \left|S_\phi(x, t)\right|
 \end{equation}
for any $N'\geq N_2\equiv C_{\alpha}N_0\theta$ and $t\leq r\tilde c$. From Theorem ~\ref{engulfing2}(a) we have  
$S_\phi(x,\frac {\theta}{r} t)\subset S_\phi(\tilde x, \frac{\theta_{\ast}\theta}{r} t)$ for 
any $\tilde x\in S_\phi(x, t)$. Therefore, by Theorem~\ref{engulfing2}(b), we see that \eqref{density-est-boundary} yields
 \eqref{improve-density-II}.
 
Next we consider the situation that $x\in \Omega$. We then have the following possibilities:

{\bf Case 1:} $t\leq h/2$, where $h :=\bh(x)$. This case can be handled as {\bf Case 1} of Lemma~\ref{lm:improved-density-I},  
using now \cite[Lemma~4.5]{GN2} and affine transformations similar to the ones at the beginning of the proof of this lemma.

{\bf Case 2:} $h/2 <t\leq r\tilde c/\bar{c}\equiv c_2$, where $\bar{c}>1$ is the constant in Proposition~\ref{dicho}. 
Then by Proposition~\ref{dicho}, we know that $ S_{\phi}(x, 2 t) \subset S_{\phi}(z, \bar{c} t) $
for some $z\in\p\Omega$. Thus, by the estimate \eqref{density-est-boundary} in the 
case of boundary section, we get
\begin{align}\label{improve-density-case2}
\left|  S_\phi(x,  t) \setminus G_{ N \gamma}(u,\Omega)\right|
&\leq \left|  S_\phi(z,  \bar{c}t) \setminus G_{ N \gamma}(u,\Omega)\right|\\
 &\leq \left\{\epsilon_0\beta' +C
  \big(\frac{1}{ \gamma N}\big)^\tau
 \Big(\fint_{S_\phi(z, \frac{\theta \bar{c}}{r}t)} | f|^n ~dx\Big)^{\frac{\tau}{n}} \right\}~ \left|S_\phi(z, \bar{c}t)\right|.\nonumber
\end{align}
For any $\tilde x\in S_\phi(x, t) \subset S_\phi(z,\frac{ \theta \bar{c}}{r}t)$,  we 
get  $S_\phi(z, \frac{\theta \bar{c}}{r}t) \subset S_\phi(\tilde x, \frac{\theta_{\ast}\theta \bar{c}}{r}t)$ by 
the engulfing property in Theorem~\ref{engulfing2}. Now, using \eqref{improve-density-case2} and the volume estimates in this theorem, we find that
\begin{equation*}
\left|  S_\phi(x,  t) \setminus G_{ N \gamma}(u,\Omega)\right|
 \leq \left\{\epsilon_0\beta' C_1^{-1} C_{2}\bar{c}^{\frac{n}{2}} +C
  \big(\frac{1}{ \gamma N}\big)^\tau
 \Big(\fint_{S_\phi(\tilde x, \frac{\theta_{\ast}\theta \bar{c}}{r}t)} | f|^n ~dx\Big)^{\frac{\tau}{n}} \right\} \left|S_\phi(x, t)\right|.
\end{equation*}
This gives \eqref{improve-density-II} with $\Theta := \theta_{\ast} \theta \bar{c}/r$
if we choose $\beta$ such that $\beta' C_1^{-1} C_2 \bar{c}^{n/2} = \beta C_1^{-2} C_2^2 (\theta\bar{c})^{n/2} =1$.
\end{proof}

In the next lemma we prove that the function $v$ defined as in the proof of Lemma~\ref{lm:furthercriticaldensitytwo} has uniform
$C^{2, \alpha}$ bound on $\p U_h\cap B_{k}$. 
\begin{lemma}\label{lm:Holder-on-bottom-boundary} Let $v$ be defined as in \eqref{affine-function-l}. There exist $C_{\alpha}, c_{\alpha}>0$  depending only on $n, \rho,\alpha$, the 
uniform convexity of $\Omega$, $\|\partial\Omega\|_{C^{2,\alpha}}$ and $\|\phi\|_{C^{2,\alpha}(\partial\Omega)}$ such that for $t\leq c_{\alpha}$, we have
\begin{equation}\label{l-half-holder}
\|v\|_{C^{2, \alpha} (\p U_h\cap B_{k}^{+})} \leq C_{\alpha}.
\end{equation}
\end{lemma}
\begin{proof}
Since  $\p \Omega$ is $C^{2,\alpha}$ at the origin and $\Omega$ is uniformly convex,  we have
$$\abs{x_n- q(x')}\leq M\abs{x'}^{2+\alpha}
\quad \mbox{for}\quad x=(x', x_n)\in \p\Omega\cap B_{\rho},$$
where $q(x')$ is a homogeneous quadratic polynomial with
\begin{equation}
\label{Hess}
D^2_{x'} q \geq C^{-1}I_{n-1}.
\end{equation}
Recall $h= \theta t.$ Then it follows from the definition of $U_h$ and Proposition~\ref{P-prop} that
\begin{equation}\abs{x_{n}- h^{1/2}q(x')}\leq C h^{\frac{1+\alpha}{2}}\abs{x'}^{2+\alpha}
\quad\mbox{on}\quad \p  U_h \cap B_{k}^{+}
\label{C2alpha-bdr}
\end{equation}
if $h\leq h_0$, where $h_0, C$ depend only on $n, \rho,\alpha$ and the $C^{2, \alpha}$ norms of $\partial\Omega$ and $\phi|_{\partial\Omega}$ at the origin.
Hence by combining with \eqref{Hess}, we see that if $h\leq h_0$ ($h_0$ now  depends also on the uniform convexity of $\Omega$) then on $\p U_h\cap B_{k}^{+}$,
\begin{equation} \frac{1}{2} h^{1/2} q (x')\leq x_n \leq 2 h^{1/2} q(x').
\label{quadbound}
\end{equation}
Let
$$l(y) =\frac{ -1}{\theta\gamma} \big[\tilde{u}_h(\bar{y}) + \nabla \tilde{u}_h(\bar{y})\cdot (y-\bar{y})\big].
$$
Then $l(y) =v(y)$ for $y \in \p U_h\cap B_{k}^{+}$. Since $\abs{v}\leq 1$ in $U_h$, 
we find that
\begin{equation}
\label{2bound}
\abs{l(y)- l(z)}= \frac{1}{\theta\gamma}\abs{\nabla \tilde{u}_h(\bar{y})\cdot (y-z)}\leq 2\quad \forall y, z\in \p U_h\cap B_{k}^{+}.
\end{equation}

All constants in this lemma, unless otherwise indicated, depend only on $n, \rho,\alpha$, the 
uniform convexity of $\Omega$ and the $C^{2, \alpha}$ norms of $\partial\Omega$ and $\phi|_{\partial\Omega}$.

We now divide the proof into three steps.\\
{\bf Step 1.} $l$ is uniformly Lipschitz at the origin: there exists $L>0$ such that
$$\abs{l(z)-l(0)}\leq L \abs{z}\quad \forall z\in \p U_h\cap B_{k^2}^{+}.$$

Take $z\in \p U_h\cap B_{k^2}^{+}\backslash\{0\}.$ Let $\mathcal{C}$ be the curve which is the intersection of $\p U_h\cap B_{k}^{+}$  and the vertical plane $(P)$ passing through $z$ and the 
origin. Let $p$ and $q$ be the intersection of $\mathcal{C}$ with $\p B_{k}^{+}.$ We now have a plane curve
$\mathcal{C}$ in $(P)$ which can be assumed to be the usual $xy$-plane. It is easy to see from \eqref{Hess}--(\ref{quadbound}) that $\mathcal{C}$ is a graph in the $y$-direction
$\mathcal{C}= \{(x, \varphi(x))\}$
with $C^{1, 1}$ norm comparable to $h^{1/2}$, that is
$$C^{-1} h^{1/2}\leq \varphi^{''}(x)\leq Ch^{1/2}.$$
Note that, this also follows from the proof of \cite[Lemma 4.2]{LS} for the case of uniformly convex domains $\Omega$.

Since $\abs{p}=\abs{q} =k$, we find that
$$y_{p}\sim h^{1/2}, y_{q}\sim h^{1/2}, ~ \abs{x_p}\sim k, \abs{x_q}\sim k.$$
Without loss of generality, we can assume that $y_{p}\leq y_{q}$ and $x_p<0< x_q$, that is, $p$ is on the left half-plane while $q$ is on the right half-plane. The horizontal line
through $p$ intersects $\mathcal{C}$ at another point $q^{'}.$ Since $\varphi^{''}\leq Ch^{1/2}$ and $y_{q^{'}}= y_{p}\sim h^{1/2}$, we must have $x_{q^{'}}\sim k.$ In particular, $z$
lies on the arc $p0q^{'}.$ We can assume that $z$ lies on the arc $0q^{'}.$ Now, take a ray emanating from $q^{'}$ and parallel to $0z.$ This ray is exactly $q^{'}0$ when $z\equiv q^{'}$ and
it is $q^{'}p$ when $z\rightarrow 0.$ Thus, by continuity, there must be a point $m$ on the arc $0p$ such that $q^{'}m$ is parallel to $0z$. Clearly,
$\abs{q^{'}-m}\geq x_{q^{'}}\sim k.$
Using $z= \frac{\abs{z}}{\abs{q^{'}-m}} (q^{'}-m),$
we find from \eqref{2bound} that
$$\abs{l(z)-l(0)}= \frac{1}{\theta\gamma}\abs{\nabla \tilde{u}_h(\bar{y})\cdot z}=\frac{\abs{z}}{\abs{q^{'}-m}} \frac{1}{\theta\gamma}\abs{\nabla \tilde{u}_h(\bar{y})\cdot (q^{'}-m)} \leq
\frac{\abs{z}}{\abs{q^{'}-m}}  \leq L\abs{z}.$$
Thus, $l$ is Lipschitz at $0$.\\
{\bf Step 2.} Let $ \frac{1}{\theta\gamma}\nabla \tilde{u}_h(\bar{y}) = (a', a_n).$ Then
$$\abs{a'}\leq 2L\quad \mbox{and}\quad \abs{a_n}h^{1/2}\leq CL.$$
First, we note that the projection of $\p U_h\cap B_{k}$ on $\{x_n=0\}$ contains a ball of radius comparable to $k$. By rotating coordinates in $\{x_n=0\}$, we can assume that 
$a'= (A, 0,\cdots, 0).$ Take a curve $\mathcal{C} =\big\{(x, 0, \cdots, 0,\varphi(x))\mid -k^2 \leq x\leq k^2\big\}$ in $\p U_h\cap B_{k}$ that lies in the $x_1x_n$ plane. 
Note that $\varphi(x)\sim h^{1/2} x^2.$
By the Lipschitz property
of $l$ in {\bf Step 1}, we have
$$ \frac{1}{\theta\gamma}\big|\nabla \tilde{u}_h(\bar{y})\cdot (x, 0, \cdots, 0,\varphi(x))\big| = \abs{A x + a_n \varphi (x)} \leq L \sqrt{x^2 + (\varphi(x))^2} \leq 2L\abs{x}$$
Dividing the above inequalities by $x$ and then letting $x\rightarrow 0$,
we get the desired bound $$\abs{a'}=\abs{A}\leq 2L.$$
As a consequence, we have
$$\abs{a_n \varphi(x)}\leq \abs{Ax} + 2L\abs{x} \leq 4L\abs{x}.$$
Using the lower bound on the growth of $\varphi$ and evaluating at $\abs{x}\sim k^2$, we obtain 
$$\abs{a_n}h^{1/2} \leq CL.$$
{\bf Step 3.} We have $$\|v\|_{C^{2, \alpha} (\p U_h\cap B_{k}^{+})}= \|l\|_{C^{2, \alpha} (\p U_h\cap B_{k}^{+})} \leq C.$$
Recall from (\ref{C2alpha-bdr}) that $\p U_h\cap B_{k}$ is a graph in the $e_n$ direction, that is, 
$$\p U_h\cap B_{k} = \big\{(x', \psi(x')):  \, |x'|\leq C_k\big\},$$  with the following properties:
\begin{equation*}
 (a)\, \|\nabla \psi\|_{L^{\infty}} + \|D^2 \psi\|_{L^{\infty}}\leq C h^{1/2},\qquad (b)\, \|D^2 \psi\|_{C^{\alpha}} \leq Ch^{\frac{1+\alpha}{2}}.
\end{equation*}
For $y\in \p U_h\cap B_{k}$, we have $y =(x', \psi(x'))$
and
$$l(y) = l(0) -\frac{1}{\theta\gamma}\nabla \tilde{u}_h( \bar{y}) \cdot y = l(0) -a' \cdot x' -a_n \psi (x')$$
where $l(0)$ is a constant bounded by $1$. Clearly, the $C^{2, \alpha}$ bound for $l$ on $\p U_h\cap B_{k}$ now follows from $(a)-(b)$ and {\bf Step 2.}
\end{proof}
\subsection{Global $W^{2,p}$ estimates}
In this subsection we will use the density estimates established in Subsection~\ref{sub:improveddensity} to derive 
global $W^{2,p}$-estimates for solution $u$ of the linearized equation $\calL_\phi u=f$ when $f\in L^q(\Omega)$ for some $q>n$ as stated in Theorem~\ref{main-result-constant} 
and Theorem~\ref{main-result}.

\begin{proof}[Proof of Theorem \ref{main-result-constant}]
The assumptions on $\Omega$ and $\phi$ in the statement of our theorem imply that $\Omega$ satisfy \eqref{global-tang-int} for some $\rho>0$
and, by Proposition~\ref{pro:quadsep}, $\phi$ satisfies \eqref{global-sep}. Thus, $\Omega$ and $\phi$ satisfy the conditions of Lemma~\ref{lm:improved-density-I} and Lemma~\ref{lm:furthercriticaldensitytwo}.

By the ABP estimate, it suffices to establish our $W^{2,p}$ estimates in the form
\begin{equation*}
\|D^2u\|_{L^p(\Omega)}\leq
C\Big( \|u\|_{L^\infty(\Omega)} +\|f\|_{L^{q}(\Omega)}\Big).
\end{equation*}
We first observe that by working with the function  $v :=\dfrac{\epsilon u}{\epsilon \|u\|_{L^\infty(\Omega)} +\|f\|_{L^{q}(\Omega)}}$ instead of $u$, it is enough to show that  there exist $\epsilon, C>0$  depending only on $p$, $q$, $n$ and $\Omega$ such that
 if  $1-\epsilon \leq \det D^2\phi \leq 1+\epsilon$ in $\Omega$, $\phi=u=0$ on $\partial\Omega$, $\calL_\phi u=f$ in $\Omega$,  $\|u\|_{L^\infty(\Omega)}\leq 1$ and $\|f\|_{L^{q}(\Omega)}\leq \epsilon$, then  
\begin{equation}\label{reduction}
\|D^2u\|_{L^p(\Omega)}\leq C.
\end{equation}
Notice that $u\in W^{2,s}_{loc}(\Omega)$ for any $n< s<q$  as a consequence of  $W^{2,p}_{loc}$ estimates in \cite{GN2}.

Let  $N_{\ast}= \max\{N_1, N_2\}$ where $N_1$ and $N_2$  are the large constants in Lemma~\ref{lm:improved-density-I} and
Lemma~\ref{lm:furthercriticaldensitytwo} and $\hat{c}=\min\{c_1, c_2\}$ where $c_1$ and $c_2$  are the small constants in the above lemmas.
Fix $M\geq N_{\ast}$ so that $1/M< \hat c$. Next select $0<\epsilon_0<1/2$ such that
\[
M^{q} \sqrt{2 \epsilon_0}  =\frac{1}{2}
\]
and  $\epsilon=\epsilon(\epsilon_0, n,\Omega)=\epsilon(p,q,n,\Omega)$ be the smallest  of the constants in
Lemma~\ref{lm:improved-density-I} and
Lemma~\ref{lm:furthercriticaldensitytwo}.
With this choice of $\epsilon$, we are going to show that \eqref{reduction} holds.
Applying Lemma~\ref{lm:improved-density-I} to the function
$u$ and using  $\|f\|_{L^{q}(\Omega)}\leq \epsilon$ we obtain 
\begin{align*}
\big|S_\phi(x,t)\cap G_{\frac{M}{t}}(u,\Omega)\big|
\geq
\left(1-\epsilon_0 - C \epsilon^\tau \right)\, |S_\phi(x,t)|
\end{align*}
as long as $x\in \overline{\Omega}$ and $t\leq \hat c$.
By taking $\epsilon$ even smaller if necessary we can assume $C \epsilon^\tau <\epsilon_0$. Then it follows from the above inequality that
\begin{equation}\label{eq:consequence-lemma6.4}
\big|S_\phi(x,t)\setminus G_{\frac{M}{t}}(u,\Omega)\big| 
\leq 2\epsilon_0 \, |S_\phi(x,t)|
\quad\mbox{for any}\quad x\in \overline{\Omega},\quad  t\leq \hat c.
\end{equation}
Let  $1/h \leq \hat c$. For $x\in
\Omega\setminus G_{h M}(u,\Omega)$, define
\[
g(t) := \dfrac{\big|(\Omega\setminus G_{h M}(u,\Omega))
\cap S_\phi(x,t)\Big|}
{|S_\phi(x,t)|}.
\]
We have $\lim_{t\to 0} g(t)=1.$ Also, if $1/h \leq t
\leq \hat c$, then \eqref{eq:consequence-lemma6.4} gives
\begin{align*}
\big|(\Omega\setminus G_{h M}(u,\Omega)) \cap
S_\phi(x,t)\big| &\leq
|S_\phi(x,t)\setminus G_{h M}(u,\Omega)|\\
&\leq
|S_\phi(x,t)\setminus G_{M/t}(u,\Omega)|
\leq 2\epsilon_0 \, |S_\phi(x,t)|.
\end{align*}
Therefore $g(t)\leq
2\epsilon_0$ for $t\in [1/h, \hat c].$ Then by continuity of
$g$, there exists $t_{x}\leq 1/h$ such that $g(t_{x})=2\epsilon_0$.

Thus for any $x\in
\Omega\setminus G_{h M}(u,\Omega)$ there is
$t_{x}\leq 1/h \leq
\hat c$ satisfying
\begin{equation}\label{eq:generalcriticaldensity}
\big|\big(\Omega\setminus G_{h M}(u,\Omega)\big) \cap
S_\phi(x,t_{x})\big| = 2\epsilon_0 \, |S_\phi(x,t_{x})|.
\end{equation}
We now claim that \eqref{eq:generalcriticaldensity} implies 
\begin{equation}\label{eq:section-contained}
S_\phi(x,t_{x}) \subset \big(\overline{\Omega}\setminus G_h
(u,\Omega)\big) \cup \big\{z\in \overline{\Omega}: \mathcal M(f^n)(z)> (c^* M h)^n\big\},
\end{equation}
where $c^* :=(\frac{\epsilon_0}{C})^{1/\tau}$, and 
\begin{equation*}
\mathcal M(F)(z) :=\sup_{t\leq \hat c}
\frac{1}{|S_{\phi}(z, t)|}\int_{S_\phi(z,t)}|F(y)|\, dy\quad \forall z\in \overline{\Omega}.
\end{equation*}
Indeed, since otherwise  there exists $\bar x\in
S_\phi(x,t_{x})\cap G_h(u,\Omega)$ such that $\mathcal M
(f^n)(\bar x)\leq (c^* M h)^n$. Note also that  $ t_{x} \leq   \hat c$. Then by
Lemma~\ref{lm:furthercriticaldensitytwo} applied to $u$ we get 
\begin{align*}
\big|S_\phi(x,t_{x})\cap G_{h M}(u,\Omega)\big| > (1-2 \epsilon_0) \, |S_\phi(x,t_{x})|
\end{align*}
yielding
\[
\big|\big(\Omega\setminus G_{h M}(u,\Omega)\big)\cap
S_\phi(x,t_{x})\big| \leq \big|S_\phi(x,t_{x})\setminus
G_{h M}(u,\Omega)\big| < 2\epsilon_0 \,|S_\phi(x,t_{x})|.
\]
This is a contradiction with \eqref{eq:generalcriticaldensity} and so
\eqref{eq:section-contained} is proved. We infer from \eqref{eq:generalcriticaldensity}, 
\eqref{eq:section-contained}
and  Theorem~\ref{thm:covering} that
\begin{align}\label{eq:generalpowerdecay}
|\Omega\setminus G_{h M}(u,\Omega)|
\leq  \sqrt{2\epsilon_0}\left[
|\Omega\setminus G_{h}(u,\Omega)|
 +  \big|\{x\in \Omega:
\mathcal M (f^n)(x)> (c^* M h)^n\}\big|\right],
\end{align}
as long as $1/h \leq \hat c$.

For $k=0,1,\dots$, set
\[
a_k:= |\Omega\setminus G_{M^k}(u,\Omega)| \quad \text{and}\quad b_k:=\big|\{x\in \Omega:
\mathcal M(f^n)(x)> (c^* M M^k)^n\}\big|. 
\]
 Let $h=M$, then we get from \eqref{eq:generalpowerdecay} that
$a_2\leq \sqrt{2\epsilon_0}(a_1 + b_1)$.
Next let $h=M^2$, then
$a_3\leq \sqrt{2\epsilon_0}(a_2 + b_2)
\leq 2\epsilon_0 a_1 +  2\epsilon_0 b_1 + \sqrt{2\epsilon_0}\, b_2$.
Continuing in this way we conclude that
\begin{equation}\label{improved-est}
|\Omega\setminus G_{M^{k+1}}(u,\Omega)|
=a_{k+1}\leq (\sqrt{2\epsilon_0})^k a_1 + \sum_{i=1}^{k}{(\sqrt{2\epsilon_0})^{(k+1)-i} b_i}\quad \mbox{for}\quad k=1,2,\dots
\end{equation}

We are now ready to prove \eqref{reduction}. We have 
\begin{align*}\label{Lp-distribution-fn}
&\int_{\Omega}{|D_{i j} u|^p ~dx}
= p\int_0^{\infty}{t^{p-1} \big|\{x\in \Omega: \, |D_{i j} u(x)| >t \}\big| ~dt}\\
&= p \int_0^{M^{\frac{q}{p}}}{t^{p-1} \big|\{x\in \Omega: |D_{i j} u(x)| >t \}\big| ~dt}
+ p\sum_{k=1}^{\infty}\int_{M^{\frac{qk}{p}}}^{M^{\frac{q(k+1)}{p}}}{t^{p-1} \big|\{x\in \Omega: |D_{i j} u(x)| >t \}\big| ~dt}\nonumber\\
&\leq |\Omega| M^{q} +\big(M^{q} -1 \big)  \sum_{k=1}^{\infty}{M^{qk} \big|\{x\in \Omega: ~|D_{i j} u(x)| > M^{\frac{qk}{p}} \}\big|} \nonumber\\
&\leq |\Omega| M^{q}
+\big(M^{q} -1 \big) \left[ \sum_{k=1}^{\infty}{M^{qk} \big|\Omega\setminus \lA_{(c M^{\frac{k(q-p)}{2p}})^{\frac{-2}{n-1}}}\big|}  +  \sum_{k=1}^{\infty}{M^{qk} \big|\Omega\setminus G_{M^k}(u,\Omega)\big|}\right]\nonumber\\
&\leq |\Omega| M^{q}
+\big(M^{q} -1 \big) \left[ C(n,\epsilon,\Omega) \sum_{k=1}^{\infty}{ M^{k\left(q + (\frac{q}{p}-1)\frac{\ln{\sqrt{C \epsilon}}}{C}\right)}}   + \sum_{k=1}^{\infty}{M^{qk} \big|\Omega\setminus G_{M^k}(u,\Omega)\big|}\right],\nonumber
\end{align*}
where we used \eqref{eq:distributionsetofsecondderivatives}
with $m =q/p>1$ and $\beta =M^k$ in the second inequality and used \eqref{eq:exponentialdecay} in the last inequality.
Since $\epsilon>0$ is small, the first summation in the last expression  is finite
and hence \eqref{reduction}  will follow if we can show that $\sum_{k=1}^{\infty}{M^{kq} |\Omega\setminus G_{M^k}(u,\Omega)|}\leq C$. For this, let us employ 
 \eqref{improved-est} to obtain
\begin{align*}
\sum_{k=1}^{\infty}{M^{kq} |\Omega\setminus G_{M^k}(u,\Omega)|}
&\leq a_1 \sum_{k=1}^{\infty} {M^{kq} (\sqrt{2\epsilon_0})^{k-1} } + \sum_{k=1}^{\infty} \sum_{i=0}^{k-1}{M^{kq} (\sqrt{2\epsilon_0})^{k-i}  b_i}\\
&=\frac{ a_1}{\sqrt{2\epsilon_0}} \sum_{k=1}^{\infty} {\big(M^{q} \sqrt{2\epsilon_0}\big)^{k} } +
\Big[\sum_{j=1}^{\infty}{\big(M^{q} \sqrt{2\epsilon_0}\big)^j}\Big]
 \Big[\sum_{i=0}^{\infty} { M^{iq} b_i}\Big]\\
 &=\frac{ a_1}{\sqrt{2\epsilon_0}} \sum_{k=1}^{\infty} {2^{-k} } +
\Big[\sum_{j=1}^{\infty}{2^{-j}}\Big]
 \Big[\sum_{i=0}^{\infty} { M^{iq} b_i}\Big]
= \frac{ a_1}{\sqrt{2\epsilon_0}}  +
 \sum_{i=0}^{\infty} { M^{iq} b_i}.
\end{align*}
But as $f^n \in L^{\frac{q}{n}}(\Omega)$ and $q>n$, by the strong-type estimate in Theorem~\ref{strongtype} we have
\begin{align*}
\int_{\Omega} \abs{\mathcal{M}(f^n)(x)}^{\frac{q}{n}} dx \leq C(n, q, \rho) \int_{\Omega}\abs{f^{n}(x)}^{\frac{q}{n}} dx \leq C(n, q, \rho)\|f\|^{q}_{L^{q}(\Omega)}\leq C(n, q, \rho)
\end{align*}
implying
$ \sum_{i=0}^{\infty} { (M^n)^{i \frac{q}{n}} b_i}\leq C$.
Thus   $\sum_{k=1}^{\infty}{M^{kq} |\Omega\setminus G_{M^k}(u,\Omega)|}\leq C$  and \eqref{reduction} is proved.
\end{proof}
Finally, we prove Theorem~\ref{main-result}.
\begin{proof}[Proof of Theorem \ref{main-result}] 
It suffices to prove the theorem for the case $\varphi=0$ since $\tilde u:= u-\varphi \in C(\overline{\Omega})\cap W^{2,n}_{loc}(\Omega)$ is the solution to the linearized
Monge-Amp\`ere equation
\begin{equation*}
\calL_\phi \tilde u =\tilde f \,\mbox{ in }\, \Omega,
\quad\mbox{and}\quad
\tilde u  = 0\,\mbox{ on }\, \partial \Omega,
\end{equation*}
where $\tilde f:=f - \Phi^{ij}\varphi_{ij}\in L^{q}(\Omega).$ Indeed, since $g\in C(\overline{\Omega})$, we have $\phi\in W^{2, \frac{(n-1)q s}{s -q}}(\Omega) $ by Savin's global
$W^{2,p}$ estimates \cite{S3}. Thus $\Phi^{ij}\in L^{\frac{q s }{s -q}}(\Omega)$ for all $i, j$ and hence
$\tilde f:= f-\Phi^{ij}\varphi_{ij}\in L^{q}(\Omega)$.

In view of Theorem \ref{main-result-constant} and the interior $W^{2,p}$ estimates obtained in \cite{GN2}, the theorem follows by localizing boundary 
sections of $\phi$ using Theorem~\ref{main_loc}. 
 For completeness, we sketch the proof.

The assumptions on $\Omega$ and $\phi$ imply that $\Omega$ satisfies \eqref{global-tang-int} for some $\rho>0$
and $\phi$ satisfies \eqref{global-sep}. 
Let $\e$ be the small  constant given by an analogous version of Theorem \ref{main-result-constant} which will be explained later.
In particular, $\e$ depends only on  $n, p, q, \lambda, \Lambda, \rho$ and $\alpha$.
 Let $c$ be as in Remark \ref{choosingc}.

Since $g\in C(\overline{\Omega})$, we can find $m\leq c$ depending only on $\e$, $\lambda$ and the modulus of continuity of $g$ such that
$$\abs{g(x)-g(y)}\leq \lambda \e~\text{ for all }~ x, y\in \overline{\Omega}~\text{ satisfying }~\abs{x-y}\leq m.$$
Hence it follows from \eqref{small-sec} that for $s\leq m^3$ and any  boundary point $y\in\p\Omega$, we have
\begin{equation}\label{modulus-continuity}
\abs{g(x)-g(y)}\leq \lambda \e~\text{ for all } ~ x\in S_{\phi}(y, s).
\end{equation}

Let us consider a boundary point  $y\in\partial \Omega$ and for simplicity we assume that $y=0$. We can assume further that $\Omega$ satisfies 
\eqref{0grad},  $\phi(0)=0$ and $\nabla\phi(0) =0$.
Then by the Localization Theorem, there is a linear  map $T_s = s^{-1/2} A_s$ such that 
\begin{equation}\label{eq:Big-Section}
\overline{\Omega} \cap B_{k}\subset T_s(S_\phi(0, s))\subset \overline{\Omega}\cap B_{k^{-1}},
\end{equation}
where $\det A_s =1$ and $\|A_s\|, \, \|A_s^{-1}\|\leq k^{-1} |\log{s}|$.
By working with the function $g(0)^{\frac{-1}{n}}\phi(x)$ instead of 
$\phi(x)$ and using \eqref{modulus-continuity}, we can also assume that $g(0) =1$ and
\[
1-\e \leq g\leq 1 + \e
\quad\mbox{in}\quad S_\phi(0, s).
\]
We now define the rescaled domains $U_s :=T_s(S_\phi(0,s))$, $\Omega_{s} :=T_s(\Omega)$ and the rescaled functions $\phi_s$, $u_s := u\circ T^{-1}_s$, $f_s$ as in
 Subsection~\ref{rescale-sec} {\it that preserve the $L^{\infty}$-norm of $u$}. We claim that
\begin{align}\label{eq:localized-Lp-est}
\|D^2 u_s\|_{L^p\big( S_{\phi_s}(0, c^9)\big)}
&\leq C
\Big(\| u_s\|_{L^\infty(U_s)}
+\|f_s\|_{L^{q}(U_s)} \Big),
\end{align}
where $C>0$ depends only on $p, q, n,  \rho,\lambda, \Lambda, \alpha$,  the uniform convexity of $\partial \Omega$, $\|\partial \Omega\|_{C^{2,\alpha}}$ and $\|\phi\|_{C^{2,\alpha}(\p\Omega)}$.
 Then by rescaling back as in the proof of Lemma~\ref{w2p-small-local} we obtain
\begin{align}\label{half-lp}
\|D^2 u\|_{L^p\big( S_{\phi}(y, c^9 s)\big)} 
&\leq C s^{\frac{n}{2p}-1} |\log{s}|^2 \|u\|_{L^\infty(\Omega)} + C s^{\frac{n}{2}(\frac{1}{p}-\frac{1}{q})} |\log{s}|^2 \|f\|_{L^{q}(\Omega)}\\
&\leq C(s) \, \Big(   \|u\|_{L^\infty(\Omega)}
+ \|f\|_{L^{q}(\Omega)}\Big)\qquad \forall y\in \partial\Omega.\nonumber
\end{align}

Let $\delta := c^9 s$. By \eqref{small-sec}, we know that
$S_{\phi}(y,\delta)\supset \overline{\Omega}\cap B(y, \delta^{2/3}).$
Therefore if we let
$$\Omega_{\delta^{2/3}} :=\{x\in\Omega: \dist(x,\p\Omega)>\delta^{2/3}\},$$
then we can cover the $\delta^{2/3}$ neighborhood of $\Omega$, that is $\Omega\setminus \Omega_{\delta^{2/3}}$, 
by a finite number of boundary sections $\{S_{\phi}(y_j, \delta)\}_{j=1}^{N}$. Then by adding \eqref{half-lp} over the family $\{S_{\phi}(y_j, \delta)\}_{j=1}^{N}$, we arrive at
the $W^{2,p}$ estimate at the boundary
$$\|D^2 u\|_{L^{p}(\Omega\setminus \Omega_{\delta^{2/3}})}\leq C(\|u\|_{L^{\infty}(\Omega)} + \|f\|_{L^{q}(\Omega)}).$$
On the other hand, by the interior estimate in \cite[Theorem 1.1]{GN2}, we also have
$$\|D^2 u\|_{L^{p}( \Omega_{\delta^{2/3}})}\leq C(\|u\|_{L^{\infty}(\Omega)} + \|f\|_{L^{q}(\Omega)}).$$
Our Theorem~\ref{main-result} follows from the above inequalities.

We now indicate how to obtain the  claim \eqref{eq:localized-Lp-est}. 
The proof consists of reviewing the proof of Theorem~\ref{main-result-constant}. By \eqref{small-sec}, we have
$$S_{\phi_s}(0, c^9)\subset U_{s}\cap B_{c^3}.$$
We use Lemma~\ref{covering_rk} to cover $U_{s}\cap B_{c^2}$.  We restrict our estimates on the distribution function for the second derivatives in Lemma~\ref{distribution-function} to
$U_s\cap B_{c^2}$. Lemma~\ref{lm:improved-density-I} holds with obvious changes for the data $(\Omega_s, \phi_s, U_s)$. So does 
Lemma~\ref{lm:furthercriticaldensitytwo} provided that
 we have an analogous version of  Lemma~\ref{lm:Holder-on-bottom-boundary}
 for our data $(\Omega_s, \phi_s, U_s)$. Precisely,
 let $S_{\phi_s}(y_0, h)$ be a section of $\phi_s$ in $U_s$ such that $y_0\in \p U_s\cap B_{c^3}$ and $S_{\phi_s}(y_0, h)\cap G_\gamma(u_s, U_s, \phi_s) \neq \emptyset$ for some $\gamma>0$
 (say, $\bar{y}\in S_{\phi_s}(y_0, h)\cap G_\gamma(u_s, U_s, \phi_s)\,$). By  Lemma~\ref{sep-lem} and the Localization Theorem~\ref{main_loc}, there exists an affine map $\tilde T_h$ such that 
 \[
  \tilde T_h (y_0) = y_0 \quad \mbox{and}\quad \overline{U_s} \cap B_{k}(y_0)\subset \tilde{U}_h := \tilde T_h(S_{\phi_s}(y_0, \theta h) )\subset \overline{U_s}\cap B_{k^{-1}}(y_0).
 \]
 Here $\theta>1$ is the same constant at the beginning of the proof of Lemma \ref{lm:furthercriticaldensitytwo}.
We need to show that the $C^{2,\alpha}$ norm on the boundary
$\p \tilde{U}_h \cap B_k(y_0)$ of the following function 
\begin{equation*}
\label{rescaled-v}
\tilde v(z) := \dfrac{1 }{\theta \gamma h }\left[ u_s(\tilde T_h^{-1} z)-  u_s(\bar y) - \nabla  u_s(\bar
y)\cdot (\tilde T_h^{-1} z-\bar y)\right],\quad z\in \tilde T_h(U_s)
\end{equation*}
 is  bounded by a constant which is independent of the uniform convexity of $U_s$. The function $\tilde v$  is defined in a similar way to  the definition of the function $v$ in \eqref{affine-function-l}.
We note that  the uniform convexity of the 
boundary $\p\Omega$   plays a key role
in the proof of Lemma~\ref{lm:Holder-on-bottom-boundary}. 
Thus we can not obtain the desired result by repeating the proof of Lemma~\ref{lm:Holder-on-bottom-boundary}  for our data $(\Omega_s, \phi_s, U_s)$ since the uniform convexity of $\p\Omega_s$ deteriorates as $s\rightarrow 0$. However, we can get away from this  as follows.

Let $T := \tilde T_h \circ  T_s$. Then $T$ normalizes the section $S_\phi(T_s^{-1}y_0, \theta hs)$, and 
\[
\|T\| \leq k^{-2} (\theta h s)^{-1/2} \, |\log{(\theta h)}| \, |\log{s}|,
\quad 
\|T^{-1}\| \leq  k^{-2} (\theta h s)^{1/2} \, |\log{(\theta h)}| \, |\log{s}|.
\]
Moreover, $\bar{x} := T_s^{-1}(\bar{y})\in S_\phi(T_s^{-1}y_0, \theta hs) \cap G_{\gamma s^{-1} }(u, \Omega,\phi)$ and 
\[T\big(S_\phi(T_s^{-1}y_0, \theta hs)\big)= \tilde T_h (S_{\phi_s}(y_0, \theta h))=\tilde U_h.
\]
Therefore, by reviewing the proof of  Lemma~\ref{lm:Holder-on-bottom-boundary} we see that the function 
\[
v(y) := \frac{1}{\theta (\gamma s^{-1}) hs}\left[ u(T^{-1} y)-  u(\bar x) - \nabla  u(\bar
x)\cdot (T^{-1} y-\bar x)\right],\quad y\in T(\Omega)
\]
satisfies 
\begin{equation}\label{bound-for-v}
\|v\|_{C^{2,\alpha}\big(\partial \tilde{U}_h \cap B_k(\tilde T_h(y_0))\big)}\leq C_\alpha
\end{equation}
with $C_\alpha$ depending on the uniform convexity of $\Omega$. But since $\tilde T_h(y_0) = y_0$ and $\tilde v \equiv v$ on $\tilde U_h$ as $u_s(y) = u(T_s^{-1} y)$, 
we conclude that the  $C^{2,\alpha}$ norm of $\tilde v$ on $\partial \tilde{U}_h \cap B_k(y_0)$ is bounded by the same constant $C_\alpha$ in \eqref{bound-for-v}. Hence the claim \eqref{eq:localized-Lp-est} follows as explained above.
\end{proof}
{\bf Acknowledgements.} T.~Nguyen gratefully acknowledges the support provided by  NSF grant DMS-0901449.

\bibliographystyle{plain}

\end{document}